\documentclass{article}
\usepackage[
    backend=biber,
    datamodel=mydatamodel,
    style=alphabetic,
    natbib=true,
    url=false,
    doi=true,
    eprint=true,
    giveninits=true,
    isbn=false,
    maxbibnames=10,
]{biblatex}
\DeclareFieldFormat{mrnumber}{%
    MR\space
    \ifhyperref
    {\href{http://www.ams.org/mathscinet-getitem?mr=#1}{\nolinkurl{#1}}}
    {\nolinkurl{#1}}}
\DeclareFieldFormat{zbl}{%
    Zbl\space
    \ifhyperref
    {\href{https://zbmath.org/?q=an:#1}{\nolinkurl{#1}}}
    {\nolinkurl{#1}}}
\renewbibmacro*{doi+eprint+url}{%
    \iftoggle{bbx:doi}
    {\printfield{doi}}
    {}%
    \newunit\newblock
    \printfield{mrnumber}%
    \newunit\newblock
    \printfield{zbl}%
    \newunit\newblock
    \iftoggle{bbx:eprint}
    {\usebibmacro{eprint}}
    {}%
    \newunit\newblock
    \iftoggle{bbx:url}
    {\usebibmacro{url+urldate}}
    {}}
\usepackage[a4paper, margin=3cm]{geometry}
\usepackage{hyperref}
\usepackage{authblk}
\usepackage{amsthm,amsmath,amsfonts,amssymb}
\usepackage{mathtools}
\usepackage{stmaryrd}
\usepackage[capitalize,nameinlink,noabbrev]{cleveref} 
\usepackage{enumitem} 
\usepackage{graphicx} 
\usepackage{dsfont} 
\usepackage{xcolor}
\usepackage[capitalize,nameinlink,noabbrev]{cleveref}
\newcommand*{\mc}[1]{\mathcal{#1}}
\newcommand*{\mb}[1]{\mathbb{#1}}

\newcommand*{\ms}[1]{\mathsf{#1}}
\newcommand*{\mo}[1]{\mathbf{#1}}
\newcommand*{\mf}[1]{\mathfrak{#1}}
\newcommand{\N}{\mathbb{N}}
\newcommand{\Nn}{\mathbb{N}_0}
\newcommand{\Z}{\mathbb{Z}}

\newcommand{\R}{\mathbb{R}}
\newcommand{\Rp}{\R_{\geq0}}
\newcommand{\Rpp}{\R_{>0}}

\newcommand{\Rppo}{\R_{>1}}
\newcommand{\ball}{\mathrm{B}}
\newcommand{\nset}[2]{{\left\llbracket #1,#2\right\rrbracket}}
\newcommand{\nnset}[1]{{\left\llbracket #1\right\rrbracket}}
\newcommand{\nnsetof}[1]{{\llbracket #1\rrbracket}}
\newcommand{\nnzset}[1]{{\left\llbracket 0, #1\right\rrbracket}}
\newcommand{\lcb}{\left\lbrace} 
\newcommand{\rcb}{\right\rbrace} 
\newcommand{\cb}[1]{\lcb #1 \rcb} 
\newcommand{\cbOf}[1]{\mathopen{}\lcb #1 \rcb\mathclose{}} 
\newcommand{\lab}{\left[} 
\newcommand{\rab}{\right]} 
\newcommand{\ab}[1]{\lab #1 \rab} 
\newcommand{\abOf}[1]{\!\ab{#1}} 
\newcommand{\lb}{\left(} 
\newcommand{\rb}{\right)} 
\newcommand{\br}[1]{\lb #1 \rb} 
\newcommand{\brOf}[1]{\!\br{#1}} 
\newcommand{\abs}[1]{\left| #1 \right|} 
\newcommand{\normof}[1]{\Vert#1\Vert}
\newcommand{\normOf}[1]{\left\Vert#1\right\Vert}

\renewcommand{\Pr}{\mathbf{P}} 

\newcommand{\PrOf}[1]{\Pr\brOf{#1}}
\newcommand{\E}{\mathbf{E}} 
\newcommand{\Eof}[1]{\E[#1]}

\newcommand{\EOf}[1]{\E\abOf{#1}}
\newcommand{\EOff}[2]{\E_{#1}\abOf{#2}}

\newcommand{\supNormof}[1]{\vert#1\vert_\infty}
\newcommand{\supNormOf}[1]{\abs{#1}_\infty}
\newcommand{\opNormof}[1]{\Vert#1\Vert_{\mathsf{op}}}
\newcommand{\opNormOf}[1]{\left\Vert#1\right\Vert_{\mathsf{op}}}
\newcommand{\euclof}[1]{\Vert#1\Vert_2}
\newcommand{\euclOf}[1]{\left\Vert#1\right\Vert_2}

\newcommand{\sizedMid}[2]{#1 \, \kern-\nulldelimiterspace\mathopen{}\left| \vphantom{#1}\,#2\right.\mathclose{}\kern-\nulldelimiterspace}
\newcommand{\setByEle}[2]{\cb{\sizedMid{#1}{#2}}}
\newcommand{\setByEleInText}[2]{\{#1 \mid #2\}}
\newcommand{\tr}{^{\!\top}\!} 
\newcommand{\pr}{^\prime}

\newcommand{\asymleq}{\preccurlyeq}
\newcommand{\asymlt}{\prec}
\newcommand{\asymgeq}{\succcurlyeq}
\newcommand{\asymgt}{\succ}
\newcommand{\asymeq}{\asymp}
\newcommand{\ind}{\mathds{1}}
\newcommand{\indOfOf}[2]{\ind_{\!#1}\!\brOf{#2}}%
\newcommand{\eqcm}{\,,} 
\newcommand{\eqfs}{\,.}
\newcommand{\dl}{\mathrm{d}} 
\newcommand{\supp}{\operatorname{supp}} 
\renewcommand{\subset}{\subseteq}
\DeclareMathOperator*{\argmin}{arg\,min}
\DeclareMathOperator*{\argmax}{arg\,max}

\newcommand{\const}[1]{C_{\mathsf{#1}}}
\newcommand{\assuRef}[1]{\texorpdfstring{\protect\hyperlink{assu#1}{\textsc{#1}}}{}}
\newcommand{\newAssuRef}[1]{\hypertarget{assu#1}{\textsc{#1}}}
\newcommand{\indset}[3]{#1_{\nset{#2}{#3}}}
\newcommand{\noise}{\varepsilon}
\newcommand{\ftrue}{f^\star}
\newcommand{\festi}{\hat f}
\newcommand{\utrue}{u^\star}
\newcommand{\dutrue}{\dot u^\star}
\newcommand{\uesti}{\hat u}
\newcommand{\duesti}{\hat{\dot u}}

\newcommand{\stepsize}{\Delta\!t}

\newcommand{\FSmooth}{\mathcal{F}_{d,\beta}}
\newcommand{\dm}[1]{d_{\mathsf{#1}}} 

\newcommand{\Lbeta}{L_{\beta}} 
\newcommand{\uztrue}{x_1}
\newcommand{\ttrue}{\tau^\star}
\newcommand{\normop}[1]{\normof{#1}_{\ms{op}}}
\newcommand{\nn}{t_{\ms {NN}}}

\newcommand{\Op}{\mo O_{\Pr}}

\newcommand{\Incr}{\Upsilon}
\newcommand{\incr}{\iota}
\newcommand{\itrue}{\incr^\star}
\newcommand{\iesti}{\hat\incr}
\newcommand{\ptrue}{p^\star}
\newcommand{\pesti}{\hat p}
\newcommand{\dptrue}{\dot p^\star}
\newcommand{\dpesti}{\dot{\hat p}}
\newcommand{\Esti}{\mc E}
\newcommand{\dEsti}{\widetilde{\mc E}}

\newcommand{\hull}{\ms{ch}}
\newcommand{\evmin}{\lambda_{\ms{min}}}
\newcommand{\Poly}[2]{\mc P_{#1, #2}}
\newcommand{\rate}{\Gamma}
\newcommand{\drate}{\Lambda}
\DeclareMathOperator{\diam}{\mathsf{diam}}

\newbox{\myorcidthanksbox}
\sbox{\myorcidthanksbox}{\large\includegraphics[height=1.8ex]{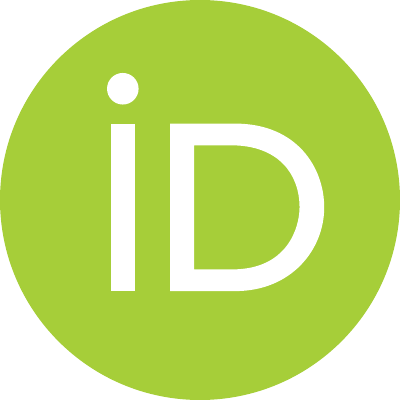}}
\newcommand{\orcidthanks}[1]{%
    \href{https://orcid.org/#1}{\raisebox{-0.5ex}{\usebox{\myorcidthanksbox}}\,#1}}
\usepackage{thmtools}

\def\postBoxSkip{1.0ex}
\def\postBoxSkipCmd{\vskip\postBoxSkip}
\def\preBoxSkip{1.0ex}
\def\preBoxSkipCmd{\vskip\preBoxSkip}
\declaretheoremstyle[
	bodyfont=\normalfont,
	postfoothook={\postBoxSkipCmd},
	preheadhook={\preBoxSkipCmd},
	mdframed={
		backgroundcolor = black!2,
		startcode={},
}]{ruledBoxStyle}
\declaretheoremstyle[
	bodyfont=\normalfont,
	postfoothook={\postBoxSkipCmd},
	preheadhook={\preBoxSkipCmd},
	mdframed={
		backgroundcolor=white,
}]{ruledBoxStyleWhite}
\declaretheoremstyle[
	bodyfont=\normalfont,
	postfoothook={\postBoxSkipCmd},
	preheadhook={\preBoxSkipCmd},
	mdframed={
		backgroundcolor=black!2,
		linecolor = black!2,
		tikzsetting = {
			draw = black,
			line width = 2pt,%
			dashed,%
			dash pattern = on 10pt off 3pt
		},
}]{dashedBoxStyle}
\declaretheoremstyle[
	bodyfont=\normalfont,
	postfoothook={\postBoxSkipCmd},
	preheadhook={\preBoxSkipCmd},
	mdframed={
		linecolor = white,
		startcode={},
		tikzsetting = {
			draw = black,
			line width = 1pt,%
			loosely dotted,
		},
	}
]{dashedStyle}
\declaretheoremstyle[
	bodyfont=\normalfont,
	postfoothook={\postBoxSkipCmd},
	preheadhook={\preBoxSkipCmd},
	mdframed={
		linecolor = black,
		innerlinewidth=1pt,outerlinewidth=1pt,
		middlelinewidth=1pt,
		linecolor=black,middlelinecolor=white,
		startcode={},
	}
]{doubleStyle}
\declaretheoremstyle[
	bodyfont=\normalfont,
	postfoothook={\postBoxSkipCmd},
	preheadhook={\preBoxSkipCmd},
	mdframed={
		backgroundcolor = black!4,
		linecolor = black!4,
		startcode={},
}]{boxStyle}
\declaretheoremstyle[
	headfont=\normalfont\itshape,
	notefont=\normalfont\itshape,
	notebraces={}{},
	bodyfont=\normalfont,
	qed=\qedsymbol,
	numbered=no,
	headindent=0pt,
	postheadspace=1ex,
	name={Proof},
	postheadhook={},
	mdframed={
		hidealllines = true,
		innerrightmargin = 0pt,
		innerleftmargin = 0pt,
		innertopmargin = 0pt,
		innerbottommargin = 0pt,
		leftmargin = 0pt,
		rightmargin = 0pt,
	}
]{proofStyle}
\declaretheoremstyle[
	bodyfont=\normalfont,
	postfoothook={\postBoxSkipCmd},
	preheadhook={\preBoxSkipCmd},
	mdframed={
		backgroundcolor = white,
		linecolor = black,
		startcode={},
		leftline = false,
		rightline = false,
}]{tobBottomStyle}
\declaretheoremstyle[
bodyfont=\normalfont,
]{standardStyle}
\declaretheorem[style=ruledBoxStyle,name=Definition, numberwithin=section]{definition}
\declaretheorem[style=ruledBoxStyle,name=Lemma,numberwithin=section,numberlike=definition]{lemma}
\declaretheorem[style=ruledBoxStyle,name=Proposition,numberwithin=section,numberlike=definition]{proposition}
\declaretheorem[style=ruledBoxStyle,name=Theorem,numberwithin=section,numberlike=definition]{theorem}
\declaretheorem[style=ruledBoxStyle,name=Corollary,numberwithin=section,numberlike=definition]{corollary}
\declaretheorem[style=boxStyle,name=Remark,numberwithin=section,numberlike=definition]{remark}
\declaretheorem[style=boxStyle,name=Notation,numberwithin=section,numberlike=definition]{notation}

\declaretheorem[style=dashedStyle,name=Assumption,numberwithin=section,numberlike=definition]{assumption}

\title{Nonparametric Estimation of Ordinary Differential Equations: Snake and Stubble}
\date{}
\author[1,2]{Christof Schötz\thanks{math@christof-schoetz.de, \orcidthanks{0000-0003-3528-4544}}}
\affil[1]{Potsdam Institute for Climate Impact Research}
\affil[2]{Technical University of Munich}
\begin{document}
\maketitle
\begin{abstract}
	We study nonparametric estimation in dynamical systems described by ordinary differential equations (ODEs). Specifically, we focus on estimating the unknown function $f \colon \mathbb{R}^d \to \mathbb{R}^d$ that governs the system dynamics through the ODE $\dot{u}(t) = f(u(t))$, where observations $Y_{j,i} = u_j(t_{j,i}) + \varepsilon_{j,i}$ of solutions $u_j$ of the ODE are made at times $t_{j,i}$ with independent noise $\varepsilon_{j,i}$.

We introduce two novel models -- the Stubble model and the Snake model -- to mitigate the issue of observation location dependence on $f$, an inherent difficulty in nonparametric estimation of ODE systems. In the Stubble model, we observe many short solutions with initial conditions that adequately cover the domain of interest. Here, we study an estimator based on multivariate local polynomial regression and univariate polynomial interpolation. In the Snake model we observe few long trajectories that traverse the domain on interest. Here, we study an estimator that combines univariate local polynomial estimation with multivariate polynomial interpolation.

For both models, we establish error bounds of order $n^{-\frac{\beta}{2(\beta +1)+d}}$ for $\beta$-smooth functions $f$ in an infinite dimensional function class of Hölder-type and establish minimax optimality for the Stubble model in general and for the Snake model under some conditions via comparison to lower bounds from parallel work.
\end{abstract}
\textbf{Keywords:} Ordinary Differential Equations, Nonparametric Regression, Minimax Optimal, Rate of Convergence
\tableofcontents
\section{Introduction}\label{sec:intro}
In statistics, we want to extract information from observations of the world. The workings of the world are described in physics (and many other sciences) with differential equations \cite{Strogatz2024}. Hence, statistical models of dynamical systems described by differential equations seem to be not only compelling but almost inevitable for study.

One of the simplest non-trivial instance of such a model can be described in the equations
\begin{equation*}
    Y_i = u(t_i) + \noise_i\eqcm i =1, \dots, n
    \qquad
    \dot u(t)  = f(u(t))
    \eqcm
\end{equation*}
where $\dot u$ denotes the time-derivative, i.e., $(\dl u)/(\dl t)$.
Here, $u\colon\R\to\R^d$ is the solution to an ordinary differential equation (ODE) described by an unknown function $f\colon\R^d\to\R^d$. It is measured at times $t_1 \leq \dots \leq t_n$, which creates our observations $Y_i$ by adding measurement noise $\noise_i$ to the true state of the system $u(t_i)$. There are different objectives one could have for such a model: reconstructing the solution $u$ in the observed time interval $[t_1, t_n]$ (\textit{reanalysis}, borrowing a term from climate modeling), predicting a future state $u(t)$, $t > t_n$ (\textit{forecasting}, as in weather forecast), or estimation of $f$ (\textit{learning the dynamics} of the dynamical system). Here, we are interested in the estimation of $f$, which is fundamental also for the other tasks as an estimate $\hat f$ allows us to create an estimate $\hat u$ by solving $\dot{\hat u}(t) = \hat f(\hat u(t))$.

We want to estimate $f$ in a nonparametric model, i.e., we assume that $f$ belongs to an infinite dimensional class of smooth functions, but we do not make any assumptions on its precise functional form. The parametric version of this model, in which $f$ belongs to a finite dimensional class of functions (e.g., polynomials of fixed degree) has been studied in \cite{Brunel_2008, qi10, gugushvili12, dattner15}. There, it is shown that the typical parametric $\sqrt{n}$-rate of convergence can be achieved. An overview of the state of the art for statistics of dynamical systems (mostly concerned with parametric settings) can be found in \cite{McGoff15, Ramsay_2017, Dattner2020}.

In the nonparametric setting, different algorithms for estimation have been proposed, e.g., \cite{heinonen18, chen18, gottwald21}.
In \cite{marzouk2023}, theoretical results for learning dynamics nonparametrically are shown, but in a density estimation context that is rather different from the regression-type model studied here. To the best of the author's knowledge, only \cite{Lahouel2023} takes a theoretical view on the ODE regression problem (additionally to proposing a new algorithm based on reproducing kernel Hilbert spaces). They show an error bound of order $n^{-1/4}$ (root mean squared error) for the reanalysis problem \cite[Theorem 1]{Lahouel2023}. This result cannot be optimal as standard nonparametric regression for $u$ (ignoring the ODE constraints) yields an error of smaller order, namely $n^{-\beta/(2\beta+1)}$, where $\beta$ describes the smoothness of $u$ (in the case of \cite[Theorem 1]{Lahouel2023}, we have $\beta=3$).

Estimation in the nonparametric setting seem much harder than in the parametric one: If $f$ is a smooth function with bounds on its derivative but otherwise unrestricted, an observation $Y_i$ only yields information on $f(x)$ for $x$ close to $u(t_i)$, i.e., the information is local. In contrast, each observation from an ODE where $f$ is known to be a polynomial of degree at most $N$ carries information for all coefficients of said polynomial, i.e., the information is global. What makes estimation in a nonparametric ODE model more difficult than in a nonparametric regression model (which also suffers from locality of information) is the fact that the location of the observation $u(t_i)$ is itself depends on the unknown $f$.

We introduce two models, called \textit{Snake model} and \textit{Stubble model}, which provide a remedy of this problem. For each of the two models, we propose estimators and study their rate of convergence. We find that in both instances, one obtains an error bound (root mean squared error) of order
\begin{equation}\label{eq:bestrate}
    n^{-\frac{\beta}{2(\beta +1)+d}}
\end{equation}
for $\beta$-smooth functions $f$ under optimal circumstances. This error bound is minimax optimal as we have corresponding lower bounds in the parallel work \cite{lowerbounds}.
\subsection{Contributions}
We now give more details on the contributions of this work.

\textbf{A general model.}
In general, we assume the true model function $\ftrue$ to be an element of Hölder-type smoothness class denoted as $\bar\Sigma^{d\to d}(\beta, \indset L0\beta)$, which means that the $k$-th derivative of $f$ for $k\in\nnzset\beta := \{0,\dots, \beta\}$ is bounded by $L_k\in\Rpp$, i.e., $\sup_{x\in\R^d} \normop{D^k f(x)} \leq L_k$, where $\normop{\cdot}$ is the operator norm. Then, for given initial conditions $x_1, \dots, x_m$, we observe the solution $t\mapsto U(\ftrue, x_j, t)$ of the ODE $\dot u(t) = \ftrue(u(t))$ with $U(\ftrue, x_j, 0) = x_j$ at time points $i\stepsize$ with time step $\stepsize\in\Rpp$, i.e.,
\begin{equation*}
    Y_{j,i} = U(\ftrue, x_j, i\stepsize) + \noise_{j,i}\eqcm\ j \in \nset{1}{m}, i \in \nset{1}{n_j}
    \eqfs
\end{equation*}
In total, we have $n = \sum_{j=1}^{m} n_j$ observations.
The noise variables $\noise_{j,i}$ are assumed to be independent, centered, and to have a finite second moment. Because of the dependency of the location of the observation $u(t_{j,i})$ on $\ftrue$, in general, consistent estimation in all of $[0,1]^d$ is impossible. Thus, we introduce the Stubble model and the Snake model, which restrict the general model.

\textbf{The Stubble model.}
We use the notation $\asymeq$ and $\asymleq$ to mean \textit{asymptotically equal} and \textit{asymptotically  lower than} (up to a positive constant), respectively (see \cref{not:asym}).
In the Stubble model, we observe many ($m \asymeq n$) short ($n_j \asymleq 1$) solutions. We assume that their initial conditions $x_j$ cover the domain of interest $[0,1]^d$ suitably. For this model, we construct an estimator $\festi$ based on a multivariate local polynomial estimator and a univariate polynomial interpolation, which is similar to the Adams--Bashforth method for numerical solutions of ODEs \cite[Chapter 24]{Butcher2016}. We obtain
 \begin{equation}\label{eq:intro:stubble:general:rate}
    \EOf{\euclOf{\festi(x) - \ftrue(x)}^2}
    \asymleq \br{\stepsize^2 n}^{-\frac{2\beta}{2\beta+d}} + \stepsize^{2\beta}
\end{equation}
for all $x\in[0,1]^d$, see \cref{cor:stubble:general}. The rate is shown to be minimax optimal by comparing it to lower bounds in \cite{lowerbounds}.

\textbf{The Snake model.}
In the Snake model, we observe few ($m \asymleq 1$) long ($n_j \asymeq n$) solutions. We require the trajectories $U(\ftrue, x_j, [0, n_j \stepsize])$ to cover the domain of interest $[0,1]^d$ suitably.
For this model, we construct an estimator $\festi$ based on a univariate local polynomial estimator and a multivariate polynomial interpolation.
Then
\begin{equation}\label{eq:intro:snake:general:rate}
    \sup_{x\in [0,1]^d}\euclOf{\festi(x) - \ftrue(x)}^2 \in
    \Op\brOf{\delta^{2\beta} + \br{\stepsize \log n}^{\frac{2\beta}{2(\beta+1)+1}}}
    \eqcm
\end{equation}
see \cref{cor:snake:general}, where we want to view $\delta\in\Rpp$ for now as the largest distance between a point $x\in[0,1]^d$ and its closest state $U(\ftrue, x_j, t)$, $t\in[0, n_j\stepsize]$. This interpretation is true for the case $\beta = 1$, but if $\beta > 1$, the definition is more complex and the result more restrictive. By comparing it to lower bounds in \cite{lowerbounds}, the rate is shown to be minimax optimal if $\delta$ and $\stepsize$ are in a certain relation. For the rate to be optimal, we essentially require observations to be rather dense in time, and temporally distant parts of the trajectories $U(\ftrue, x_j, [0, n_j\stepsize])$ to be distant enough in state space.

\textbf{Connection and further results.}
Note the complementary nature of the two models and their estimators. In spite of this, in the optimal setting regarding $\stepsize$ and $\delta$, we obtain an error bound of the order given in \eqref{eq:bestrate} in both models (up to a $\log$-factor when considering the sup-norm), see \cref{cor:stubble:fast} and \cref{cor:snake:general:fast}.
For both models, we begin our discussion with the case $\beta = 1$, which allows for simple estimators and a gentle introduction to the main ideas as well as slightly stronger and more specific results.
See \cref{thm:stubble:lip} and \cref{cor:stubble:lip} for the Stubble model, and \cref{thm:snake:lip} and \cref{cor:snake:lip} for the Snake model. For the case of general smoothness $\beta\in\N$, we first show \emph{black box} results for estimation strategies that can be used with an arbitrary regression estimator, \cref{thm:stubble:general} and \cref{thm:snake:general}, respectively. If the chosen regression estimator achieves the minimax rate for a standard nonparametric regression problem, the estimation strategies achieve the rates \eqref{eq:intro:stubble:general:rate} and \eqref{eq:intro:snake:general:rate}, respectively. In \cref{cor:stubble:general} and \cref{cor:snake:general}, we apply the general results with the (minimax optimal) local polynomial estimator as regression estimator.
Furthermore, we study how the smoothness of $f$ influences the smoothness of $t\mapsto U(f, x, t)$ and the smoothness of $x\mapsto U(f, x, t)$ in \cref{coro:trajSmooth} and \cref{lmm:highDerivIncrem}, respectively, which may be of independent interest.
\subsection{Overview}
The remaining article is structured as follows.
In section \ref{sec:prelim}, we introduce basic concepts for the study of ordinary differential equations (section \ref{ssec:prelim:ode}), give a full formal description of the general model (section \ref{ssec:prelim:model}), argue why this model is not useful without further restriction (section \ref{ssec:prelim:restrict}), and introduce some notation for standard regression problems that allows us to formulate estimation strategies referring to arbitrary regression estimators (section \ref{ssec:prelim:regression}). Sections \ref{sec:stubble} and \ref{sec:snake} formally introduce the Stubble and Snake model, respectively, and describe the estimators $\festi$ of $\ftrue$, upper bounds on the error $\euclof{\festi(x) - \ftrue(x)}$ and their proofs. Both sections are separated into two parts: The first one (section \ref{ssec:stubble:lipschitz} and \ref{ssec:snake:lipschitz} respectively) concerns the Lipschitz-case ($\beta=1$). The second one (section \ref{ssec:stubble:general} and \ref{ssec:snake:general} respectively) concerns the general case ($\beta\in\N$).
Extensions to the main results are discussed in section \ref{sec:extension}.
The appendix makes this article largely self-contained: 
We first discuss how our ODE model relates to other more established statistical models in appendix \ref{app:sec:relation}.
Then we state some basics of multivariate derivatives and smoothness classes (appendix \ref{app:sec:derivative}). We recall classical results on the upper error bound for multivariate local polynomial regression in appendix \ref{app:sec:localreg}. In appendix \ref{app:sec:interpol}, uni- and multivariate polynomial interpolation are discussed. And finally, appendix \ref{app:sec:ode} provides smoothness results for the solutions of ODEs.

\section{Preliminaries}\label{sec:prelim}
In this section, we recall some basic concepts related to ODEs, we introduce a general statistical model for observing solutions of ODEs, we explain why this model requires further restriction in order to be a useful model, and we introduce some terminology for a standard regression problem that will later allow us to construct ODE estimators in a generic (black box) fashion.
\subsection{Ordinary Differential Equations}\label{ssec:prelim:ode}
We introduce some basic terminology and fundamental properties concerning ODEs. See also \cite{Hartman2002, Arnold2006}.
\begin{notation}\mbox{ }
    \begin{enumerate}[label=(\roman*)]
        \item
        Let $\Z$ be the set of integers and $\R$ the set of reals. For $\mb K \in \{\Z, \R\}$ and $a\in\R$, denote $\mb K_{>a} = \setByEleInText{x\in\mb K}{x>a}$.
        Define $\mb K_{\geq a}$, $\mb K_{<a}$, $\mb K_{\leq a}$ accordingly.
        \item
        Set $\N := \Z_{\geq 1}, \N_0 := \Z_{\geq 0}$. Let $a,b\in\Z$ with $a \leq b$. Set $\nset ab := \mb Z_{\geq a} \cap \mb Z_{\leq b}$. Set $\nnset a := \nset 1a$.
    \end{enumerate}
\end{notation}
Let $d\in\N$. Let $f\colon\R^d\to\R^d$. For a differentiable function $u\colon\R\to\R^d$, denote its derivative as $\dot u\colon\R\to\R^d$.
Then
\begin{equation}\label{eq:ode}
    \dot u(t) = f(u(t))\qquad\text{for } t\in\R
\end{equation}
is an \textit{autonomous}, \textit{first-order}, \textit{ordinary differential equation}.
It is of first-order, as \eqref{eq:ode} only depends on the first derivative of $u$. It is autonomous, as the right-hand side term $f(u(t))$ only depends on $t$ via $u$. In contrast, a non-autonomous, first-order ODE has the form $\dot u(t) = g(t, u(t))$ for a function $g\colon\R\times\R^d\to\R^d$.
Any differentiable $u\colon\R\to\R^d$ that fulfills \eqref{eq:ode} is a \textit{solution} to the ODE. The domain of $u$ is called \textit{time}. The codomain of $u$ is called \textit{state space}. A single element of the state space is a \textit{state}. The image of $u$ is called \textit{trajectory}.
We call $f$ the \textit{model function} of the ODE. Let $x \in\R^d$. Consider the requirement
\begin{equation}\label{eq:ic}
    u(0) = x
    \eqfs
\end{equation}
We call $x$ the \textit{initial conditions}. We call \eqref{eq:ode} together with \eqref{eq:ic} \textit{initial value problem} (IVP). If $u$ fulfills \eqref{eq:ode} and \eqref{eq:ic}, it is a solution to the IVP.

Assume that $f$ is (globally) Lipschitz continuous. Then the IVP \eqref{eq:ode}, \eqref{eq:ic} has a unique solution (Picard--Lindelöf theorem, also known as the Cauchy--Lipschitz theorem). Denote this solution as $U(f, x, \cdot) \colon \R\to\R^d, t\mapsto U(f, x, t)$.
The function $U(f, \cdot, \cdot) \colon \R^d\times \R\to\R^d$ is called \textit{flow} of the ODE \eqref{eq:ode}.
Note the \textit{semigroup property} of the flow,
\begin{align*}
    U(f, x, 0) &= x\eqcm\\
    U(f, x, s+t) &= U(f, U(f, x, s), t)
\end{align*}
for all $x\in\R^d$, $s,t\in\R$.
For a given time step $\stepsize\in\Rpp$, the mapping $x \mapsto U(f, x, \stepsize)$ is called \textit{propagator}. If $u$ is a solution to the ODE \eqref{eq:ode}, then $u(t+\stepsize) =
U(f, u(t), \stepsize)$ for all $t\in\R$. Furthermore, we call
\begin{equation*}
    \Incr(f, \stepsize, x) := U(f, x, \stepsize) - x = U(f, x, \stepsize) - U(f, x, 0)
\end{equation*}
the \textit{increment}.
If $u$ is a solution to the ODE \eqref{eq:ode}, then $u(t+\stepsize) = u(t) + \Incr(f, \stepsize, u(t))$ for all $t\in\R$.

The setting of autonomous, first-order ODEs is not a strong restriction: Consider the $d$-dimensional, non-autonomous ODE of order $\ell$,
\begin{equation}\label{eq:odehigh}
    v^{(\ell)}(t) = g(t, v(t), v^{(1)}(t), \dots, v^{(\ell-1)}(t))
    \eqcm
\end{equation}
where $v^{(k)}$ for $k\in\nnzset{\ell}$ denotes the $k$-th derivative of the $\ell$-times differentiable function $v\colon\R\to\R^d$ and $g \colon\R\times (\R^d)^{\ell}\to\R^d$ is a function. We represent the derivatives and the time variable with new variables: $u_{k} = v^{(k-1)}$ for $k\in\nnset{\ell}$ and $u_{\ell+1}(t) = t$. For $u(t) = (u_1(t), \dots, u_{\ell+1}(t))\tr$, we obtain the ODE
\begin{equation}\label{eq:odelarge}
    \begin{pmatrix*}
        \dot u_1(t)\\
        \vdots \\
        \dot u_{\ell-1}(t)\\
        \dot u_\ell(t)\\
        \dot u_{\ell+1}(t)
    \end{pmatrix*}
    =
    \begin{pmatrix*}
        u_2(t)\\
        \vdots \\
        u_\ell(t)\\
        g(u_{\ell+1}(t), u_1(t), \dots, u_\ell(t))\\
        1
    \end{pmatrix*}
    \eqfs
\end{equation}
It is of the form $\dot u(t) = f(u(t))$ for a suitably chosen $f\colon\R^{\tilde d}\to\R^{\tilde d}$, where $\tilde d := d\ell+1$. Hence, \eqref{eq:odelarge} is a $\tilde d$-dimensional, autonomous, first-order ODE. If $u$ is a solution to the ODE \eqref{eq:odelarge}, then $u_1$ is a solution of the ODE \eqref{eq:odehigh}.
\subsection{Formal Description of the General Model}\label{ssec:prelim:model}
We introduce a general model for observations from solutions of an ODE. It is similar to the one in \cite{Lahouel2023}.
\begin{notation}\mbox{ }\label{nota:formalmodel}
    \begin{enumerate}[label=(\roman*)]
        \item
        Let $a,b\in\Z$ with $a \leq b$. For $i\in\nset ab$, let $x_i$ be some object. Set $\indset xab := (x_i)_{i\in\nset ab}$.
        \item\label{nota:formalmodel:norm}
        Let $d\in\N$. For $x\in\R^d$, let $\euclOf{x}$ denote the Euclidean norm. Let $k\in\N$. For a matrix $A \in\R^{k\times d}$, denote the operator norm as $\opNormOf{A} := \sup_{v\in\R^d, \euclof{v}=1}\euclof{A v}$.
        Let $\dm x, \dm y\in\N$.
        Let $f\colon\R^{\dm x}\to\R^{\dm y}$.
        Denote the sup norm of $f$ as $\supNormof{f} := \sup_{x\in\R^{\dm x}} \euclOf{f(x)}$. For linear operators, like the derivative $Df(x)$, we abuse the notation slightly and set $\supNormof{Df} := \sup_{x\in\R^{\dm x}}\opNormOf{Df(x)}$. For more details on derivatives and their norms, see appendix \ref{app:sec:derivative}.
        \item
        Let $\dm{x}, \dm{y}, \beta\in\N$. Denote the set of $\beta$-times continuously differentiable functions $f\colon \R^{\dm x} \to \R^{\dm y}$ as $\mc D^\beta(\R^{\dm x},\R^{\dm y})$. Let $\indset L0\beta\subset\Rpp \cup \{\infty\}$. Denote by $\bar\Sigma^{\dm x\to \dm y}(\beta, \indset{L}{0}{\beta}) \subset \mc D^\beta(\R^{\dm x},\R^{\dm y})$ the set of $\beta$-times continuously differentiable functions $f\colon \R^{\dm x} \to \R^{\dm y}$ with $\supNormOf{D^k f} \leq L_k$ for $k\in\nnzset\beta$.
        More details on derivatives and the smoothness class $\bar\Sigma^{\dm x\to \dm y}(\beta, \indset{L}{0}{\beta})$ are described in appendix \ref{app:sec:derivative}.
    \end{enumerate}
\end{notation}
Let $d\in\N$ be the dimension of the state space $\R^d$. Let $\beta\in\N$ be the smoothness parameter. Let $\indset{L}{0}{\beta}\subset\Rpp$ be the Lipschitz parameters. Let $\FSmooth := \bar\Sigma^{d\to d}(\beta, \indset{L}{0}{\beta})$ be the smoothness class. Let the true model function be $\ftrue\in\FSmooth$. Let $m\in\N$ be the number of observed solutions and let $x_1, \dots, x_m \in \R^d$ be their initial conditions. For $j\in\nnset m$, let $n_j \in\N$ be the number of observations for the $j$-th solution. Denote the total number of observations as $n := \sum_{j=1}^{m} n_j$. For $j\in\nnset m$, let $T_j\in\Rpp$ be the observation duration of the $j$-th solution. For $j\in\nnset m, i\in\nnzset{n_j}$, let $t_{j,i}\in\Rp$ be the observation times such that $0 = t_{j,0}\leq t_{j,1} \leq \dots \leq t_{j,n_j} = T_j$.
Let $\sigma\in\Rp$ be the variance bound of the noise. Let the noise variables $(\epsilon_{j,i})_{j\in\nnset m, i\in\nnset{n_j}}$ be independent $\R^d$-valued random variables such that $\Eof{\epsilon_{j,i}} = 0$ and $\Eof{\euclof{\epsilon_{j,i}}^2} \leq \sigma^2$. For $j\in\nnset m$, $i\in\nnset{n_j}	$, let the observations be
\begin{equation*}
    Y_{j,i} := U(\ftrue, x_j, t_{j,i}) + \epsilon_{j, i}
    \eqfs
\end{equation*}
We observe $Y_{j,i}$ and know $x_j$ and $t_{j,i}$, but $\ftrue$ is unknown and to be estimated.
We assume $d$, $\beta$, $\indset{L}{0}{\beta}$, and $\sigma$ to be fixed and we are interested in upper bounds when $n\to\infty$ for the mean squared error for estimators $\festi$ of  $\ftrue$ on the domain of interest $[0,1]^d$.
\subsection{Need for Restriction of the General Model}\label{ssec:prelim:restrict}
We argue that the model introduced in section \ref{ssec:prelim:model} is not suitable for estimation of the model function in a fixed domain of interest and give first informal descriptions of two possible remedies of the problem: the Snake model and the Stubble model.
\begin{notation}\label{not:asym}\mbox{ }
        Let $(a_n)_{n\in\N}, (b_n)_{n\in\N} \subset \Rpp$. Write
        \begin{align*}
            a_n \asymleq b_n &\ \text{ for } \limsup_{n\to\infty} a_n/b_n < \infty\eqcm\\
            a_n \asymlt b_n &\ \text{ for } \limsup_{n\to\infty} a_n/b_n = 0\eqcm\\
            a_n \asymeq b_n &\  \text{ for } a_n \preccurlyeq b_n \text{ and } b_n \preccurlyeq a_n\eqfs
        \end{align*}
        Define $\asymgt$ and $\asymgeq$ accordingly.
\end{notation}

For the general model, consistent estimation is not possible when considering the domain of interest $[0,1]^d$. We can easily construct settings, where we will never make any observation in say in the ball $\ball^d(0, 1/3) = \setByEleInText{x\in\R^d}{\euclof{x} \leq 1/3}$, i.e., $u(t_{j,i}) \not\in\ball^d(0, 1/3)$ for all $j\in\nnset m, i\in\nnzset{n_j}$, no matter how large $n$ is (e.g., a model function that results in periodic circular trajectories centered at the origin with radius $2/3$ when started from initial conditions $x_j \in \R^2\times \{0\}^{d-2}$ with $\euclof{x_j} = 2/3$). As we are considering a nonparametric class of model functions $\ftrue\in\FSmooth$, we have, for the maximal risk of any estimator $\festi$,
\begin{equation*}
    \sup_{\ftrue\in\FSmooth}\EOff{\ftrue}{\euclOf{\ftrue(0) - \festi(0)}^2} \asymgeq 1
    \eqfs
\end{equation*}

This behavior is different from parametric ODE models (finite dimensional class $\FSmooth$): In the nonparametric setting, observations contain only local information (information about $\ftrue(x)$ for $x$ close to $u(t_{j,i})$), whereas in a parametric setting, observations typically contain global information so that consistent estimation is possible under mild restrictions \cite{Brunel_2008, qi10, gugushvili12, dattner15}.

In this work, we introduce two restrictions to the general nonparametric model that make consistent estimation possible:

In the \textit{Stubble model} (section \ref{sec:stubble}), we require the initial conditions $(x_j)_{j\in\nnset m}$ to suitably cover the domain of interest $[0,1]^d$. By doing so, we ensure that we obtain information from every part of the domain of interest. In this model, we assume to observe many short trajectories, i.e., $m \asymgt 1$ and $\max_{j\in\nnset m}n_j \asymleq 1$, which looks like \textit{stubble}.

In the \textit{Snake model} (section \ref{sec:snake}), we directly require the trajectories $U(\ftrue, x_j, [0, T_j])$ to cover the domain of interest $[0,1]^d$. In this model, we assume to observe few long trajectories, i.e., $\min_{j\in\nnset m}n_j \asymgt 1$. In the extreme case, we have $m=1$. Consistent estimation on all of $[0,1]^d$ with one observed solution is only possible, if its trajectory covers $[0, 1]^d$ suitably. Thus, this trajectory must not intersect itself, as otherwise it would mean that the solution is periodic. This behavior of self-avoidance in a bounded domain resembles the video game \textit{Snake} \cite[chapter 22]{donovan2010replay}.
\subsection{Standard Regression Model}\label{ssec:prelim:regression}
In this article, we present estimation strategies based on black-box regression estimators. To formalize what we mean by \textit{black-box regression estimators} and to be able to talk about their properties, we introduce some terminology in this section and give some examples at the end of the section.
\begin{notation}\mbox{ }
    \begin{enumerate}[label=(\roman*)]
        \item
        Let $(X_n)_{n\in\N}$ be a sequence of real-valued random variables. Let $(a_n)_{n\in\N} \subset \Rpp$. We use the \textit{big O in probability} notation: By $X_n \in \Op(a_n)$, we mean that for all $\epsilon \in\Rpp$, there are $B\in \Rpp$ and $n_0\in\N$ such that
        \begin{equation*}
            \forall n\in \N_{\geq n_0}\colon \PrOf{\abs{\frac{X_n}{a_n}} > B} < \epsilon\eqfs
        \end{equation*}
        \item
        Let $d\in\N$, $k\in\nnset d$, $v = (v_1, \dots, v_d)\in\R^d$. Denote the projection to the $k$-th dimension as $\Pi_k$, i.e., $\Pi_k v = v_k$.
    \end{enumerate}
\end{notation}
Let $\dm x, \dm y\in\N$ be the dimensions of the predictor and the target, respectively.
Let $\mc F$ be a set of functions of the form $\R^{\dm x}\to \R^{\dm y}$, the \textit{smoothness class}.
Let $\ftrue \in\mc F$ be the true \textit{regression function}.
Let $n\in\N$ be the number of observations.
Let $T\in\Rpp$ be the extent of the \textit{domain of interest} $[0, T]^{\dm x}$. For $i\in\nnset n$, let $x_i\in[0, T]^{\dm x}$.
Let $\sigma \in \Rpp$ be the \textit{variance bound} of the noise.
For $i\in\nnset n$, let $\noise_i$ be independent $\R^{\dm y}$-valued random variables with $\Eof{\epsilon_i} = 0$ and $\Eof{\euclof{\epsilon_i}^2} \leq \sigma^2$. We call $\noise_i$  the \textit{noise}.
For $i\in\nnset n$, set the observations as
\begin{equation}\label{eq:prelim:regression}
    Y_i = \ftrue(x_i) + \noise_i
    \eqfs
\end{equation}
The standard regression problem is to estimate $\ftrue$ with an estimator $\festi$ given the \textit{data} $\mc D := (x_i, Y_i)_{i\in\nnset n}$. We denote this statistical model as $\mf P^{\dm x \to \dm y}(\mc F, T, (x_i)_{i\in\nnset n}, \sigma)$.

An \textit{estimator for the regression function} $\ftrue$ in $\mf P^{\dm x \to \dm y}(\mc F, T, (x_i)_{i\in\nnset n}, \sigma)$ is any measurable map
\begin{equation*}
    \Esti^{\dm x\to\dm y}(\cdot, \cdot)\colon(\R^{\dm x}\times\R^{\dm y})^n \times \R^{\dm x} \to \R^{\dm y}
    \eqfs
\end{equation*}
Given the data $\mc D$, it estimates the regression function $\ftrue$ as $x \mapsto \festi(x) := \Esti^{\dm x\to\dm y}(\mc D, x)$.
We define the term \textit{regression estimator for the derivative} $D\ftrue$ of the regression function accordingly.

We say $\Esti^{\dm x\to\dm y}$ is the \textit{componentwise} estimator $\Esti^{\dm x\to1}$, when we apply $\Esti^{\dm x\to1}$ in each target dimension:
\begin{equation*}
    \Esti^{\dm x\to\dm y}(\mc D, \cdot)\colon\R^{\dm x}\to\R^{\dm y},\ x\mapsto
    \br{\Esti^{\dm x\to 1}((x_i, \Pi_1 Y_i)_{i\in\nnset n}, x), \dots, \Esti^{\dm x\to 1}((x_i, \Pi_{\dm y} Y_i)_{i\in\nnset n}, x)}
    \eqfs
\end{equation*}

The \textit{maximal risk} for the model $\mf P_n := \mf P^{\dm x \to \dm y}(\mc F, T, (x_i)_{i\in\nnset n}, \sigma)$ at a point $x_0\in[0, T]^{\dm x}$ of an estimator $\festi = \Esti^{\dm x\to\dm y}(\mc D, \cdot)$ is
\begin{equation}\label{eq:stdregrate}
    r(\festi, \mf P_n, x_0) := \sup_{\Pr\in\mf P_n} \EOff{\Pr}{\euclOf{\festi(x_0) - \ftrue_{\Pr}(x_0)}^2}
    \eqcm
\end{equation}
where $\Pr\in\mf P_n$ is the probability measure that gives rise to \eqref{eq:prelim:regression} with $\ftrue = \ftrue_\Pr$.
We say that $r^{\ms{pr},\ms{sup}}_0(\festi, \mf P_n) \in \Rpp$ is a sup-norm error bound in probability of an estimator $\festi = \Esti^{\dm x\to\dm y}(\mc D, \cdot)$ if, for all $\Pr\in\mf P_n$, for $n\to\infty$, we have
\begin{equation}\label{eq:reg:supnprmrisk}
    \sup_{x\in[0, T]^{\dm x}}\euclOf{\festi(x) - \ftrue_\Pr(x)} \in \Op\brOf{r_0^{\ms{pr},\ms{sup}}(\festi, \mf P_n)}
    \eqfs
\end{equation}
We define the sup-norm error bound of the first derivative accordingly and denote it by $r_1^{\ms{pr},\ms{sup}}(\widehat{Df}, \mf P_n)$.

An example of a regression estimator $\festi$ suitable for the Hölder-type smoothness class $\mc F = \Sigma^{\dm x \to \dm y}(\beta, L)$ with $\beta,L\in\Rpp$ (see \cref{def:Hoelder}), is the componentwise local polynomial estimator described in appendix \ref{app:sec:localreg}. \textit{Suitable} here means that it achieves the lowest values possible for $r(\festi, \mf P_n, x_0)$ for large enough $n\in\N$, up to a constant.

Further examples of (nonparametric) regression estimators that may be used in the general estimation strategies constructed in this article include orthogonal series estimators \cite[chapter 1.7]{Tsybakov09Introduction}, wavelets \cite{Antoniadis2007}, penalized splines \cite{Xiao2019}, and neural networks \cite{SchmidtHieber2020}.
\section{The Stubble Model and Estimation of the Increment Map}\label{sec:stubble}
In the Stubble model, we observe many solutions to an ODE with different initial conditions. For each solution, we have only a few observations. The initial conditions are known and located so that they cover the domain of interest.

In this section, we first present an explicit estimation procedure for a Liptschitz-continuous class of model functions with optimal rate of convergence of the mean squared error at a point. The estimation procedure is based on a local constant estimator for the increment map and a linear interpolation for the model function.  In the second part of this section, we generalize these results: For a general Hölder-smoothness class, we present a black box estimation strategy based on an arbitrary multivariate regression estimator for different increment maps and a univariate polynomial interpolation for the model function. If the nonparametric estimator enjoys certain optimality criteria with respect to the standard nonparametric regression problem, it induces an optimal procedure for nonparametric ODE estimation.
\subsection{Lipschitz}\label{ssec:stubble:lipschitz}
In this specific instance of the general ODE estimation model, we consider Lipschitz-continuous functions $f$, initial conditions $x_j$ on a uniform grid, and an estimation procedure for $\ftrue$ that is based on the local constant estimator for the increments $\Incr(\ftrue, \stepsize, \cdot)$.
\subsubsection{Model}\label{sssec:stubble:lipschitz:model}
The following is a restriction of the general model of section \ref{ssec:prelim:model}.

Let $d\in\N$. Set $\beta=1$. Let $L_0, L_1 \in \Rpp$.
Set $\mc F := \bar\Sigma^{d\to d}(1, L_0, L_1)$.
Let $\ftrue\in\mc F$ and $n_0\in\N$.
Set $n := n_0^d$.
Let $x_i \in [0,1]^d$ for $i\in\nnset n$ form a uniform grid in $[0,1]^d$, i.e.,
\begin{equation*}
	\cb{x_1, \dots, x_n} =
	\setByEle{
		\begin{pmatrix}
			\frac{k_1}{n_0}, \dots, \frac{k_d}{n_0}
		\end{pmatrix}\tr
	}{k_1, \dots, k_d\in\nnset{n_0}}
    \eqfs
\end{equation*}
Let $\sigma\in\Rp$.
Let $\epsilon_j := \epsilon_{j,1}$, $j\in\nnset{n}$ be independent $\R^d$-valued random variables such that $\Eof{\epsilon_j} = 0$ and $\Eof{\euclof{\epsilon_j}^2} \leq \sigma^2$.
Let $\stepsize\in\Rpp$.
Set
\begin{equation*}
	Y_j := Y_{j, 1} := U(\ftrue, x_j, \stepsize) + \epsilon_j
	\eqfs
\end{equation*}
We observe $Y_j$ and know $x_j$ and $\stepsize$, but $\ftrue$ is unknown and to be estimated.
We assume $d$, $L_0, L_1$, and $\sigma$ to be fixed. We are interested in upper bounds for the mean squared error at a point $x_0 \in [0,1]^d$ depending on the asymptotics of $n$ and $\stepsize$.
\subsubsection{Estimator}\label{sssec:stubble:lipschitz:esti}
For $f\in\mc F$, $\stepsize\in\Rp$, $x\in\R^d$, recall the definition of the increment map $\Incr(f, \stepsize, x) = U(f, x, \stepsize) - x$ and set $\itrue(x) := \Incr(\ftrue, \stepsize, x)$.

Let $\iesti$ be the componentwise (see section \ref{ssec:prelim:regression}) local constant estimator (Nadaraya--Watson) of $\itrue$ with kernel $K$ using the data $(x_j, Y_j)_{j\in\nnset n}$ and a bandwidth of optimal order. See appendix \ref{app:sec:localreg} with $\beta=1, \ell = 0, s=0$ for details on the local constant estimator.

For  $x\in\R^d$, define the scaled increment estimator of $\ftrue$ as
\begin{equation}\label{eq:stubble:lipschitz:festi}
	\festi(x) := \frac{\iesti(x)}{\stepsize}
	\eqfs
\end{equation}
\subsubsection{Result}
\begin{assumption}\label{ass:strictkernel}\mbox{ }
    \begin{itemize}
        \item \newAssuRef{StrictKernel}: The kernel $K \colon \Rp \to \R$ is Lipschitz-continuous, nonnegative, non-trivial ($K\neq 0$), and there is $\const{ker}\in\Rpp$ such that
        \begin{equation*}
            K(x) \leq \const{ker} \ind_{[0, 1]}(x)
        \end{equation*}
        for all $x\in\Rp$.
    \end{itemize}
\end{assumption}
\begin{theorem}\label{thm:stubble:lip}
    Use the model of section \ref{sssec:stubble:lipschitz:model} and the estimator of section \ref{sssec:stubble:lipschitz:esti}.
    Assume \assuRef{StrictKernel}.
	Assume $\stepsize \asymleq 1$.
	Then
	\begin{equation*}
		\EOf{\euclOf{\festi(x_0)-\ftrue(x_0)}^2} \asymleq \br{\stepsize^2 n}^{-\frac{2}{2+d}} + \stepsize^2
		\eqfs
	\end{equation*}
\end{theorem}
\begin{remark}
    By \cite[Corollary 3.11]{lowerbounds}, the error rate in \cref{thm:stubble:lip} is minimax optimal.
\end{remark}
The following corollary minimizes the error bound over $\stepsize$, i.e., it shows the asymptotic behavior of $\stepsize$ that allows the best estimates of $f$ for the same amount of data.
\begin{corollary}\label{cor:stubble:lip}
    Use the model of section \ref{sssec:stubble:general:model} and the estimator of section \ref{sssec:stubble:lipschitz:esti}.
    Assume \assuRef{StrictKernel}.
	Assume $\stepsize \asymeq n^{-\frac1{4+d}}$. Then
	\begin{equation*}
		\EOf{\euclOf{\festi(x_0)-\ftrue(x_0)}^2} \asymleq n^{-\frac2{4+d}}
		\eqfs
	\end{equation*}
\end{corollary}
\subsubsection{Proof}
\begin{lemma}\label{lmm:stubble:smooth}
	Assume $f\in\mc F$.
	Let $\incr := \Incr(f, \stepsize, \cdot)$.
	Then $\incr \in\bar\Sigma^{d\to d}(1, \tilde L_0, \tilde L_1)$ with
	\begin{equation*}
		\tilde L_0 = \stepsize L_0\,,\quad \tilde L_1 = \exp(L_1 \stepsize )-1
		\eqfs
	\end{equation*}
\end{lemma}

\begin{proof}
	By the definition of $U$ and the fundamental theorem of calculus,
	\begin{equation*}
		U(f, x, t) = x + \int_0^t f(U(f, x, s)) \dl s
		\eqfs
	\end{equation*}
	Plugging this into the definition of $\incr$ yields
	\begin{equation}\label{eq:incrementintegral}
		\incr(x) = U(f, x, \stepsize) - U(f, x, 0) = \int_0^{\stepsize} f(U(f, x, s)) \dl s
		\eqfs
	\end{equation}
	Thus, as $f$ is uniformly bounded by $L_0$,
	\begin{equation*}
		\abs{\incr(x)} \leq  \int_0^{\stepsize} \abs{f(U(f, x, s))} \dl s \leq \int_0^{\stepsize} L_0 \dl s = \stepsize L_0
		\eqfs
	\end{equation*}
	This is the value of $\tilde L_0$. \cref{lmm:derivIncrem} yields the value of $\tilde L_1$ by
	\begin{equation*}
		\abs{D \incr}_\infty \leq \exp(L_1 t)-1
		\eqfs
	\end{equation*}
\end{proof}
\begin{lemma}\label{lmm:stubble:lip:nw}
	Assume $\stepsize \asymleq 1$.
	Then, for all $x_0\in[0,1]^d$,
	\begin{equation*}
		\EOf{\euclOf{\iesti(x_0)-\itrue(x_0)}^2} \asymleq \br{\stepsize^d n^{-1}}^{\frac{2}{2+d}}
		\eqfs
	\end{equation*}
\end{lemma}
\begin{proof}
	According to \cref{lmm:stubble:smooth}, $\itrue$ is Lipschitz-continuous with constant $\tilde L_1 = \exp(L_1 \stepsize)-1$. As we consider $L_1$ to be constant and $\stepsize\asymleq 1$, we obtain $\tilde L_1 \asymleq \stepsize$.

    By \cref{cor:localpoly} and \cref{prp:uniformgrid}, the local constant estimator $\iesti$ with optimal bandwidth then fulfills
    \begin{equation*}
        \EOf{\euclOf{\iesti(x_0)-\itrue(x_0)}^2} \asymleq \stepsize^{\frac{2d}{2+d}} n^{-\frac{2}{2+d}}
        \eqfs
    \end{equation*}
\end{proof}
\begin{lemma}\label{lmm:stubble:lip:approx}
	Let $\stepsize\in\Rpp$. Assume $f\in\mc F$. Let $\incr := \Incr(f, \stepsize, \cdot)$.
	Then
	\begin{equation*}
		\sup_{x\in[0,1]^d}\euclOf{\frac{\incr(x)}{\stepsize} - f(x)} \leq \frac12 L_0  L_1 \stepsize
		\eqfs
	\end{equation*}
\end{lemma}
\begin{proof}
	Let $x\in[0,1]^d$. With \eqref{eq:incrementintegral}, as $f$ is $L_1$-Lipschitz,
	\begin{align*}
		\euclOf{\frac{\incr(x)}{\stepsize} - f(x)}
		&=
		\euclOf{\frac1{\stepsize}\int_0^{\stepsize} f(U(f, x, s)) - f(x)\dl s}
		\\&\leq
		\frac{L_1}{\stepsize}\int_0^{\stepsize} \euclOf{U(f, x, s) - x} \dl s
        \eqfs
	\end{align*}
	Using $U(f, x, s) = x + \int_0^s f(U(f, x, r)) \dl r$, we obtain
	\begin{align*}
		\int_0^{\stepsize} \euclOf{U(f, x, s) - x} \dl s
		&\leq
		\int_0^{\stepsize} \int_0^s \euclOf{f(U(f, x, r))} \dl r \dl s
		\\&\leq
		\int_0^{\stepsize}\int_0^s L_0  \dl r \dl s
		\\&=
		\frac12 L_0 \stepsize^2
		\eqfs
	\end{align*}
	Combining the above inequalities yields
	\begin{align*}
		\euclOf{\frac{\incr(x)}{\stepsize} - f(x)}
		&\leq
		\frac12 L_0  L_1 \stepsize
		\eqfs
	\end{align*}
	As this inequality holds for all $x\in[0,1]^d$, we have finished the proof.
\end{proof}
\begin{proof}[Proof of \cref{thm:stubble:lip}]
	Recall the definition of $\festi$ in \eqref{eq:stubble:lipschitz:festi} and combine the bounds in \cref{lmm:stubble:lip:nw} and
	\cref{lmm:stubble:lip:approx} to obtain
	\begin{align*}
		\EOf{\euclOf{\festi(x_0)-\ftrue(x_0)}^2}
		&=
		\EOf{\euclOf{\frac{\iesti(x_0)-\itrue(x_0)}{\stepsize}+\frac{\itrue(x_0)}{\stepsize}-\ftrue(x_0)}^2}
		\\&\asymleq
		\frac1{\stepsize^2} \EOf{\euclOf{\iesti(x_0)-\itrue(x_0)}^2} +
		\euclOf{\frac{\itrue(x_0)}{\stepsize} - \ftrue(x_0)}^2
		\\&\asymleq
		\stepsize^{-\frac{4}{2+d}} n^{-\frac{2}{2+d}} + \stepsize^2
		\eqfs
	\end{align*}
\end{proof}
\subsection{General}\label{ssec:stubble:general}
We generalize the Lipschitz setting of section \ref{ssec:stubble:lipschitz} in the following way: We consider an arbitrary smoothness parameter $\beta\in\N$ for the model function $\ftrue$. The initial conditions are not restricted to a uniform grid, but must still be \textit{uniform enough} in some sense. We present an estimation strategy that can be used with any regression estimator. The convergence rate results are given in a black box fashion, i.e., they depend on the convergence rates of the chosen regression estimator. If that estimator achieves the optimal rate of convergence for a standard nonparametric regression problem, the resulting ODE estimator is also optimal in a minimax sense.
\begin{notation}\mbox{ }\label{not:poly}
    \begin{enumerate}[label=(\roman*)]
        \item
        Let $d\in\N$. Let $\alpha\in\N_0^d$.
        \begin{itemize}
            \item Denote $\abs{\alpha} := \sum_{i=1}^d \alpha_i$,
            \item Denote $\alpha! := \prod_{i=1}^d \alpha_i!$,
            \item Let $x\in\R^d$. Denote $x^\alpha = \prod_{i=1}^d x_i^{\alpha_i}$.
        \end{itemize}
        \item
        For $\ell,d\in\N$, denote $\Poly d\ell$ be the set of all polynomials of degree at most $\ell$ defined over $\R^d$,
        \begin{equation*}
            \Poly d\ell := \setByEle{p\colon\R^d\to\R}{p(x) = \sum_{\alpha\in\N_0^d,\abs{\alpha}\leq\ell} \beta_\alpha x^\alpha\,,\ \beta_\alpha\in\R}
            \eqfs
        \end{equation*}
        \item
        Let $A$ be a set. Denote the indicator function of $A$ as $\indOfOf{A}{\cdot}$.
    \end{enumerate}
\end{notation}
\subsubsection{Model}\label{sssec:stubble:general:model}
The following is a restriction of the general model of section \ref{ssec:prelim:model}.

Let $d\in\N$, $\beta\in\N$, and $\indset L0\beta\subset \Rpp$. Set the smoothness class as $\FSmooth := \bar\Sigma^{d\to d}(\beta, \indset{L}{0}{\beta})$, see \cref{def:Hoelder}. Let $\ftrue\in\FSmooth$. Let $m\in\N$ and $x_1, \dots, x_m \in \R^d$. Let $\stepsize\in\Rpp$. Set $n_j = \beta$ for all $j\in\nnset m$. Set $t_{j,i} = t_i = i\stepsize$ for $i \in \nnzset{\beta}$.
Let $\sigma\in\Rp$. Let $(\epsilon_{j,i})_{j\in\nnset m, i\in\nnset\beta}$ be independent $\R^d$-valued random variables such that $\Eof{\epsilon_{j, i}} = 0$ and $\Eof{\euclof{\epsilon_{j, i}}^2} \leq \sigma^2$. For $j\in\nnset m$, $i\in\nnset \beta$, let
\begin{equation*}
    Y_{j,i} = U(\ftrue, x_j, i\stepsize) + \epsilon_{j, i}
    \eqfs
\end{equation*}
We observe $Y_{j,i}$, and know $x_j$ and $\stepsize$, but $\ftrue$ is unknown and to be estimated.
We assume $d$, $\indset{L}{0}{\beta}$, and $\sigma$ to be fixed and we are interested in asymptotic upper bounds for the mean squared error at a point $x_0 \in [0,1]^d$ depending on $n$ and $\stepsize$.
\begin{remark}\mbox{ }
        The results below are the same (up to a constant) if we allow $n_j\geq \beta$ as long as $m \asymeq n$.
\end{remark}
\subsubsection{Estimation}\label{sssec:stubble:general:esti}
For $i\in\nnset{\beta}$, let $\itrue_{i} = \Incr(\ftrue, t_i, \cdot)$. Let $L\in\Rpp$. Let $\Esti^{d\to d}$ be an arbitrary regression estimator for the standard regression model $\mf P := \mf P^{d \to d}(\Sigma^{d\to d}(\beta, L), 1, (x_j)_{j\in\nnset m}, \sigma)$, see section \ref{ssec:prelim:regression}. For $i\in\nnset\beta$, estimate the increment maps $\itrue_{i}$ by $\iesti_{i} := \Esti^{d\to d}((x_j, Y_{j,i}-x_j)_{j\in\nnset m}, \cdot)$. For convenience, set $\iesti_{0}(x) := 0$ for all $x\in\R^d$.
Denote the maximal risk of the regression estimator, see \eqref{eq:stdregrate}, as,
\begin{equation*}
    r(d, \beta, L, \sigma, (x_j)_{j\in\nnset m}) := \sup_{x\in[0,1]^d}r(\Esti^{d\to d}, \mf P, x)
    \eqfs
\end{equation*}

For $x_0\in\R^d$, let $\pesti(x_0, \cdot) := (\pesti_1(x_0, \cdot), \dots, \pesti_d(x_0, \cdot))$, where $\pesti_k(x_0, \cdot)\in\Poly1\beta$ is the (univariate) polynomial of at most degree $\beta$ that interpolates $(t_i, \Pi_k\iesti_{i}(x_0))_{i\in\nnzset\beta}$. To estimate $\ftrue$, we locally approximate $U(\ftrue, x_0, t)$ by $x_0 + \pesti(x_0, t)$, i.e., we consider $\dpesti(x_0, t) \approx \ftrue(x_0 + \pesti(x_0, t))$. As $\ftrue(x_0) = \ftrue(x_0 + \pesti(x_0, 0))$, this approach yields the estimator
\begin{equation}\label{eq:stubble:general:festi}
	\festi(x_0) := \dpesti(x_0, 0)
	\eqfs
\end{equation}
\begin{remark}\mbox{ }
	\begin{enumerate}[label=(\roman*)]
        \item The value of $L$ is to be specified later and depends on $\stepsize$, $\beta$, and $\indset L0\beta$. Typically the estimator $\Esti^{d\to d}$ does not depend on $L$, but the risk $r(\Esti^{d\to d}, \mf P, x)$ does.
        \item If we set $\beta = 1$ and use the componentwise local constant estimator for $\Esti^{d\to d}$, we obtain the procedure of section \ref{ssec:stubble:lipschitz}.
		\item The polynomial interpolation in the second step can be viewed as an Adams–Bashforth method: a linear multistep method for solving ordinary differential equations numerically, see \cite[Chapter 24]{Butcher2016}.
	\end{enumerate}
\end{remark}
\subsubsection{Result}\label{sssec:stubble:general:result}
\begin{theorem}\label{thm:stubble:general}
    Use the model of section \ref{sssec:stubble:general:model} and the estimator of section \ref{sssec:stubble:general:esti}.
	Assume $\stepsize \asymleq 1$.
    Then
	\begin{equation*}
		\EOf{\euclOf{\festi(x_0) - \ftrue(x_0)}^2}
		\asymleq
        \stepsize^{-2}  r(d, \beta, C\stepsize, \sigma, (x_j)_{j\in\nnset m}) + \stepsize^{2\beta}
        \eqfs
	\end{equation*}
    for some constant $C\in\Rpp$ depending only on $\beta$ and $\indset L0\beta$.
\end{theorem}
The minimax rate for the squared error in the standard regression problem with smoothness class $\Sigma^{d \to d}(\beta, L)$ is $L^{\frac{2d}{2\beta+d}}n^{-\frac{2\beta}{2\beta+d}}$ up to some constants. It is achieved, for example, by the local polynomial estimator with optimal bandwidth under some conditions on $x_i$, which are fulfilled by a uniform grid. See \cref{cor:localpoly}. Considering $m \asymeq n$, we then have
\begin{equation*}
    r(d, \beta, C\stepsize, \sigma, (x_j)_{j\in\nnset m}) \asymleq \stepsize^{\frac{2d}{2\beta+d}}n^{-\frac{2\beta}{2\beta+d}}
    \eqfs
\end{equation*}
The conditions for the optimal rate for the local polynomial estimator with kernel $K$ of appendix \ref{app:sec:localreg} are as follows.
\begin{assumption}\label{ass:localreg}\mbox{ }
    \begin{itemize}
        \item \newAssuRef{Eigenvalue}:
        There is $\const{egv}\in\Rpp$ such that
        \begin{equation*}
            \lambda_{\ms{min}}(B(x))^{-1} \leq \const{egv}
        \end{equation*}
        for all $x \in [0, T]^d$, where $B(x)$ is defined in \eqref{eq:lp:B}.
        \item \newAssuRef{Cover}:
        There is $\const{cvr}\in\Rpp$ such that
        \begin{equation*}
            \frac1n\sum_{i=1}^n \indOfOf{\ball^d(x, r)}{x_i} \leq \max\brOf{\frac1n,\, \const{cvr}\br{\frac{r}{T}}^d}
            \eqcm
        \end{equation*}
        for all $x\in[0,T]^d$, $r\in\Rpp$.
        \item \newAssuRef{Kernel}:
        The support of the kernel $K\colon \Rp \to \R$ fulfills $\supp(K) \subset [0, 1]$.
        Furthermore, there is $\const{ker}\in\Rpp$ such that
        \begin{equation*}
            K(z) \leq \const{ker}
        \end{equation*}
        for all $z \in \Rp$.
    \end{itemize}
\end{assumption}
The following corollary shows the result of \cref{thm:stubble:general} using a minimax optimal estimator $\Esti^{d\to d}$.
\begin{corollary}\label{cor:stubble:general}
    Use the model of section \ref{sssec:stubble:general:model}. Use the estimator of section \ref{sssec:stubble:general:esti} with the componentwise local polynomial estimator of degree $\ell := \beta-1$ with kernel $K$ and optimal bandwidth as the estimator $\Esti^{d\to d}$.
    Assume $\stepsize \asymleq 1$.
    Assume \assuRef{Cover}, \assuRef{Eigenvalue}, and \assuRef{Kernel}.
    Then, for all $x_0\in[0,1]^d$,
    \begin{equation*}
        \EOf{\euclOf{\festi(x_0) - \ftrue(x_0)}^2}
        \asymleq \br{\stepsize^2 n}^{-\frac{2\beta}{2\beta+d}} + \stepsize^{2\beta}
        \eqfs
    \end{equation*}
\end{corollary}
\begin{remark}\mbox{ }
    \begin{enumerate}[label=(\roman*)]
        \item
        Of course, the local polynomials can be replaced by any estimator that achieves the same (optimal) rate of convergence.
        \item
        By \cite[Corollary 3.11]{lowerbounds}, the error rate in \cref{cor:stubble:general} is minimax optimal.
    \end{enumerate}
\end{remark}
The following corollary minimizes the error bound over $\stepsize$, i.e., it shows the asymptotic behavior of $\stepsize$ that allows the best estimates of $f$ for the same amount of data.
\begin{corollary}\label{cor:stubble:fast}
    Use the setting and assumptions of \cref{cor:stubble:general}.
	Assume
    \begin{equation*}
		\stepsize \asymeq n^{-\frac1{2(\beta+1)+d}}
        \eqfs
	\end{equation*}
    Then
	\begin{equation*}
		\EOf{\euclOf{\festi(x_0) - \ftrue(x_0)}^2}
		\asymleq
		n^{-\frac{2\beta}{2(\beta+1)+d}}
        \eqfs
	\end{equation*}
\end{corollary}
\subsubsection{Proof}
For $x_0\in\R^d$, let $\ptrue(x_0, \cdot)\in\Poly1\beta$ be the (univariate) polynomial of at most degree $\beta$ that interpolates $(t_i, \itrue_{i}(x_0))_{i\in\nnzset{\beta}}$. By the definition of $\festi$ in \eqref{eq:stubble:general:festi}, we can write
\begin{equation*}
	\festi(x_0) - \ftrue(x_0) =
	\dpesti(x_0, 0) - \dptrue(x_0, 0) + \dptrue(x_0, 0) - \ftrue(x_0)
	\eqfs
\end{equation*}
Using this, we split the error into two parts,
\begin{equation}\label{eq:stubble:general:split}
	\EOf{\euclOf{\festi(x_0) - \ftrue(x_0)}^2}
	\leq
	2\EOf{\euclOf{\dpesti(x_0, 0) - \dptrue(x_0, 0)}^2} + 2 \euclOf{\dptrue(x_0, 0) - \ftrue(x_0)}^2
	\eqfs
\end{equation}
We start with the second part: The polynomial $t\mapsto \ptrue(x_0, t)$ interpolates the translated solution $\bar u^\star(x_0, t) := U(\ftrue, x_0, t) - x_0$. By \cref{coro:trajSmooth}, we have $\abs{D^{\beta+1} \bar u^\star(x_0, \cdot)}_{\infty} \leq C_L$, where $C_L$ depends only on $\beta$ and $\indset L0\beta$.
By the definition of $\bar u^\star(x_0, t)$, it fulfills $\dot{\bar u}^\star(x_0, 0) = \ftrue(x_0)$.
Thus, \cref{coro:univariatePolyApprox} provides a constant $C_\beta$ depending only on $\beta$ such that
\begin{align*}
	\euclOf{\dptrue(x_0, 0) - \ftrue(x_0)}
	&=
	\euclOf{\dptrue(x_0, 0) - \dot{\bar u}^\star(x_0, 0)}
	\\&\leq
	C_{\beta} C_L \stepsize^{\beta}
	\eqfs
\end{align*}
In other words,
\begin{equation}\label{eq:stubbles:general:polyderiv}
	\euclOf{\dptrue(x_0, 0) - \ftrue(x_0)}^2 \asymleq \stepsize^{2\beta}
	\eqfs
\end{equation}
Let us turn our attention to the first part of the right-hand side of \eqref{eq:stubble:general:split}: By \cref{lmm:highDerivIncrem}, $\abs{D^\beta\itrue_i}_\infty \leq C \stepsize$ for some constant $C$ depending only on $d,\beta$, $\indset{L}{0}{\beta}$ and the maximal $\stepsize$ as $\ftrue \in\bar\Sigma^{d\to d}(\beta, \indset L0\beta)$ and $\stepsize \asymleq 1$.
As $\Esti^{d\to d}$ achieves the rate of convergence $r(d, \beta, L, \sigma, (x_j)_{j\in\nnset m})$ on $\Sigma^{d\to d}(\beta, L)$, we obtain
\begin{equation}\label{eq:stubble:increamentMSE}
	\EOf{\euclOf{\iesti_{i}(x_0) - \itrue_{i}(x_0)}^2} \asymleq r(d, \beta, C\stepsize, \sigma, (x_j)_{j\in\nnset m})
	\eqfs
\end{equation}
Applying \cref{lem:polycompare} with the bound \eqref{eq:stubble:increamentMSE} yields
\begin{equation}\label{eq:stubbles:general:polyesti}
	\EOf{\euclOf{\dpesti(x_0, 0) - \dptrue(x_0, 0)}^2} \asymleq \stepsize^{-2} r(d, \beta, C\stepsize, \sigma, (x_j)_{j\in\nnset m})
	\eqfs
\end{equation}
We combine \eqref{eq:stubbles:general:polyderiv} and \eqref{eq:stubbles:general:polyesti} and get
\begin{equation*}
	\EOf{\euclOf{\festi(x_0) - \ftrue(x_0)}^2}
	\asymleq \stepsize^{-2} r(d, \beta, C\stepsize, \sigma, (x_j)_{j\in\nnset m}) + \stepsize^{2\beta}
    \eqfs
\end{equation*}
\section{The Snake Model and Estimation of the Observed Solutions}\label{sec:snake}
In the Snake model, we observe one (or a few) solution(s) to an ODE. For each solution, we have many observations. The associated estimation problem depends strongly on how the trajectories of the solutions are located in the state space.

In this section, we first present an explicit estimation procedure for a Lipschitz-continuous class of model functions and an upper bound on the maximal risk in sup norm. The estimation procedure is based on a local linear estimator for the solutions and a nearest neighbor interpolation for the model function. It is rate-optimal in some settings. In the second part of this section, we generalize these results: For a general Hölder-smoothness class, we present a black box estimation strategy based on a generic nonparametric estimator for the solutions and a multivariate polynomial interpolation for the model function. If the nonparametric estimator enjoys certain optimality criteria with respect to the standard nonparametric regression problem, it induces a procedure for nonparametric ODE estimation that is optimal in some settings.
\subsection{Lipschitz}\label{ssec:snake:lipschitz}
In this specific instance of the general ODE estimation model, we consider Lipschitz-continuous model functions $\ftrue$ and an estimation procedure $\festi$ that is based on local linear estimators $\uesti$ and $\duesti$ for the observed solution $U(\ftrue, x_1, \cdot)$ and its derivative, and a nearest neighbor interpolation of $\uesti(t) \mapsto \duesti(t)$ for the model function. We obtain an upper bound on the expected sup norm error. For this, we either require the trajectories $U(\ftrue, x_1, [0, T_1])$ to cover the fixed domain of interest $[0,1]^d$ in a suitable way or we adapt our domain of interest to these trajectories, i.e., the upper bound only holds close to the trajectory.
\begin{notation}\mbox{ }
        Let $d\in\N$, $r\geq 0$, and $z\in \R^d$. Denote the closed ball with radius $r$ around $z$ as
        \begin{equation*}
            \ball^d(z, r) := \setByEle{x\in\R^d}{\euclOf{x-z} \leq r}
            \eqfs
        \end{equation*}
        Let $\mc Z\subset \R^d$. Denote the closed ball with radius $r$ around $\mc Z$ as
        \begin{equation*}
            \ball^d(\mc Z, r) := \bigcup_{z\in\mc Z} \ball^d(z, r)
            \eqfs
        \end{equation*}
\end{notation}
\subsubsection{Model}\label{sssec:snake:lip:model}
The following is a restriction of the general model of section \ref{ssec:prelim:model}.

Let $d\in\N_{\geq 2}$ and $L_0, L_1\in\Rpp$.
Set $\mc F := \bar\Sigma^{d\to d}(1, L_0, L_1)$.
Let $\ftrue\in\mc F$.
Set $m=1$. Let $x_1 \in \R^d$, $n=n_1\in\N$, and $\stepsize \in \Rpp$.
Set $t_{1,i} := t_i := i \stepsize$ for $i=\nnzset{n}$ and $T := T_1 := n \stepsize$.
Let $\sigma\in\Rp$.
Let $\epsilon_i := \epsilon_{1,i}$, $i\in\nnset{n}$ be independent $\R^d$-valued random variables such that $\Eof{\epsilon_i} = 0$ and $\Eof{\euclof{\epsilon_i}^2} \leq \sigma^2$.
Set
\begin{equation*}
	Y_i := Y_{1,i} := U(\ftrue, \uztrue, t_i) + \epsilon_i \qquad \text{for } i\in\nnset{n}
	\eqfs
\end{equation*}
We observe $Y_i$ and know $x_1$ and $\stepsize$, but $\ftrue$ is unknown and to be estimated.
We assume $d$, $L_0, L_1$, and $\sigma$ to be fixed. For an estimator $\festi$, we are interested in asymptotic upper bounds for the mean squared sup-norm of $\ftrue - \festi$ in the domain of interest $[0,1]^d$ depending on the asymptotics of $n$ and $\stepsize$.
\subsubsection{Estimator}\label{sssec:snake:lip:esti}
Let $\uesti \colon[0, T] \to\R$ be the componentwise local linear estimator (as described in \cref{app:sec:localreg}, \eqref{eq:lp:def} with $\ell = 1$, $s=0$, $d=1$) of $\utrue := U(\ftrue, \uztrue, \cdot)$ using the data $(t_i, Y_i)_{i\in\nnset{n}}$. Similarly, denote by $\duesti$ the local linear estimator of the first derivative of $U(\ftrue, \uztrue, \cdot)$  ($\ell = 1$, $s=1$, $d=1$ in terms of \cref{app:sec:localreg}). We assume that the bandwidth $h$ is chosen of optimal order for the error bound in sup-norm, see \cref{cor:localpoly:supnorm}.

Let $\festi$ be the nearest neighbor interpolation of $(\uesti, \duesti)$: For $x\in\R^d$ and a continuous function $u\colon[0, T]\to \R^d$, let
\begin{equation*}
    \nn(u, T, x) \in \argmin_{t\in[0, T]} \euclof{u(t) - x}
\end{equation*}
with an arbitrary choice if the minimizer is not unique. Then, for $x\in\R^d$, we define the estimator of $\ftrue$ as
\begin{equation}\label{eq:snake:lip:esti}
	\festi(x) := \duesti\brOf{\nn(\uesti, T, x)}
	\eqfs
\end{equation}
\subsubsection{Result}\label{sssec:snake:lip:result}
For our results in sup-norm, we require the noise to be sub-Gaussian, which is a standard assumption for regression results in sup-norm, see \cite[section 1.6.2]{Tsybakov09Introduction}.
\begin{definition}[Sub-Gaussian]\label{def:subgauss:variable}
    A real-valued random variable $Z$ is called \textit{sub-Gaussian} if $\Eof{\abs{Z}} < \infty$ and there is $v\in\Rpp$ such that
    \begin{equation*}
        \PrOf{\abs{Z - \EOf{Z}} \geq t} \leq 2 \exp\brOf{-\frac{t^2}{2v}}
    \end{equation*}
    for all $t\in\Rpp$. In this case, $v$ is called \textit{variance parameter}.
\end{definition}
\begin{assumption}\mbox{ }
	\begin{itemize}
	   \item \newAssuRef{SubGaussian}:  In each component, the noise $\noise_i$ is sub-Gaussian with variance parameter $\sigma^2$.
	\end{itemize}
\end{assumption}
For $\mc X \subset \R^d$ and a function $u\colon[0, T]\to \R^d$, define $\delta_{\ms{NN}}(u, T, \mc X)$ as smallest radius so that the ball around $u([0,T])$ includes $\mc X$:
\begin{equation*}
    \delta_{\ms{NN}}(u, T, \mc X) := \sup_{x\in\mc X} \inf_{t\in[0, T]} \euclOf{u(t) - x}
    \eqfs
\end{equation*}
The results come in two versions with respect to the domain of interest: We either fix the domain of interest, e.g., $\mc X = [0,1]^d$ and make the error bound depend on $\delta_{\ms{NN}}(\utrue, T, \mc X)$, or we fix $\delta$ (or a sequence $\delta = \delta_n$) and adapt our domain of interest to be $\mc X = \ball^d(\utrue([0, T]), \delta)$.
\begin{theorem}\label{thm:snake:lip}
    Use the model of section \ref{sssec:snake:lip:model} and the estimator of section \ref{sssec:snake:lip:esti}.
    Assume \assuRef{StrictKernel}, \assuRef{SubGaussian}.
	Assume
	\begin{equation}\label{eq:snake:Tbound}
        \br{\frac{n}{\log(n)}}^{-\frac{1}{4}} \asymleq T \asymleq n \log(n)^{\frac14}
        \eqfs
    \end{equation}
    \begin{enumerate}[label=(\roman*)]
        \item\label{thm:snake:lip:self}
        Let $\delta \in\Rp$, potentially changing with $n$.
        Let $\mc X := \ball^d(\utrue([0, T]), \delta)$.
        Then
        \begin{equation*}
            \EOf{\sup_{x\in\mc X}\euclOf{\festi(x) - \ftrue(x)}^2}
            \asymleq
            \delta^2 + \br{\frac{T\log n}{n}}^{\frac25}
            \eqfs
        \end{equation*}
        \item\label{thm:snake:lip:domain}
        Let $\delta = \delta_{\ms{NN}}(\utrue, T, [0,1]^d)$.
        Then
        \begin{align*}
            \EOf{\sup_{x\in[0,1]^d}\euclOf{\festi(x) - \ftrue(x)}^2}
            \asymleq
            \delta^2 + \br{\frac{T\log n}{n}}^{\frac25}
            \eqfs
        \end{align*}
    \end{enumerate}
\end{theorem}
\begin{remark}\mbox{ }\label{rem:snake:results:lip}
    \begin{enumerate}[label=(\roman*)]
        \item
        Compare the upper error bound in \cref{thm:snake:lip} \ref{thm:snake:lip:domain} with the lower error bound \cite[Corollary 4.11]{lowerbounds}
        (for simplicity, we here ignore all factors in the error bounds that are polynomial in $\log(n)$):
        \begin{equation*}
        	\delta + \delta^{\frac{d-1}{4+d}} \br{\frac{T}{n}}^{\frac1{4+d}}
        	\asymleq
        	\EOf{\sup_{x\in[0,1]^d}\euclOf{\festi(x) - \ftrue(x)}^2}^{\frac12}
        	\asymleq
        	\delta + \br{\frac{T}{n}}^{\frac15}
        	\eqfs
        \end{equation*}
        Note that 
        \begin{equation*}
        	\br{\frac Tn}^{\frac15} \leq \delta
        	\quad\Leftrightarrow\quad
        	\br{\frac Tn}^{\frac15} \leq \delta^{\frac{d-1}{4+d}}\br{\frac{T}{n}}^{\frac1{4+d}}
        	\quad\Leftrightarrow\quad
        	\delta^{\frac{d-1}{4+d}}\br{\frac{T}{n}}^{\frac1{4+d}} \leq \delta
        	\eqfs
        \end{equation*}
        Thus, the error rate in \cref{thm:snake:lip} \ref{thm:snake:lip:domain} is minimax optimal if
        \begin{equation}\label{eq:space:separate}
            \br{\frac{T}{n}}^{\frac{1}{5}} \asymleq \delta
            \eqcm
        \end{equation}
        ignoring log-factors.
        \item
        Note that \cite[Corollary 4.11]{lowerbounds} uses $m\asymgt 1$ many solutions. But in \cite[Appendix E]{lowerbounds}, it is argued that $m=1$ suffices.
        \item One can think of \eqref{eq:space:separate} as the requirement of large distances between temporally different parts of trajectory. If this is not the case, then the estimator seems suboptimal: For the estimation of $\uesti$ and $\duesti$, it ignores information from observations that are distant in time but close in state space. An estimator that is optimal for regimes where different trajectory parts are not well-separated, must not ignore this information.
    \end{enumerate}
\end{remark}
\begin{corollary}\label{cor:snake:lip}
    Use the model of section \ref{sssec:snake:lip:model} and the estimator of section \ref{sssec:snake:lip:esti}.
    Assume \assuRef{StrictKernel}, \assuRef{SubGaussian}.
    \begin{enumerate}[label=(\roman*)]
        \item
        Let $\delta \in\Rp$, potentially changing with $n$.
        Let $\mc X := \ball^d(\utrue([0, T]), \delta)$.
        Assume $T \asymleq \br{\frac{n}{\log n}}^{\frac{d-1}{4+d}}$ and $\delta \asymleq \br{\frac{n}{\log n}}^{-\frac{1}{4+d}}$.
        Then
        \begin{equation*}
           \EOf{\sup_{x\in\mc X}\euclOf{\festi(x) - \ftrue(x)}^2}
            \asymleq \br{\frac n{\log n}}^{-\frac{2}{4+d}}
            \eqfs
        \end{equation*}
        \item
        Let $\delta = \delta_{\ms{NN}}(\utrue, T, [0,1]^d)$.
        Assume $T \asymleq \br{\frac{n}{\log n}}^{\frac{d-1}{4+d}}$ and $\delta \asymleq \br{\frac{n}{\log n}}^{-\frac{1}{4+d}}$.
        Then
        \begin{equation*}
            \EOf{\sup_{x\in[0,1]^d}\euclOf{\festi(x) - \ftrue(x)}^2}
            \asymleq \br{\frac n{\log n}}^{-\frac{2}{4+d}}
            \eqfs
        \end{equation*}
    \end{enumerate}
\end{corollary}
\begin{remark}\mbox{ }
        The conditions on $T$ and $\delta$ ensure an optimal trade-off between the ability to reconstruct $\utrue$ and coverage of $[0,1]^d$.
        They can be fulfilled for certain $f$, but it is rather restrictive. Furthermore, to cover the hypercube $[0,1]^d$ with $\ball^d(u([0, T]), \delta)$, where $u$ has a speed bounded by $L$, we require $T \asymgeq \delta^{-(d-1)}$. Hence, the conditions on $\delta$ and $T$ in \cref{thm:snake:lip} are tight in the sense that one variable determines the other up to a constant.
\end{remark}
\subsubsection{Proof}
\begin{lemma}\label{lmm:nnbound}
	Let $\delta > 0$.
	Let $x_0\in\ball^d(\utrue([0, T]), \delta)$. Then
	\begin{align*}
		\euclOf{\festi(x_0) - \ftrue(x_0)}
		\leq
		L_1 \left(\delta + 2\sup_{t\in[0,T]}\euclOf{\uesti(t) - \utrue(t)}\right) + \sup_{t\in[0,T]}\euclOf{\duesti(t) - \dutrue(t)}
		\eqfs
	\end{align*}
\end{lemma}
\begin{proof}
	Recall $\utrue = U(\ftrue, \uztrue, \cdot)$. Denote $\ttrue := \nn(\utrue, T, x_0)$ and $\hat\tau := \nn(\uesti, T, x_0)$.
	Using the triangle inequality and the definition of $\festi$ \eqref{eq:snake:lip:esti} and $\dutrue$, we obtain
	\begin{align*}
		\euclOf{\festi(x_0) - \ftrue(x_0)}
		&\leq
		\euclOf{\festi(x_0) - \ftrue(\utrue(\hat\tau))} + \euclOf{\ftrue(\utrue(\hat\tau)) - \ftrue(x_0)}
		\\&=\label{eq:nnbound:split}
		\euclOf{\duesti(\hat\tau) - \dutrue(\hat\tau)} + \euclOf{\ftrue(\utrue(\hat\tau)) - \ftrue(x_0)}
        \eqfs
	\end{align*}
    Let us consider the second term of the last line.
	As $\ftrue$ is $L_1$-Lipschitz, we have
	\begin{equation*}
		\euclOf{\ftrue(\utrue(\hat\tau)) - \ftrue(x_0)} \leq L_1 \euclOf{\utrue(\hat\tau) - x_0}
		\eqfs
	\end{equation*}
	We use the triangle inequality in $\euclOf{\utrue(\hat\tau) - x_0} \leq \euclOf{\utrue(\hat\tau)-\uesti(\hat\tau)} + \euclOf{\uesti(\hat\tau) - x_0}$ and bound
	\begin{align*}
		\euclOf{\uesti(\hat\tau) - x_0}
		&\leq
		\euclOf{\uesti(\ttrue) - x_0}
		\\&\leq
		\euclOf{ \uesti(\ttrue) - \utrue(\ttrue)} + \euclOf{\utrue(\ttrue) - x_0}
		\eqcm
	\end{align*}
	where we used the minimizing property of $\hat\tau$ and again the triangle inequality.
	Putting all previous bounds together, we obtain
	\begin{align*}
    	&\euclOf{\festi(x_0) - \ftrue(x_0)}
    	\\&\leq
    	\euclOf{\duesti(\hat\tau) - \dutrue(\hat\tau)}
    	+
    	L_1 \br{\euclOf{\utrue(\hat\tau)-\uesti(\hat\tau)} + \euclOf{ \uesti(\ttrue) - \utrue(\ttrue)} + \euclOf{\utrue(\ttrue) - x_0}}
    	\eqfs
	\end{align*}
	The claim follows by taking supremum over all possible values of $\ttrue,\hat\tau\in[0,T]$ and noting the definition of $x_0$ and $\ttrue$.
\end{proof}
\begin{proof}[Proof of \cref{thm:snake:lip}]
	As $\ftrue\in\bar\Sigma(1, L_0, L_1)$, we have $\utrue\in\bar\Sigma(2, \infty, L_0, \tilde L_1)$, where $\tilde L_1$ depends on $L_0$ and $L_1$, by \cref{coro:trajSmooth}.
	We want to apply \cref{cor:localpoly:supnorm} to obtain bounds on the sup-norm of the error for $\uesti$ and $\duesti$.
    \assuRef{StrictKernel} together with the uniform grid $t_i = i\stepsize$ implies \assuRef{Kernel}, \assuRef{Eigenvalue}, and \assuRef{Cover} by \cref{prp:uniformgrid}.
    \assuRef{SubGaussian} is assumed and \eqref{eq:snake:Tbound} implies \eqref{eq:lp:Tbound}. Thus, \cref{cor:localpoly:supnorm} yields
	\begin{align*}
		\EOf{\sup_{t\in[0, T]}\euclOf{\uesti(t) - \utrue(t)}^2} &\asymleq \br{\frac{T \log(n)}{n}}^{\frac{4}{5}}
		\eqcm\\
		\EOf{\sup_{t\in[0, T]}\euclOf{\duesti(t) - \dutrue(t)}^2} &\asymleq \br{\frac{T \log(n)}{n}}^{\frac{2}{5}}
		\eqfs
	\end{align*}
	Together with \cref{lmm:nnbound}, we obtain
	\begin{align*}
		\EOf{\sup_{x\in[0,1]^d}\euclOf{\festi(x) - \ftrue(x)}^2}
		&\asymleq \delta^2 + \EOf{\sup_{t\in[0, T]}\euclOf{\uesti(t) - \utrue(t)}^2}  + \EOf{\sup_{t\in[0, T]}\euclOf{\duesti(t) - \dutrue(t)}^2}
		\\&\asymleq \delta^2 + \br{\frac{T\log (n)}{n}}^{\frac{4}{5}} + \br{\frac{T \log(n)}{n}}^{\frac25}
		\\&\asymleq \delta^2 + \br{\frac{T\log (n)}{n}}^{\frac25}
		\eqfs
	\end{align*}
\end{proof}
\subsection{General}\label{ssec:snake:general}
We now want to consider estimation under higher order smoothness. We generalize the estimation procedure of section \ref{sssec:snake:lip:esti} as follows: We replace the local linear estimation of the solution $\utrue$ and its derivative $\dutrue$ by arbitrary regression estimators $\Esti^{1\to d}$ and $\dEsti^{1\to d}$, respectively. E.g., this could be local polynomial estimators of the appropriate degree. Furthermore, the nearest neighbor interpolation is replaced by multivariate polynomial interpolation of the appropriate degree. The convergence rate results are given in a black box fashion, i.e., they depend on the convergence rates of the chosen regression estimators $\Esti^{1\to d}$ and $\dEsti^{1\to d}$.
\begin{notation}\mbox{ }
    \begin{enumerate}[label=(\roman*)]
        \item
        Let $A \subset \R^d$. Denote the diameter of $A$ as
        \begin{equation*}
            \diam(A) := \sup_{a,a\pr\in A} \euclOf{a-a\pr}
            \eqfs
        \end{equation*}
        \item
        Let $A \subset \R^d$. Denote the convex hull of $A$ as
        \begin{equation*}
            \hull(A) := \setByEle{\sum_{k=1}^K w_i a_i}{K\in\N,\ a_1,\dots,a_K\in A,\ w_1,\dots,w_K\in [0, 1], \sum_{k=1}^K w_i = 1}
            \eqfs
        \end{equation*}
        \item
        Let $A \subset \R^d$, $\mu\in\Rp$. Denote the relative $\mu$-interior of the convex hull of $A$ as
        \begin{equation*}
        	\hull_\mu(A) := \setByEle{x \in \R^d}{\ball^d(x,\mu\diam(A))\subset \hull(A)}
        	\eqfs
        \end{equation*}
    \end{enumerate}
\end{notation}
\subsubsection{Model}\label{sssec:snake:general:model}
The following is a restriction of the general model of section \ref{ssec:prelim:model}.

Let $d\in\N_{\geq 2}$, $\beta\in\N_{\geq2}$, and $\indset L0\beta\subset\Rpp$.
Set $\FSmooth := \bar\Sigma^{d\to d}(\beta, \indset{L}{0}{\beta})$, see \cref{def:Hoelder}.
Let $\ftrue\in\FSmooth$. Set $m=1$. Let $x_1 \in \R^d$, $n = n_1\in\N$, and $T := T_1 \in\Rpp$.
Set $t_{1,i} := t_i\in[0, T]$ for $i\in\nnzset{n}$ with $0 = t_0 \leq \dots \leq t_n = T$.
Let $\sigma\in\Rp$.
Let $\epsilon_i := \epsilon_{1,i}$, $i\in\nnset{n}$ be independent $\R^d$-valued random variables such that $\Eof{\epsilon_i} = 0$ and $\Eof{\euclof{\epsilon_i}^2} \leq \sigma^2$.
Set
\begin{equation*}
    Y_i := Y_{1,i} := U(\ftrue, \uztrue, t_i) + \epsilon_i \qquad \text{for } i\in\nnset{n}
    \eqfs
\end{equation*}
We observe $Y_i$ and know $x_1$ and $t_i$, but $\ftrue$ is unknown and to be estimated.
We assume $d$, $\indset{L}{0}{\beta}$, and $\sigma$ to be fixed. For an estimator $\festi$, we are interested in asymptotic upper bounds for the sup-norm of $\ftrue - \festi$ in the domain of interest $[0,1]^d$ in probability depending on the asymptotics of $n$ and $T$.
\begin{remark}\mbox{ }
    \begin{enumerate}[label=(\roman*)]
        \item
            It does not matter whether $x_1$ is known or not: Below, we derive upper bounds on the estimation error for the potentially more difficult problem of unknown $x_1$. The lower bounds from \cite[Corollary 4.11]{lowerbounds} are formulated with known $x_1$, which are then also lower bounds for the case of unknown $x_1$.
        \item
            This model and following estimation procedure can easily be generalized to larger $m$ as long as $\min_{j\in\nnset m} n_j \asymgt 1$ fast enough and the $t_{j,i}$ behave in an appropriate way.
   \end{enumerate}
\end{remark}
\subsubsection{Multivariate Polynomial Interpolation}\label{sssec:snake:general:poly}
The estimation strategy for the model described in section \ref{sssec:snake:general:model} uses multivariate polynomial interpolation. In this section, we introduce some basics of polynomial interpolation and some further objects that are required to describe the ODE estimator. Recall \cref{not:poly} on polynomials.
\begin{definition}[Polynomial interpolation]\label{def:poly}
    Let $\ell,\dm x,\dm y\in\N$.
    \begin{enumerate}[label=(\roman*)]
        \item
        Set
        \begin{equation*}
            N := N_{{\dm x},\ell} := \dim(\Poly{\dm x}\ell) = \begin{pmatrix} \ell + {\dm x}\\{\dm x}\end{pmatrix}
            \eqfs
        \end{equation*}
        \item
        For $x\in\R^{\dm x}$, denote the vector of monomials of degree at most $\ell$ in $\dm x$ dimensions as
        \begin{equation*}
            \psi(x) := \psi_{\ell}(x) := (x^\alpha)_{|\alpha|\leq \ell} \in \R^N
            \eqfs
        \end{equation*}
        For $M\in\N$, $\mo x = (x_1, \dots, x_M)\in(\R^{\dm x})^M$, denote $\Psi(\mo x) := \Psi_{\ell}(\mo x) := \br{\psi(x_1), \dots, \psi(x_M)}\tr \in \R^{M\times N}$.
        \item
        For $x\in\R^{\dm x}$, $\mo x\in(\R^{\dm x})^N$ with $\Psi(\mo x) \in \R^{N \times N}$, $\mo y\in(\R^{\dm y})^N = \R^{N \times {\dm y}}$, denote the \textit{componentwise polynomial interpolation} as
        \begin{equation*}
            I(\mo x, \mo y, \cdot ) := I_\ell(\mo x, \mo y, \cdot ) \colon \R^{\dm x} \to \R^{\dm y},\, x \mapsto
            \begin{cases}
                \psi(x)\tr\Psi(\mo x)^{-1}\mo y &\text{ if } \Psi(\mo x) \text{ is invertible,}\\
                0 &\text{ otherwise.}
            \end{cases}
        \end{equation*}
    \end{enumerate}
\end{definition}
\begin{remark}\mbox{ }
    \begin{enumerate}[label=(\roman*)]
        \item
        Let $\mo x = (x_1, \dots, x_N) \in (\R^{\dm x})^N$ and $\mo y = (y_1, \dots, y_N) \in (\R^{\dm y})^N$.
        The function $I(\mo x, \mo y, \cdot )$ is an \textit{interpolation} as $I(\mo x, \mo y, x_k) = y_k$. It is \textit{polynomial} as $x\mapsto \Pi_k I(\mo x, \mo y, x) \in \Poly{\dm x}{\ell}$ for each $k\in\nnset{\dm y}$. It is \textit{componentwise} as
        \begin{equation*}
            I(\mo x, \mo y, x) = \br{I\brOf{\mo x, \br{\Pi_k y_1, \dots \Pi_k y_N}, x}}_{k\in\nnset{\dm y}}
            \eqfs
        \end{equation*}
        \item The definition of $I(\mo x, \mo y, \cdot ) = 0 $ if $\Psi(\mo x)$ is not invertible is for convenience only. The value 0 is not of importance.
        \item
        If $\dm x = 1$, the matrix $\Psi(\mo x)\in\R^{N\times N}$ is also called \textit{Vandermonde matrix}. It is invertible if and only if all $x_k\in\R$ are distinct. In the multivariate case, conditions for invertibility of $\Psi(\mo x)$ (and hence, unique existence of the polynomial interpolation) are more complex, see \cite[Theorem 4.1]{Olver06}.
    \end{enumerate}
\end{remark}
The geometric configuration of the base points $\mo x$ for the multivariate polynomial interpolation influences the approximation quality. We next define a normalization that removes location and scale from $\mo x$ to isolate the geometric configuration. 
\begin{definition}[Normalization]\label{def:normalization}
    Let $d,N\in\N$.
    For $\mo x = (x_1, \dots, x_N) \in(\R^d)^N$ with $\diam(\mo x) >0$, let
    \begin{equation*}
        \eta_{\mo x}\colon \R^d \to \R^d\eqcm
        x \mapsto \frac{x-\frac{1}{N}\sum_{k=1}^{N}x_k}{\diam(\mo x)}
        \eqfs
    \end{equation*}
    Furthermore, denote  $\eta_{\mo x}(\tilde{\mo x}) = (\eta_{\mo x}(\tilde x_1), \dots, \eta_{\mo x}(\tilde x_M))$ for $\tilde{\mo x} = (\tilde x_1, \dots, \tilde x_M) \in(\R^d)^M$, $M\in\N$.
\end{definition}
Next, we define the set $\mc H$ of sets of base points $\mo x$ that can be used in a polynomial interpolation to suitably approximate the function value at $x$, the set $\mc G$ of points $x$ with $\mc H \neq \emptyset$, and an (almost) optimal choice $\chi$ of $\mo x \in \mc H$.
\begin{definition}\label{def:snake:general:sets}
    Let $d,\ell\in\N$, $N = \dim(\Poly d\ell)$.
    Let $\mc X \subset \R^d$ and $s,\delta\in\Rpp$.
    \begin{enumerate}[label=(\roman*)]
        \item
        Let $\mu\in\Rp$.
        Define $\mc H_{\ell, s, \mu}(\mc X, x)$ as the set of all sets $\mo x\in \mc X^N$ such that $x \in \hull_\mu(\mo x)$ and $\opNormOf{\Psi(\eta_{\mo x}(\mo x))^{-1}} \leq s$.
        \item
        Fix a constant $D \in (1,\infty)$.
        Fix a constant $\mu \in (0, 1/N)$.
        Define $\mc G_{\ell,\delta,s}(\mc X)$ as the set of all points $x \in \R^d$ with following property: There is a set of points $\mo x\in \mc H_{\ell, s, \mu}(\mc X, x)$ with $D^{-1}\delta \leq \diam(\mo x) \leq D\delta$.
        \item
        Let $\chi_{\ell, s}(\mc X, x)$ be an element $\mo x_0 \in \mc H_{\ell, 2s, 0}(\mc X, x)$ such that
        \begin{equation*}
            \opNormOf{\Psi(\eta_{\mo x_0}(\mo x_0))^{-1}} \diam(\mo x_0)^{\ell+1} \leq 2 \min_{\mo x \in \mc H_{\ell, 2s, 0}(\mc X, x)}\opNormOf{\Psi(\eta_{\mo x}(\mo x))^{-1}} \diam(\mo x)^{\ell+1}
            \eqfs
        \end{equation*}
    \end{enumerate}
\end{definition}
\begin{remark}
	The parameters $\mu,s,D$ are fixed and do not change with $n$. They ensure the stability of the multivariate polynomial interpolation and are used in some technical arguments of the main proof.
	\begin{enumerate}[label=(\roman*)]
		\item For the polynomial interpolation, we have to solve a system of linear equations given by the matrix $\Psi(\mo x)$. The requirements $\opNormOf{\Psi(\eta_{\mo x}(\mo x))^{-1}} \leq s$ and $\diam(\mo x) \leq D\delta$ ensure stability and precision of the obtained approximation.
		\item 
		The lower bound on the diameter of the interpolation base points, $D^{-1}\delta \leq \diam(\mo x)$, ensures the stability of the normalization $\eta_{\mo x}$, where we divide by $\diam(\mo x)$.
		\item 
		Multivariate polynomial approximation results are known for points $x_0\in\R^d$ in the convex hull of the base points $\mo x\in (\R^d)^N$, see \cref{lmm:polyinterpol}. To ensure that $x_0$ belongs also to the convex hull of sufficiently accurate estimates of the base points $\hat{\mo x}$, we employ the concept of the $\mu$-interior of the convex hull: Specifically, $x_0\in \hull_\mu(\mo x)$ with $\mu>0$ implies $x_0\in \hull(\hat{\mo x})$ if the estimation error is small enough relative $\diam(\mo x)$.
	\end{enumerate}
\end{remark}
\subsubsection{Estimator}\label{sssec:snake:general:esti}
Let $L\in\Rpp$.
Let $\Esti^{1\to d}$ and $\dEsti^{1\to d}$ be arbitrary regression estimators for the regression function and its derivative in the standard regression model $\mf P := \mf P^{1 \to d}(\Sigma^{1\to d}(\beta+1, L), T, (t_i)_{i\in\nnset n}, \sigma)$, see section \ref{ssec:prelim:regression}.
With these estimators, we estimate the observed ODE-solution and its derivative,
\begin{equation*}
    \uesti(t) :=  \Esti^{1\to d}\brOf{(t_i, Y_i)_{i\in\nnset{n}}, t}
    \qquad\text{and}\qquad
    \duesti(t) := \dEsti^{1\to d}\brOf{(t_i, Y_i)_{i\in\nnset{n}}, t}
    \eqfs
\end{equation*}
Denote the maximal sup-norm risk in probability of the regression estimators, see \eqref{eq:reg:supnprmrisk}, as
\begin{align*}
    \rate := \rate(n, \beta+1,  L, T) &:= r^{\ms{pr}, \ms{sup}}_0(\Esti^{1\to d}, \mf P)
    \eqcm\\
    \drate := \drate(n, \beta+1,  L, T) &:= r^{\ms{pr}, \ms{sup}}_1(\dEsti^{1\to d}, \mf P)
    \eqfs
\end{align*}
Set $\ell := \beta-1$.
Fix $s \in\Rpp$. Let $x\in\R^d$. Assume $\chi_{\ell, s}(\uesti([0, T]), x)$ exists.
Let $\hat{\mo x}(x) := (\uesti(\hat\tau_1), \dots, \uesti(\hat\tau_N))$, where $\hat\tau_i$ is chosen so that $\mo{\hat x}(x) = \chi_{\ell, s}(\uesti([0, T]), x)$.
Let $\hat{\mo y}(x) := (\duesti(\hat\tau_1), \dots, \duesti(\hat\tau_N))$. Define the estimator
\begin{equation}\label{eq:snake:general:esti}
	\festi(x) := I_\ell(\mo{\hat x}(x), \hat{\mo y}(x), x)
	\eqfs
\end{equation}
\begin{remark}\mbox{ }\label{rem:snake:general:estimator}
    The estimator is only defined for $x\in\mc G_{\ell,\infty,s}(\uesti([0, T]))$, which is a subset of the convex hull of $\uesti([0, T])$. If it is not defined, a fallback such as nearest neighbor (see section \ref{ssec:snake:lipschitz}) could be applied in practice.
\end{remark}
\subsubsection{Results}
For $\mc X \subset \R^d$ and a function $u\colon[0, T]\to \R^d$, define
\begin{equation*}
    \delta_{\ms{NN}}^{\ell, s}(u, T, \mc X) := \inf\cbOf{\delta\pr \geq 0\colon \mc X \subset \mc G_{\ell,\delta\pr,s}(u([0,T]))}
    \eqfs
\end{equation*}
\begin{theorem}\label{thm:snake:general}
    Use the model of section \ref{sssec:snake:general:model} and the estimator of section \ref{sssec:snake:general:esti}.
	Let $\delta \in\Rp$, potentially changing with $n$.
    Assume $T \asymlt n$, $\rate \asymlt 1$, $\rate \asymlt \delta$, $\rate \asymleq \drate \asymleq 1$.
    Set $L : = \beta!\sup_{k\in\nnset\beta}{L_k}^{\beta+1}$.
    \begin{enumerate}[label=(\roman*)]
        \item
        Set $\mc X := \mc G_{\ell,\delta,s}(\utrue([0,T]))$.
        Then
        \begin{equation*}
            \sup_{x\in \mc X}\euclOf{\festi(x) - \ftrue(x)}^2\in \Op\brOf{\delta^{2\beta} + \drate(n, \beta+1,  L, T)^2}
            \eqfs
        \end{equation*}
        \item
        Set $\delta := \delta_{\ms{NN}}^{\ell, s}(\utrue, T, [0,1]^d)$.
        Then
        \begin{equation*}
            \sup_{x\in [0,1]^d}\euclOf{\festi(x) - \ftrue(x)}^2\in \Op\brOf{\delta^{2\beta} + \drate(n, \beta+1,  L, T)^2}
            \eqfs
        \end{equation*}
    \end{enumerate}
\end{theorem}
We want to use the componentwise local polynomial estimators of appendix \ref{app:sec:localreg} as the estimators $\Esti^{1\to d}$ and $\dEsti^{1\to d}$ (appendix \ref{app:sec:localreg}: $d=1$, $\ell = \beta$, and $s = 0$ and $s = 1$, respectively). We assume \assuRef{Kernel}, \assuRef{Eigenvalue}, and \assuRef{Cover} for the observation times $t_i$ and \assuRef{SubGaussian} for the noise $\noise_i$. Furthermore, we require \eqref{eq:lp:Tbound}. Then, assuming $L \asymeq 1$,  \cref{cor:localpoly:supnorm} yields
\begin{align*}
    \rate(n, \beta+1,  L, T) &\asymleq \br{\frac{T\log n}{n}}^{\frac{\beta+1}{2(\beta+1)+1}}\eqcm\\
    \drate(n, \beta+1,  L, T) &\asymleq \br{\frac{T\log n}{n}}^{\frac{\beta}{2(\beta+1)+1}}
    \eqfs
\end{align*}
\begin{corollary}\label{cor:snake:general}
   Use the model of section \ref{sssec:snake:general:model} and the estimator of section \ref{sssec:snake:general:esti} with the componentwise local polynomial estimator of degree $\beta$ (see appendix \ref{app:sec:localreg}) as $\Esti^{1\to d}$ and $\dEsti^{1\to d}$.
   Let $\delta \in\Rp$, potentially changing with $n$, and assume
   \begin{equation*}
       \br{\frac{T\log n}{n}}^{\frac{\beta+1}{2(\beta+1)+1}} \asymlt \delta
       \eqfs
   \end{equation*}
   Assume \assuRef{Kernel}, \assuRef{Eigenvalue}, \assuRef{Cover}, where the symbols $d$ and $x_i$ in the assumptions are $d=1$ and $\indset x1n  = \indset t1n$. Assume \assuRef{SubGaussian} and
   \begin{equation*}
       \br{\frac{\log(n)}{n}}^{\frac{1}{2(\beta+1)}} \asymleq T \asymlt n
       \eqfs
   \end{equation*}
    \begin{enumerate}[label=(\roman*)]
        \item\label{cor:snake:general:ball}
        Set $\mc X := \mc G_{\ell,\delta,s}(\utrue([0,T]))$.
        Then
        \begin{equation*}
            \sup_{x\in \mc X}\euclOf{\festi(x) - \ftrue(x)}^2 \in
            \Op\brOf{\delta^{2\beta} + \br{\frac {T \log n}n}^{\frac{2\beta}{2(\beta+1)+1}}}
            \eqfs
        \end{equation*}
        \item\label{cor:snake:general:cube}
        Set $\delta := \delta_{\ms{NN}}^{\ell, s}(\utrue, T, [0,1]^d)$.
        Then
        \begin{equation}\label{eq:snake:general:cor}
            \sup_{x\in [0,1]^d}\euclOf{\festi(x) - \ftrue(x)}^2 \in
            \Op\brOf{\delta^{2\beta} + \br{\frac {T \log n}n}^{\frac{2\beta}{2(\beta+1)+1}}}
            \eqfs
        \end{equation}
    \end{enumerate}
\end{corollary}
\begin{remark}\mbox{ }
    \begin{enumerate}[label=(\roman*)]
        \item From \cite[Corollary 4.11]{lowerbounds}, we obtain
        \begin{equation*}
            \EOf{\sup_{x\in [0,1]^d}\euclOf{\festi(x) - \ftrue(x)}^2} \asymgeq \delta^{2\beta} + \br{\delta^{d-1}\frac{T \log(n)}{n}}^{\frac{2\beta}{2(\beta+1) + d}}\eqfs
        \end{equation*}
        if $\log (n) \asymeq \log(\delta^{-(d-1)}n T^{-1})$ and
        \begin{equation*}
            \max\brOf{\br{\frac{T}{n}}^{2\beta+d+1}, T^{-1}} \asymleq \delta^{d-1} \asymleq \min\brOf{1, \frac{n}{T}}
            \eqfs
        \end{equation*}
        Thus, the rate given in \eqref{eq:snake:general:cor} is minimax optimal if $\delta^{2\beta}$ is the dominating term the error bound, i.e., if
        \begin{equation*}
            \delta \asymgeq \br{\frac {T \log n}n}^{\frac{1}{2(\beta+1)+1}}
            \eqfs
        \end{equation*}
        This condition is fulfilled in the setting of \cref{cor:snake:general:fast} below.
        See \cref{rem:snake:results:lip} for a discussion of the condition for $\beta =1$.
        \item
        In section \ref{ssec:snake:lipschitz}, we can extrapolate beyond the convex hull of $\uesti([0,T])$ in the case of $\beta=1$ using nearest neighbors. In contrast, our general results are only available for interpolation, i.e., for $x \in \hull(\uesti([0,T]))$ --- to be precise, only for the even more restrictive assumption $x\in\mc G_{\ell,\delta,s}(\utrue([0,T]))$. Extrapolation seems also possible using polynomials of degree $>0$, but the technicalities seem more difficult.
        \item
        In practice, finding $\chi_{\ell, s}$ may be computationally demanding. Furthermore, using only results of interpolation may be inconveniently restrictive. Thus, one may want to replace the polynomial interpolation step by a polynomial regression step on the $k \geq N$ nearest neighbors of $\uesti(\tau_j)$ for some chosen time points $\tau_1, \dots, \tau_J \in [0,T]$. Another alternative to polynomials would be Gaussian process interpolation (or regression) \cite{Rasmussen2006Gaussian}, potentially also restricted to $k$ nearest neighbors for performance reasons.
    \end{enumerate}
\end{remark}
The lowest upper bound on the error is obtained if the two terms in \eqref{eq:snake:general:cor} are balanced. Additionally, we require $T \delta^{d-1} \asymgeq 1$ to be able to cover the domain of interest with $\ball^d(U(\ftrue, x_1, [0, T]), \delta)$, which is necessary for $[0,1]^d \subset \mc G_{\ell,\delta,s}(\utrue([0,T]))$.
\begin{corollary}\label{cor:snake:general:fast}
    Use the setting and assumptions of \cref{cor:snake:general}.
    Set $\delta := \delta_{\ms{NN}}^{\ell, s}(\utrue, T, [0,1]^d)$.
    Assume $T \asymleq \br{\frac{n}{\log n}}^{\frac{d-1}{2(\beta+1)+d}}$ and $\delta \asymleq \br{\frac{n}{\log n}}^{-\frac{1}{2(\beta+1)+d}}$.
    Then
    \begin{equation*}
    	\sup_{x\in[0,1]^d}\euclOf{\festi(x) - \ftrue(x)}^2 \in \Op\brOf{
    	\br{\frac n{\log n}}^{-\frac{2\beta}{2(\beta+1)+d}}}
    	\eqfs
    \end{equation*}
\end{corollary}
\subsubsection{Proof}
\begin{lemma}\label{lmm:snake:general:exist}
	Let $\delta, s \in \Rpp$.
	Let
	\begin{equation*}
		\gamma := \sup_{t\in[0, T]}\euclOf{\uesti(t) - \utrue(t)}
		\eqfs
	\end{equation*}
	Assume $\gamma \asymlt \min(1, \delta)$.
	Then $\festi(x)$ eventually (meaning for $n$ large enough) exists for all $x \in \mc G_{\beta,\delta, s}(\utrue([0, T]))$.
\end{lemma}
\begin{proof}
	To show that $\festi(x)$, as defined in \eqref{eq:snake:general:esti}, eventually exists, we need to prove that $\mc H_{\ell, 2s, 0}(\uesti([0, T]), x)$ is eventually nonempty.
	As $x\in \mc G_{\beta,\delta, s}(\utrue([0, T]))$, there is $\mo x\in \utrue([0,T])^N$ with $x \in \hull_\mu(\mo x)$ such that $\opNormOf{\Psi(\eta_{\mo x}(\mo x))^{-1}} \leq s$. As $\gamma \asymlt \min(1, \delta)$ and $\mu > 0$, there must eventually be a $\hat{\mo x}\in \uesti([0,T])^N$ with $x \in \hull(\hat{\mo x})$ and $\opNormOf{\Psi(\eta_{\hat{\mo x}}(\hat{\mo x}))^{-1}} \leq 2s$. Thus, $\hat{\mo x}(x) = \chi_{\ell, s}(\uesti([0, T]), x) \in \mc H_{\ell, 2s, 0}(\mc X, x)$ exists and so does $\festi(x)$ for all $x \in \mc G_{\beta,\delta, s}(\utrue([0, T]))$.
\end{proof}
\begin{lemma}\label{lmm:snake:general:prob}
    Let $\delta, s > 0$.
	Let $x_0\in \mc G_{\beta, \delta, s}(\utrue([0, T]))$.
	Let
	\begin{equation}\label{eq:snake:general:prob:uesti}
		\gamma := \sup_{t\in[0, T]}\euclOf{\uesti(t) - \utrue(t)}\eqcm \qquad
		\lambda := \sup_{t\in[0, T]}\euclOf{\duesti(t) - \dutrue(t)}\eqfs
	\end{equation}
    Assume that $\festi(x_0)$ exists.
    Assume 
    \begin{equation}\label{eq:snake:general:prob:cond}
       	\max\brOf{4, 16 N \beta s} \gamma \leq D^{-1} \delta
    \end{equation}
    with $D$ from \cref{def:snake:general:sets}.
    Then
    \begin{equation*}
       	\euclOf{\festi(x_0) - \ftrue(x_0)}
       	\leq
       	c_{\beta,d}\br{ s \br{\lambda + L_1\gamma} + 
       		L_\beta \br{s + \frac{s^2\gamma}{D^{-1} \delta}} \br{D \delta  + \gamma}^\beta}
       	\eqfs
    \end{equation*}
    where $c_{\beta,d} \in \Rpp$ is a constant depending only $\beta, d$.
\end{lemma}
\begin{proof}
	Let us establish some notation.
	Let $\delta_{\ms{min}} := D^{-1} \delta$ and $\delta_{\ms{max}} := D \delta$ with $D$ from \cref{def:snake:general:sets}.
	Let $\hat{\mo x} := \hat{\mo x}(x_0)$, Recall that $\hat{\mo x} = (\uesti(\hat\tau_1), \dots, \uesti(\hat\tau_N)) = \chi_{\ell,s}(\uesti([0, T]), x_0)$. Let $(\utrue(\tau_1), \dots, \utrue(\tau_N)) = \mo x$ be any set such that $x_0\in\hull_\mu(\mo x)$,  $\opNormof{\Psi({\mo x})^{-1}} \leq s$, and $\delta_{\ms{min}} \leq \diam(\mo x) \leq \delta_{\ms{max}}$, which exists as $x_0\in \mc G_{\beta,\delta, s}(\utrue([0, T]))$. We summarize and add some additional notation:
    \begin{align*}
        \mo x &= (\utrue(\tau_1), \dots, \utrue(\tau_N)) \eqcm & \mo y &:= (\dutrue(\tau_1), \dots, \dutrue(\tau_N))\eqcm\\
        \bar{\mo x} & := (\utrue(\hat\tau_1), \dots, \utrue(\hat\tau_N))\eqcm & \bar{\mo y} &:= (\dutrue(\hat\tau_1), \dots, \dutrue(\hat\tau_N))\eqcm\\
        \tilde{\mo x} & := (\uesti(\tau_1), \dots, \uesti(\tau_N))\eqcm & \check{\mo y} &:= (\ftrue(\uesti(\hat\tau_1)), \dots, \ftrue(\uesti(\hat\tau_N)))\eqcm\\
        \hat{\mo x} &= (\uesti(\hat\tau_1), \dots, \uesti(\hat\tau_N)) \eqcm & \hat{\mo y} &:= (\duesti(\hat\tau_1), \dots, \duesti(\hat\tau_N))\eqfs
    \end{align*}
    In the following, $c_{\beta,d}$ will always refer to a constant in $\Rpp$, which depends only on $\beta,d$ and may have a different value each time it occurs.

	Using the definition of $\festi$, \eqref{eq:snake:general:esti}, and the triangle inequality, we can split the error into two parts,
	\begin{equation}\label{eq:snake:split}
		\euclOf{\festi(x_0) - \ftrue(x_0)}
		\leq
		\euclOf{I(\mo{\hat x}, \hat{\mo y}, x_0) - I(\mo{\hat x}, \check{\mo y}, x_0)} + \euclOf{I(\mo{\hat x}, \check{\mo y}, x_0) - \ftrue(x_0)}
		\eqfs
	\end{equation}

    For the first term in \eqref{eq:snake:split}, we apply \cref{lmm:poly:linear} and get
    \begin{align*}
        \euclOf{I(\mo{\hat x}, \hat{\mo y}, x_0) - I(\mo{\hat x}, \check{\mo y}, x_0)}
        &=
        \euclOf{\psi(x_0)\tr \Psi(\mo{\hat x})^{-1} \br{\hat{\mo y} - \check{\mo y}}}
        \\&=
        \euclOf{\psi(\eta_{\mo{\hat x}}(x_0))\tr \Psi(\eta_{\mo{\hat x}}(\mo{\hat x}))^{-1} \br{\hat{\mo y} - \check{\mo y}}}
        \\&\leq
        \euclOf{\psi(\eta_{\mo{\hat x}}(x_0))} \opNormOf{\Psi(\eta_{\mo{\hat x}}(\mo{\hat x}))^{-1}} \opNormOf{\hat{\mo y} - \check{\mo y}}
        \eqfs
    \end{align*}
    To find an upper bound for the first two factors of the last term, we use $\euclof{\eta_{\mo{\hat x}}(x_0)} \leq 1$ and $\mo{\hat x} \in \mc H_{\ell, 2s, 0}(\uesti([0, T]), x_0)$ to obtain
    \begin{equation*}
    	\euclOf{\psi(\eta_{\mo{\hat x}}(x_0))} \opNormOf{\Psi(\eta_{\mo{\hat x}}(\mo{\hat x}))^{-1}}  \leq c_{\beta,d} s
    	\eqfs
    \end{equation*}
    For the third factor, $\opNormOf{\hat{\mo y} - \check{\mo y}} \leq \opNormOf{\hat{\mo y} - \bar{\mo y}} + \opNormOf{\bar{\mo y} - \check{\mo y}}$ by the triangle inequality. By the definition of $\lambda$, $\opNormOf{\hat{\mo y} - \bar{\mo y}} \leq \sqrt{N} \lambda$. Using $L_1$-Lipschitz continuity of $\ftrue$ and the definition of $\gamma$, we obtain
    \begin{align*}
    	\opNormOf{\bar{\mo y} - \check{\mo y}}
    	&\leq
    	\opNormOf{\Big(\ftrue(\utrue(\hat\tau_1)) - \ftrue(\uesti(\hat\tau_1)), \dots,  \ftrue(\utrue(\hat\tau_N)) - \ftrue(\uesti(\hat\tau_N))\Big)}
    	\\&\leq
    	L_1 \sqrt{N} \max_{k\in\nnset N} \euclOf{\utrue(\hat\tau_k) - \uesti(\hat\tau_k)}
    	\\&\leq
    	L_1 \sqrt{N} \gamma
    	\eqfs
    \end{align*}
    Note that $N$ is a constant depending only on $\beta, d$.
    The considerations above yield the following upper bound for the first term in \eqref{eq:snake:split}:
    \begin{equation}\label{eq:snake:termonebound}
        \euclOf{I(\mo{\hat x}, \hat{\mo y}, x_0) - I(\mo{\hat x}, \bar{\mo y}, x_0)} \leq c_{\beta,d} s \br{\lambda + L_1 \gamma}
        \eqfs
    \end{equation}

    For the second term in \eqref{eq:snake:split}, we apply \cref{lmm:polyinterpol} and get
    \begin{equation}\label{eq:snake:termtwo}
        \euclOf{I(\mo{\hat x}, \check{\mo y}, x_0) - \ftrue(x_0)} \leq L_\beta C_\ell(\mo{\hat x}) \diam(\mo{\hat x})^\beta
        \eqfs
    \end{equation}
    With \cref{lmm:snake:poly:constbound} and the minimizing property of $\mo{\hat x} = \chi_{\ell,s}(\uesti([0, T]), x_0)$, we obtain
    \begin{align*}
        C_\ell(\mo{\hat x}) \diam(\mo{\hat x})^\beta
        &\leq
        \frac{N^{\frac32}}{\beta!} \opNormOf{\Psi(\eta_{\hat{\mo x}}(\mo{\hat x}))^{-1}} \diam(\mo{\hat x})^\beta
        \\&\leq
        2\frac{N^{\frac32}}{\beta!} \opNormOf{\Psi(\eta_{\tilde{\mo x}}(\tilde{\mo x}))^{-1}} \diam(\tilde{\mo x})^\beta
        \eqfs
    \end{align*}
    With $\euclof{\uesti(\tau_j) - \utrue(\tau_j)} \leq \gamma$ and the definition of $\mo x$, we have
    \begin{equation}\label{eq:snake:termtwo:b}
        \diam(\tilde{\mo x}) \leq \diam(\mo{x})  + 2 \gamma  \leq \delta_{\ms{max}}  + 2\gamma
        \eqfs
    \end{equation}
    By the triangle inequality
    \begin{equation}\label{eq:snake:termtwo:triangle}
    	\opNormOf{\Psi(\eta_{\tilde{\mo x}}(\tilde{\mo x}))^{-1}}
    	\leq
    	\opNormOf{\Psi(\eta_{{\mo x}}({\mo x}))^{-1}} + \opNormOf{\Psi(\eta_{\tilde{\mo x}}(\tilde{\mo x}))^{-1} - \Psi(\eta_{{\mo x}}(\mo x))^{-1}}
    	\eqfs
    \end{equation}
    As we assume
    \begin{equation*}
    	\max\brOf{4, 16 N \beta s} \gamma \leq \delta_{\ms{min}}
    	\eqcm
    \end{equation*}
    we can use \cref{lmm:perturb} to obtain from \eqref{eq:snake:termtwo:triangle} that
    \begin{equation}\label{eq:snake:termtwo:c}
        \opNormOf{\Psi(\eta_{\tilde{\mo x}}(\tilde{\mo x}))^{-1}}
        \leq
        \opNormOf{\Psi(\eta_{{\mo x}}({\mo x}))^{-1}} + c_{\beta,d} \frac{s^2\gamma}{\delta_{\ms{min}}}
        \eqfs
    \end{equation}

    Putting the previous inequalities into \eqref{eq:snake:termtwo}, we obtain the following bound on the second term of \eqref{eq:snake:split}:
    \begin{equation}\label{eq:snake:termtwobound}
        \euclOf{I(\mo{\hat x}, \check{\mo y}, x) - \ftrue(x)} 
        \leq
        2 L_\beta \frac{N^{\frac32}}{\beta!} \br{\opNormOf{\Psi(\eta_{{\mo x}}({\mo x}))^{-1}} + c_{\beta,d} \frac{s^2\gamma}{\delta_{\ms{min}}}} \br{\delta_{\ms{max}}  + 2\gamma}^\beta
        \eqfs
    \end{equation}

    The bounds \eqref{eq:snake:termonebound} and \eqref{eq:snake:termtwobound} on the two terms of the right-hand side of \eqref{eq:snake:split} together yield
    \begin{equation*}
        \euclOf{\festi(x_0) - \ftrue(x_0)}
        \leq
        c_{\beta,d}\br{ s \br{\lambda + L_1\gamma} + 
        L_\beta \br{s + \frac{s^2\gamma}{\delta_{\ms{min}}}} \br{\delta_{\ms{max}}  + \gamma}^\beta}
        \eqfs
    \end{equation*}
\end{proof}
\begin{proof}[Proof of \cref{thm:snake:general}]
	Set $\tilde\beta = \beta + 1$, $\tilde L_0 = \infty$ and $\tilde L_k = (k-1)! \sup_{\ell\in\nnzset{\beta}}L_\ell^k$ for $k\in\nnsetof{\tilde\beta}$.
    As $f\in\bar\Sigma^{d\to d}(\beta, \indset{L}{0}{\beta})$, \cref{coro:trajSmooth} yields $u \in \bar\Sigma^{1\to d}(\beta+1, \indset{\tilde L}{0}{\tilde\beta})$.
    Hence, using the definitions of the random variables $\lambda$ and $\gamma$ in \cref{lmm:snake:general:prob} as well as of the numbers $\rate$ and $\drate$ in section \ref{sssec:snake:general:esti}, we obtain
	\begin{equation}\label{eq:thm:snake:general:lambdagamma}
        \gamma \in \Op\brOf{\rate(n, \tilde\beta, \tilde L_{\tilde\beta}, T)}
        \qquad\text{and}\qquad
        \lambda \in \Op\brOf{\drate(n, \tilde\beta, \tilde L_{\tilde\beta}, T)}
        \eqfs
    \end{equation}
    As $\Gamma \asymlt 1$, $\festi(x)$ eventually exists for all $x\in\mc X = \mc G_{\beta, \delta, s}(\utrue([0, T]))$ by \cref{lmm:snake:general:exist}.
    The assumption of the theorem, $\rate \asymlt \delta$, implies that the condition of \cref{lmm:snake:general:prob}, \eqref{eq:snake:general:prob:cond}, is fulfilled with probability arbitrarily close to $1$.
    As $\gamma$ and $\lambda$ do not depend on $x$, \cref{lmm:snake:general:prob} therefore yields
    \begin{align*}
        \sup_{x\in \mc X}\euclOf{\festi(x) - \ftrue(x)}^2
        &\leq 
        c_{\beta,d}\br{ s \br{\lambda + L_1\gamma} + 
        	L_\beta \br{s + \frac{s^2\gamma}{D^{-1} \delta}} \br{D \delta  + \gamma}^\beta}^2
        \\&\in
        \Op\brOf{\drate(n, \tilde\beta,  \tilde L_{\tilde\beta}, T)^2 + \delta^{2\beta}}
        \eqcm
    \end{align*}
    where we used \eqref{eq:thm:snake:general:lambdagamma} and $\rate \asymlt \delta$, $\Gamma \asymleq \Lambda$ as well as $s, D, L_1, L_\beta$ being fixed.	
\end{proof}
\section{Extensions}\label{sec:extension}
This section explores simple extensions of error bounds proven in this article and highlights open questions in settings where further research is needed to achieve comprehensive results.

So far, we have shown upper bounds on the error for estimating the model function of an autonomous, first order ODE in two different settings:
In the Stubble model of $m \asymeq n$ short ($n_j \asymeq 1$) trajectories with equidistant measurement times, we obtained a bound on the point-wise mean squared error, see \cref{cor:stubble:general}. 
In the Snake model of a single trajectory ($m=1$) with $n_1 \asymeq n$ measurements with a flexible time step design and sub-Gaussian noise, we obtained a bound on the sup-norm error in probability, see \cref{cor:snake:general}.

Both results are available in a black box form (\cref{thm:stubble:general} and \cref{thm:snake:general}, respectively), where arbitrary nonparametric regression estimators can be plugged in. This makes these results quite powerful. E.g., we can apply the black box results with adaptive nonparametric regression estimators to obtain adaptive ODE estimators. In this manner, questions regarding hyperparameter optimization can be answered in the same way as for the regression problem. Only the multivariate polynomial interpolation in the general case of the Snake model with $\beta \geq 2$ requires additional hyperparameters ($s$, $D$) regarding the stability of the interpolation. But note that the major purpose of the proposed estimator in this setting is the mathematical derivation of upper bounds on the estimation error (which are minimax optimal in some settings) rather than demonstrating a practical algorithm. In particular, computational requirements may be unfeasible. In contrast, in the Snake model with $\beta=1$ and in the Stubble model, the proposed estimators do not suffer from such obstacles for application.

It is straight forward to extend the error bounds in the Stubble model to integrated and sup-norm loss (the latter typically requires sub-Gaussian noise and adds a factor that is polynomial in $\log(n)$ to the variance part of the error). Furthermore, the bound in expectation implies the bound in probability, so that we can obtain results that are weaker but more directly comparable to those of the Snake model. 
An extension of the Snake model results to point-wise or integrated errors (without the log-factor) seems more difficult as our proof already requires a bound on the maximal error of the initial regressions, independent of the target error measure.

Further trivial extensions include: going from $m = 1$ to $m \asymeq 1$ in the Snake model (even some slowly growing $m$ seems possible without much difficulty); non-equidistant timesteps in the Stubble model given that they are equal between different trajectories and asymptotically equal; treating non-autonomous and higher order ODEs as autonomous first order ODEs, see section \ref{ssec:prelim:ode}. 

Finally, an open question that is more difficult to answer is how we construct a suitable estimator with minimax optimal rates of convergence in the cases with few or an intermediate amount of trajectories ($m \asymgeq 1$, $n_j \asymgt 1$) that is not (optimally) covered by the estimator of the Snake model, i.e., in the case
\begin{equation*}
	\delta \asymlt \br{\frac{T}{n}}^{\frac{1}{2(\beta+1)+1}}
	\eqfs
\end{equation*}
Note that this setting is particularly interesting for the study of chaotic systems \cite{Strogatz2024}, where one expects to revisit points arbitrarily close to past points in the future of the same long trajectory, i.e., $\delta$ is small.
Some practical algorithms for such a setting are compared in \cite{schotz2024machinelearningpredictingchaotic} in a simulation study.
\begin{appendix}
	\section{Relation to Other Models}\label{app:sec:relation}
The ODE models introduced in this article are novel, and the problem of nonparametric ODE estimation has received almost no attention in the mathematical statistics literature. Thus, it is prudent to compare these models with more established ones.

\textbf{Regression on the solutions.} The ODE model can be viewed as a regression model, with the ODE as a constraint on the regression function $\utrue = U(\ftrue, x, \cdot)$ in $[0, T]$. Solutions have a smoothness parameter $\tilde\beta = \beta+1$ if the model function is $\beta$-smooth. Thus, we can solve the reanalysis problem, i.e., estimation of $\utrue$ in $[0, T]$, with a the standard nonparametric rate of convergence for the squared error
\begin{equation*}
    \EOf{\normOf{\uesti(t) - \utrue(t)}^2} \asymleq \br{\frac{T}{n}}^{\frac{2(\beta+1)}{2(\beta+1)+1}}
\end{equation*}
by ignoring the ODE-constraint and using a suitable estimator $\uesti$, see, e.g., \cite{Tsybakov09Introduction}. Furthermore, the derivative of the solution can similarly be estimated with
\begin{equation*}
    \EOf{\normOf{\duesti(t) - \dutrue(t)}^2} \asymleq \br{\frac{T}{n}}^{\frac{2\beta}{2(\beta+1)+1}}
    \eqfs
\end{equation*}
As $\dutrue(t) = \ftrue(\utrue(t))$, this essentially yields an estimate of $\ftrue$ on the set $\utrue([0, T])$. This view is taken for our estimator in the Snake model of section \ref{sec:snake}.

\textbf{Minimax rate of nonparametric regression.}
For standard nonparametric estimation of the $s$-th derivative of a $\tilde \beta$-smooth regression function $g^\star\colon\R^d\to \R$ (or $\R^d\to \R^d$), we have the minimax rate
\begin{equation}\label{eq:intro:regrate}
    \EOf{\normOf{\widehat{D^s g}(x) - D^s g^\star(x)}^2} \asymleq n^{-\frac{2(\tilde\beta-s)}{2\tilde\beta+d}}
    \eqcm
\end{equation}
see, e.g., appendix \ref{app:sec:localreg}.
In the ODE model, we estimate a first derivative ($s=1$) in $d$-dimensions of the solutions that have smoothness $\tilde\beta = \beta +1$. So, the rate of convergence that we obtain in this work \eqref{eq:bestrate} fits the formula of the minimax regression rate \eqref{eq:intro:regrate}, even though the models are quite different.

\textbf{Errors in variables.}
Assume that time steps are constant, i.e., $t_{i+1} - t_{i} = \stepsize$. Consider the so-called propagator $x \mapsto U(\ftrue, x, \stepsize)$. If $\stepsize\to0$, we can infer $f$ from it, as
\begin{equation*}
    \frac{U(\ftrue, x, \stepsize) - x}{\stepsize} =\frac{U(\ftrue, x, \stepsize) - U(\ftrue, x, 0)}{\stepsize} \xrightarrow{\stepsize\to0} \ftrue(x)
    \eqfs
\end{equation*}
Estimating the propagator from the data $(Y_{i-1}, Y_{i})$ can be viewed as an errors-in-variables regression problem, e.g., \cite{mammen12}. In such models, the response $g(x_i)$ as well as the predictors $x_i$ are unknown and observed with noise. Specifically, the data are given as $Z_i = g(x_i) + \varepsilon_i$ and $X_i = x_i + \xi_i$, where $\varepsilon_i$ and $\xi_i$ are noise terms. Setting $(X_i, Z_i) = (Y_{i-1}, Y_{i})$ aligns the problem of propagator estimation in the ODE model with errors-in-variables regression. The Stubble model of section \ref{sec:stubble} avoids the difficulties attached to errors-in-variables by assuming known initial conditions and observing many solutions for only a short time.

\textbf{Nonparametric time series analysis.}
The ODE model with the sequence $(Y_i)$ can be viewed from the perspective of nonparametric time series analysis, where discrete-time stochastic processes in one variable (time) are studied. See, e.g., \cite[Chapter 6]{Fan2003}. On one hand, the focus on discrete time and often only one state dimension, allows for a deep understanding of more complex time evolutions than in an one-dimensional ODE model. On the other hand, it seems that the structure of higher dimensional, time continuous ODEs with measurement noise is typically not captured in such models.

\textbf{Stochastic differential equations.}
The noise in stochastic differential equation (SDE) models, e.g., \cite{Strauch2016, Comte2020} can be described as \textit{system noise} in contrast to the \textit{measurement noise} of the ODE model: It changes the state of the system and solutions are stochastic processes. In the ODE model, the state of the system is not influenced by the noise and solutions are deterministic. In comparison to a model with observed solutions of a SDE in one variable (time), estimation in the ODE model seems more difficult: The observation $Z_t$ at time $t$ of the solution $(Z_t)$ in the SDE model is the true state of the system at that time. If the SDE is Markovian, then the distribution of the future states $(Z_{t\pr})_{t\pr>t}$ only depends on the now known state $Z_t$. In the ODE model, we do not know the true state $\utrue(t)$, but only a noisy observation of it. Even though the future evolution is deterministic, it can be hard to predict if the initial state is not known perfectly (see the topic of \textit{chaotic systems}, e.g., in \cite{Strogatz2024}). Despite these differences, there are similarities between the literature on SDE drift estimation and the two models discussed in this article: \cite{Strauch2016} considers a single, increasingly long trajectory of an SDE solution, analogous to the Snake model, while \cite{Comte2020} considers an increasing number of solutions over a fixed time interval, akin to the Stubble model.

\textbf{Hidden markov models.}
As the true state of the system $\utrue(t)$ is hidden from us and only noisily observed, the ODE model is connected to hidden Markov models (HMMs), e.g., \cite{Bickel1998}. In HMMs, one observes $Z_i$ which depends on the unobserved random variable $X_i$, which in turn depends on $X_{i-1}$. But $Z_{i}$ is independent of $X_{i-1}$ and $Z_{i-1}$ given $X_i$. If we set $X_i = \utrue(t_i)$, $Z_i = Y_i$. We can view the ODE model as a HMM. In contrast to HMMs, the transition from $\utrue(t_i)$ to $\utrue(t_{i+1})$ is deterministic, and the time and state domains are continuous.
\section{Derivatives}\label{app:sec:derivative}
In this section, we introduce notation for derivatives in multiple dimensions and Hölder-type smoothness classes. Furthermore, some elementary analytical results are presented for reference in the main proofs.

For any finite-dimensional $\R$-vector space $V$, we denote the Euclidean norm as $\euclof{x}$ for $x\in V$.
For $k\in\N$ and finite-dimensional $\R$-vector spaces $V$ and $W$, let $\mc L_k(V, W)$ be the set of $k$-multilinear functions from $V^k$ to $W$.
An element of $\mc L_k(V, W)$ is called symmetric if it is invariant under permutations of its $k$ arguments.
A function $f\colon V \to W$ is differentiable at $x\in V$ if there is $A \in \mc L_1(V, W)$ such that
\begin{equation}\label{eq:derivative:first}
    \lim_{\euclOf{v} \to 0} \frac{\euclOf{f(x+v)-f(x) - A(v)}}{\euclOf{v}} = 0
    \eqfs
\end{equation}
In this case we write $D f(x) = A$.
The function $f$ is differentiable if it is differentiable at all $x\in V$. Denote the set of differentiable functions from $V$ to $W$ as $\mc D(V, W)$.
From now on, assume $V = \R^d$ for a $d\in\N$. Define $\mb D_d := \setByEle{v\in\R^d}{\euclOf{v} = 1}$.
The directional derivative of $f\in\mc D(V, W)$ at $x\in V$ in the direction $v\in\mb D_d$ is
\begin{equation*}
    D_v f(x) := D f(x)(v)
    \eqfs
\end{equation*}
We set $D^0$ to the identity, i.e., $D^0 f = f$.
Let $k\in\N$. Define the set of $k$-times differentiable functions $\mc D^k(V, W)$ and the $k$-th derivative operator $D^k$ recursively:
For $k = 1$, we set $\mc D^1(V, W) := \mc D(V, W)$ and $D^1 := D$.
Let $k \in \N_{\geq 2}$.
A function $f\in\mc D^{k-1}(V,W)$ is $k$-times differentiable at $x\in V$ if there is symmetric $A \in \mc L_{k}(V,W)$ such that
\begin{equation*}
    \lim_{\euclOf v\to 0} \sup_{\mo v\in\mb D_d^{k-1}} \frac{\euclOf{D^{k-1}_{\mo v}f(x+v)-D^{k-1}_{\mo v}f(x) - A((\mo v, v))}}{\euclOf{v}} = 0
    \eqfs
\end{equation*}
In this case, we write $D^{k} f(x) = A$.
We may use square brackets for the arguments of the linear operator to make formulas more clear, e.g., $D f(x)[v] := D f(x)(v)$
The function $f$ is $k$-times differentiable if it is $k$-times differentiable at all $x\in V$. Denote the set of $k$-times differentiable functions from $V$ to $W$ as $\mc D^k(V, W)$.
The $k$-th directional derivative of $f\in\mc D^k(V, W)$ at $x\in V$ in the directions $\mo v\in\mb D_d^k$ is
\begin{equation*}
    D_{\mo v} f(x) := D^k f(x)[\mo v]
    \eqfs
\end{equation*}
We may indicate that differentiation takes place with respect to $x$ by $D_{x,\mo v} g(x,y) = D_{\mo v} f_y(x)$, where $f_y(x) = g(x, y)$.
Let $p\in\N$. For $f\in\mc D^k(\R^d, \R^p)$, the derivative operator acts component-wise, i.e., for $f = (f_1, \dots, f_p)$ with $f_\ell\in\mc D^k(\R^d, \R)$, we have
\begin{equation*}
    D_{\mo v}f(x) = \begin{pmatrix}	D_{\mo v}f_1(x) \\ \vdots \\ D_{\mo v}f_p(x) \end{pmatrix}
    \eqcm
\end{equation*}
for $x\in\R^d$ and $\mo v\in \mb D_d^k$.
The operator norm for $A \in \mc L_k(V, W)$ is
\begin{equation*}
    \opNormOf{A} := \sup_{\mo v \in\mb D_d^k} \euclOf{A(\mo v)}
    \eqfs
\end{equation*}
Define the sup-norm of the $k$-th derivative as
\begin{equation*}
    \supNormof{D^kf} := \sup_{x\in\R^d} \opNormof{D^kf(x)} = \sup_{x\in\R^d}  \sup_{\mo v\in \mb D_d^k} \euclOf{D_{\mo v}f(x)}
    \eqfs
\end{equation*}
\begin{notation}
     Let $x\in\R$. Denote the largest integer strictly smaller than $x$ as $\llfloor x\rrfloor := \max \Z_{<x}$.
\end{notation}
\begin{definition}[Hölder-smoothness classes]\label{def:Hoelder}
    Let $d\in\N$ and $L\in\Rpp \cup \{\infty\}$.
    Define
    \begin{equation*}
        \Sigma^{d\to1}(0, L) := \setByEle{f\colon\R^d\to\R \text{ Lebesgue measurable}}{\supNormOf{f} \leq L}\eqfs
    \end{equation*}
    For $\beta\in(0,1]$, define
    \begin{equation*}
        \Sigma^{d\to1}(\beta, L) := \setByEle{f\colon\R^d\to\R}{\forall x,\tilde x \in\R^d\colon\abs{f(x)-f(\tilde x)} \leq L \euclof{x-\tilde x}^{\beta}}\eqfs
    \end{equation*}
    Let $\beta\in\Rppo$ and $\ell:=\llfloor\beta\rrfloor$. Define
    \begin{equation*}
        \Sigma^{d\to1}(\beta, L) := \setByEle{f \in \mc D^{\ell}(\R^d, \R)}{\forall x,\tilde x \in\R^d\colon\opNormof{D^\ell f(x)-D^\ell f(\tilde x)} \leq L \euclof{x-\tilde x}^{\beta - \ell}}\eqfs
    \end{equation*}
    For $\beta\in\Rp$, use the short notation
    \begin{equation*}
        \Sigma(\beta, L) := \Sigma^{1\to1}(\beta, L)\eqfs
    \end{equation*}
    Let $L_0, \dots, L_\ell, \Lbeta \in \Rpp\cup\{\infty\}$. Define
    \begin{equation*}
        \Sigma^{d\to 1}(\beta, \indset{L}{0}{\ell}, \Lbeta) := \Sigma^{d\to1}(\beta, \Lbeta) \cap \bigcap_{k=0}^\ell \Sigma^{d\to1}(k, L_k)\eqfs
    \end{equation*}
    Let $\dm{in},\dm{out}\in\N$.  Define
    \begin{equation*}
        \Sigma^{\dm{in}\to \dm{out}}(\beta, \indset{L}{0}{\ell}, \Lbeta) := \br{\Sigma^{\dm{in}\to 1}(\beta, \indset{L}{0}{\ell}, \Lbeta)}^{\dm{out}}\eqcm
    \end{equation*}
    where $f = (f_1, \dots, f_{\dm{out}})\in \Sigma^{\dm{in}\to \dm{out}}(\beta, \indset{L}{0}{\ell}, \Lbeta)$ is treated as a function
    \begin{equation*}
        f\colon \R^{\dm{in}} \to \R^{\dm{out}}, x\mapsto (f_1(x), \dots, f_{\dm{out}}(x))\tr
        \eqfs
    \end{equation*}
    For $\beta \in \N$, we denote
     \begin{equation*}
        \bar\Sigma^{\dm{in}\to \dm{out}}(\beta, \indset{L}{0}{\beta}) := \Sigma^{\dm{in}\to \dm{out}}(\beta, \indset{L}{0}{\beta}) \cap \mc D^\beta(\R^{\dm{in}}, \R^{\dm{out}})\eqfs
    \end{equation*}
\end{definition}
A bound on the derivative yields the Lipschitz constant, as the next lemma shows.
\begin{lemma}\label{app:diff:lem}
    Let $d\in\N$.
    Let $k\in\N$.
    Let $f\in \mc D^{k}(\R^d, \R)$.
    Let $L\in\Rpp$.
    Assume $\supNormof{D^{k}f} \leq L$.
    Then $f \in \Sigma^{d\to1}(k, L)$.
\end{lemma}
\begin{proof}[Proof of \cref{app:diff:lem}]
    Let $\mo v\in\mb D_d^{k-1}$. Then, by the mean value theorem,
    \begin{equation*}
        \abs{D^{k-1}_{\mo v}f(x) - D^{k-1}_{\mo v}f(\tilde x)} \leq \sup_{\tilde v \in \mb D_d}\supNormOf{D^{k}_{(\mo v, \tilde v)}f} \euclOf{x - \tilde x}
    \end{equation*}
    with $(\mo v, \tilde v)\in\mb D_d^{k}$. Thus,
    \begin{equation*}
        \opNormOf{D^{k-1}f(x) - D^{k-1}f(\tilde x)} \leq \supNormOf{D^{k}f} \euclOf{x - \tilde x}
        \eqfs
    \end{equation*}
\end{proof}
The sub-multiplicativity of the operator norm in the following lemma is a direct consequence of its definition.
\begin{lemma}\label{lmm:opnorm}
    Let $k,d,p\in\N$. Let $A \in \mc L_k(\R^d, \R^p)$ and $v \in \mb D_d^k$. Then
    \begin{equation*}
        \euclOf{A(v_1, \dots, v_k)} \leq  \opNormOf{A} \prod_{j=1}^k \euclOf{v_j}
    \end{equation*}
\end{lemma}
For reference, we state the chain and product rule as well as Taylor's approximation theorem in the following lemmas.
\begin{lemma}\label{lmm:derivrules}
    Let $d\in\N$. Let $A \colon \R^d \to \mc L_k(\R^d, \R^d)$ symmetric be differentiable.
    Let $v \in \mb D_d$.
    \begin{enumerate}[label=(\roman*)]
        \item \label{lmm:derivrules:product}
        Product rule:
        Let $f_\ell \in \mc D(\R^d, \R^d)$ for $\ell \in \nnset k$. Then
        \begin{align*}
            D_v \br{A(x)[f_1(x),\dots,f_k(x)]}
            &=
            (D_v A(x))[f_1(x),\dots,f_k(x)] \,+
            \\&\phantom{=}\
            \sum_{\ell=1}^k A(x)\abOf{f_1(x),\dots, f_{\ell-1}(x),D_v f_\ell(x), f_{\ell+1}(x),\dots,f_k(x)}
            \eqfs
        \end{align*}
        \item \label{lmm:derivrules:chain}
        Chain rule:
        Let $f \in \mc D(\R^d, \R^d)$ and $a_1, \dots, a_k \in \R^d$. Then
        \begin{equation*}
            D_v \Big(A(f(x))[a_1,\dots,a_k]\Big)
            =
            (D A)(f(x))\abOf{a_1,\dots,a_k, D_vf(x)}
            \eqfs
        \end{equation*}
    \end{enumerate}
\end{lemma}
\begin{notation}
    Let $d, m\in\N$. Let $x \in \R^d$. Denote $\{x\}^m = (x, \dots, x) \in (\R^d)^m$.
\end{notation}
\begin{lemma}[Taylor theorem with Peano remainder]\label{lmm:taylor}
    Let $d \in \N$.
    \begin{enumerate}[label=(\roman*)]
        \item
        Let $\ell\in\N$. Let $f \in \mc D^\ell(\R^d, \R)$.
        Then, there is $\gamma \in [0,1]$, such that
        \begin{equation*}
            f(x) = \sum_{k=0}^{\ell-1} \frac{D^k f(x_0)(\{x-x_0\}^k)}{k!} + \frac{D^\ell f\brOf{\tilde x}(\{x-x_0\}^\ell)}{\ell!}
        \end{equation*}
        with $\tilde x := x_0 + \gamma(x-x_0)$.
        \item
        Let $\beta, L \in \Rpp$. Set $\ell:=\llfloor\beta\rrfloor$. Let $f \in \Sigma^{d\to1}(\beta, L)$. Then
        \begin{equation*}
            \abs{f(x) - \sum_{k=0}^{\ell} \frac{D^k f(x_0)(\{x-x_0\}^k)}{k!}} \leq L \frac{\euclOf{x-x_0}^\beta}{\ell!}
            \eqfs
        \end{equation*}
    \end{enumerate}
\end{lemma}

    \section{Local Polynomial Regression}\label{app:sec:localreg}
In order to make this article more self-contained, we provide here upper bounds on the estimation error for a multivariate regression problem using local polynomial estimators. We follow \cite{Tsybakov09Introduction} and extend the proof there from the domain $[0, 1]$ to the domain $[0, T]^d$ and also discuss estimation of derivatives. See also \cite{Ruppert1994}. It is important for the results in the main part of the article that the dependence on $T$ and on the Lipschitz constant $L$ of the regression function are given explicitly.
\subsection{Model}\label{app:sec:localreg:model}
Let $d\in\N$.
Let $\beta, L \in\Rpp$. Set $\ell := \llfloor\beta\rrfloor$.
Let $f \in\Sigma^{d\to1}(\beta, L)$.
Let $n\in\N$.
Let $T\in\Rpp$. For $i\in\nnset n$, let $x_i\in[0, T]^d$.
Let $\sigma \in \Rpp$.
For $i\in\nnset n$, let $\noise_i$ be independent $\R$-valued random variables with $\Eof{\epsilon_i} = 0$ and $\Eof{\epsilon_i^2} \leq \sigma^2$.
For $i\in\nnset n$, set
\begin{equation}\label{eq:regression:model}
    Y_i = f(x_i) + \noise_i
    \eqfs
\end{equation}
We observe $Y_i$ and know $x_i$. The function $f$ is unknown.
Let $s\in\nnzset{\ell}$. Let $\mo v\in\mb D_d^s$ be $s$-many directions.
Our target is to estimate the $s$-th directional derivative $ D_{\mo v} f \colon \R^d \to \R$ of $f$.
For the local polynomial estimator $\hat g$ defined below, we want to show upper bounds for the mean squared error $\Eof{\normof{\hat g(x) - D_{\mo v} f(x)}^2}$ at a point $x\in[0, T]^d$ depending on $d, \beta, L, n, T, \sigma, s$.
\subsection{Estimator}\label{app:sec:localreg:estimator}
Recall \cref{not:poly}. Set the number of scalar coefficients for polynomials $\R^d\to \R$ of degree at most $\ell$ as
\begin{equation*}
    N := N_{d,\ell} := \dim(\Poly d\ell) = \begin{pmatrix} \ell + d \\ d\end{pmatrix}
    \eqfs
\end{equation*}
For $x\in\R^d$, denote $\psi(x)  := \psi_\ell(x) := ((\alpha!)^{-1}x^\alpha)_{\alpha \in \N_0^d,|\alpha|\leq \ell} \in \R^N$.
Let $K\colon \Rp \to \R$. Let $h\in\Rpp$.
Denote the local polynomial mean squared error coefficients as
\begin{equation*}
    \hat\theta(x) \in \argmin_{\theta\in\R^{N}} \sum_{i=1}^{n} \br{Y_i - \theta\tr \psi\brOf{\frac{x_i-x}{h}}}^2 K\brOf{\frac{\euclOf{x_i-x}}{h}}
    \eqfs
\end{equation*}
Then the local polynomial estimator of degree $\ell$ for the $s$-th derivative in directions $\mo v\in\mb D_d^s$ is defined as
\begin{equation}\label{eq:lp:def}
    \hat g(x) := h^{-s} D_{\mo v} \psi(0)\tr \hat\theta(x)
    \eqfs
\end{equation}

For $x\in\R^d$, define the symmetric matrix $B(x) \in\R^{N \times N}$ and the vector $a(x)\in\R^N$ as
\begin{align}\label{eq:lp:B}
    B(x) &:= \frac{T^d}{nh^d} \sum_{i=1}^{n} \psi\brOf{\frac{x_i - x}{h}}\psi\brOf{\frac{x_i - x}{h}}\tr K\brOf{\frac{\euclOf{x_i-x}}{h}}
    \eqcm\\\label{eq:lp:a}
    a(x) &:= \frac{T^d}{nh^d} \sum_{i=1}^{n} Y_i \psi\brOf{\frac{x_i - x}{h}}K\brOf{\frac{\euclOf{x_i-x}}{h}}
    \eqfs
\end{align}
\begin{notation}
    Let $k\in\N$ and $A\in\R^{k\times k}$ be a symmetric matrix. We write $\evmin(A)$ for the smallest eigenvalue of $A$.
\end{notation}
Assume $\evmin(B(x)) > 0$. Then
\begin{equation*}
     \hat\theta(x) = B(x)^{-1} a(x)
     \eqfs
\end{equation*}
Furthermore, with the weights
\begin{equation}\label{eq:lp:defweights}
    w_i(x) := w_{i,s}(x) := \frac{T^d}{nh^{d+s}} D_{\mo v} \psi(0)\tr B(x)^{-1} \psi\brOf{\frac{x_i - x}{h}}K\brOf{\frac{x_i-x}{h}}
    \eqcm
\end{equation}
we can write the estimator as
\begin{equation}\label{eq:regression:lp}
    \hat g(x) = h^{-s} D_{\mo v} \psi(0)\tr \hat\theta(x) = \sum_{i=1}^{n} w_i(x) Y_i
    \eqfs
\end{equation}
\subsection{Result}\label{app:sec:localreg:results}
\begin{notation}\mbox{}
    All lower case $c$, with or without index, are elements of $\Rpp$ and universal insofar as they only depend on the variables written as index, e.g., $c_{d,\beta}$ depends only on $\beta$ and $d$. In particular, a constant $c$ with no index refers to a fixed positive number. Every occurrence of such a variable may refer to a different value.
\end{notation}
Recall \assuRef{Kernel}, \assuRef{Eigenvalue}, and \assuRef{Cover} defined in \cref{ass:localreg}.
\begin{theorem}\label{thm:localpoly}
    Use the model of section \ref{app:sec:localreg:model} and the estimator of section \ref{app:sec:localreg:estimator}.
    Assume \assuRef{Kernel}, \assuRef{Eigenvalue}, and \assuRef{Cover}. Assume $n \geq \const{cvr}^{-1} h^{-d} T^d$.
    Then, for all $x \in [0, T]^d$, we have
    \begin{equation*}
        \EOf{\br{\hat g(x) - D_{\mo v} f(x)}^2}
        \leq
        c_\ell \br{\const{egv}^2 \const{ker}^2 \const{cvr}^2 L^2 h^{2(\beta-s)} + \const{egv}^2 \const{ker}^2 \const{cvr} \frac{\sigma^2 T^{d}}{nh^{d+2s}}}
        \eqfs
    \end{equation*}
\end{theorem}
\begin{corollary}\label{cor:localpoly}
    Use the model of section \ref{app:sec:localreg:model} and the estimator of section \ref{app:sec:localreg:estimator}.
    Assume \assuRef{Kernel}, \assuRef{Eigenvalue}, and \assuRef{Cover}.
    Assume $n \geq n_0$, where $n_0\in\N$ depends only on $d, \beta, \const{egv}, \const{ker}, \const{cvr}, \sigma, T, L$.
    Then there are constants $C_1, C_2 \in \Rpp$, depending only on $d, \beta, \const{egv}, \const{ker}, \const{cvr}$, with the following property: Set the bandwidth as
    \begin{equation*}
         h = C_1 \br{\frac{\sigma^2 T^d}{L^2n}}^{\frac{1}{2\beta+d}}
         \eqfs
    \end{equation*}
    Then, for all $x \in [0, T]^d$, we have
    \begin{equation*}
        \EOf{\br{\hat g(x) - D_{\mo v} f(x)}^2} \leq C_2 L^{\frac{2d+4s}{2\beta+d}} \br{\frac{\sigma^2 T^d}{n}}^{\frac{2(\beta-s)}{2\beta+d}}
        \eqfs
    \end{equation*}
\end{corollary}
\begin{remark}\mbox{ }
    \begin{enumerate}[label=(\roman*)]
        \item
        The error rates obtained in \cref{cor:localpoly} are minimax optimal up to the constant $C_2$, which can be seen by comparing them to lower bounds for nonparametric regression, e.g., \cite[Theorem 8]{lowerbounds}.
        \item
        Integrating the point-wise error over $[0, T]^d$ yields the minimax optimal rate for the integrated error. For a bound on the error in sup-norm, we need to do some extra work.
    \end{enumerate}
\end{remark}
\begin{theorem}\label{thm:localpoly:supnorm}
    Assume \assuRef{Kernel}, \assuRef{Eigenvalue}, and \assuRef{Cover}.
    Assume \assuRef{SubGaussian}.
    Assume $n \geq \const{cvr}^{-1} h^{-d} T^d$
    Assume $T/h \geq 2$.
    Then, we have
    \begin{equation*}
        \EOf{\supNormOf{\hat g - D_{\mo v} f}^2}
        \leq
        c_{d,\ell,K}  \br{\br{\const{egv} \const{cvr} L h^{\beta-s}}^2 + \const{cvr} \const{egv}^2 \frac{\sigma^2T^d}{nh^{d+2s}} \log(T/h)}
        \eqfs
    \end{equation*}
\end{theorem}
\begin{corollary}\label{cor:localpoly:supnorm}
    Assume \assuRef{Kernel}, \assuRef{Eigenvalue}, and \assuRef{Cover}.
    Assume \assuRef{SubGaussian}.
    There are $C_1, C_2, C_3, C_4 \in \Rpp$ depending only on $d, \beta, \const{egv}, \const{cvr}, K$ with the following property:
    Assume
    \begin{equation}\label{eq:lp:Tbound}
         C_1 \br{\frac{\log(n)}n}^{\frac{d}{2\beta}} \leq C_2 \br{\frac{L^2}{\sigma^2}}^{\frac d{2\beta}} T^d \leq n \log(n)^{\frac{d}{2\beta}}
         \eqfs
    \end{equation}
    Set the bandwidth as
    \begin{equation*}
        h = C_3 \br{\frac{\sigma^2 T^d \log(n)}{L^2n}}^{\frac{1}{2\beta+d}}
        \eqfs
    \end{equation*}
    Then
    \begin{equation*}
        \EOf{\supNormOf{\hat g - D_{\mo v} f}^2}
        \leq
        C_4  L^{\frac{2d+4s}{2\beta+d}} \br{\frac{\sigma^2 T^d \log(n)}{n}}^{\frac{2(\beta-s)}{2\beta+d}}
        \eqfs
    \end{equation*}
\end{corollary}
\begin{remark}\mbox{ }
        The error rates obtained in \cref{cor:localpoly:supnorm} are minimax optimal up to the constant $C_4$, which can be seen by comparing them to lower bounds for nonparametric regression, e.g., \cite[Theorem 8]{lowerbounds}.
\end{remark}
Following \cite[Lemma 1.4, Lemma 1.5]{Tsybakov09Introduction}, we show that the uniform grid in $[0, T]^d$ fulfills \assuRef{Cover}, and also \assuRef{Eigenvalue} if the kernel is not degenerated. Recall \cref{ass:strictkernel} for \assuRef{StrictKernel}.
\begin{proposition}\label{prp:uniformgrid}
    Let $n_0\in\N$.
    Set $n := n_0^d$.
    Let $x_i \in [0,T]^d$ for $i\in\nnset n$ form a uniform grid in $[0,T]^d$, i.e.,
    \begin{equation*}
        \cb{x_1, \dots, x_n} =
        \setByEle{
            \begin{pmatrix}
                T\frac{k_1}{n_0}, \dots, T\frac{k_d}{n_0}
            \end{pmatrix}\tr
        }{k_1, \dots, k_d\in\nnset{n_0}}
        \eqfs
    \end{equation*}
    Assume \assuRef{StrictKernel} with constant $\const{ker} \in\Rpp$.
    Then \assuRef{Kernel} is fulfilled with  $\const{ker}$, \assuRef{Eigenvalue} is fulfilled with $\const{egv} = c_{d,\ell,K}(1 + \frac{T}{n_0 h})$, and \assuRef{Cover} is fulfilled with $\const{cvr} = 4^d$.
\end{proposition}
\subsection{Proof}\label{app:sec:localreg:proof}
\begin{lemma}\label{lmm:reproduction}
    Let $q\in\Poly d{\ell}$. Assume $\evmin(B(x)) > 0$.
    Then
    \begin{equation*}
        \sum_{i=1}^{n} w_i(x)  q(x_i) = D_{\mo v} q(x)
        \eqfs
    \end{equation*}
\end{lemma}
\begin{proof}
    We set $Y_i$ to $q(x_i)$ and consider the local polynomial fit. The coefficients $\theta_q(x)$ in the fit are
    \begin{equation*}
        \theta_q(x) \in \argmin_{\theta\in\R^{N}} \sum_{i=1}^{n} \br{q(x_i) - \theta\tr \psi\brOf{\frac{x_i-x}{h}}}^2 K\brOf{\frac{\euclOf{x_i-x}}{h}}
        \eqfs
    \end{equation*}
    Let $z\in\R^d$.
    Since $q(z)$ is a polynomial of degree at most $\ell$ and $\evmin(B(x)) > 0$, the fit of the degree $\ell$ polynomial with base $\psi((z-x)/h)$ retrieves the original function, i.e.,
    \begin{align*}
        q(z) = \theta_q(x)\tr\psi\brOf{\frac{z-x}{h}}
        \eqfs
    \end{align*}
    Denote by $D_{z, \mo v}$ the $s$-th directional derivative in directions $\mo v$ with respect to $z$. Then,
    \begin{equation*}
        D_{\mo v} q(z)
        =
        \theta_q(x)\tr D_{z, \mo v}\br{\psi\brOf{\frac{z-x}{h}}}
        =
        h^{-s} \theta_q(x)\tr \br{D_{\mo v}\psi}\brOf{\frac{z-x}{h}}
        \eqfs
    \end{equation*}
    Thus, at $z = x$, we obtain, using the definition of $w_i$,
    \begin{equation*}
        D_{\mo v} q(x) = h^{-s} \theta_q(x)\tr\br{D_{\mo v}\psi}\brOf{0} = \sum_{i=1}^{n} w_i(x) q(x_i)
        \eqfs
    \end{equation*}
\end{proof}
\begin{lemma}\label{lmm:lp:weights}
    Assume \assuRef{Kernel}, \assuRef{Eigenvalue}, and \assuRef{Cover}. Assume $n \geq \const{cvr}^{-1} h^{-d} T^d$.
    \begin{enumerate}[label=(\roman*)]
        \item\label{lmm:lp:weights:supp}
            If $\normof{x - x_i} \geq h$, we have $w_i(x) = 0$.
        \item\label{lmm:lp:weights:sumabs}
            We have
            \begin{equation*}
                \sum_{i=1}^{n} \abs{w_i(x)} \leq c_\ell \const{egv} \const{ker} \const{cvr} h^{-s}
                \eqfs
            \end{equation*}
         \item\label{lmm:lp:weights:sumsqr}
            We have
            \begin{equation*}
                \sum_{i=1}^{n} w_i(x)^2 \leq c_\ell \const{egv}^2 \const{ker}^2 \const{cvr} \frac{T^{d}}{nh^{d+2s}}
                \eqfs
            \end{equation*}
    \end{enumerate}
\end{lemma}
\begin{proof}
    Using \assuRef{Kernel} and \assuRef{Eigenvalue}, we obtain from \eqref{eq:lp:defweights}
    \begin{align*}
        nh^{d+s}T^{-d} \abs{w_i(x)}
        &\leq
        \euclOf{D_{\mo v}\psi(0)} \opNormof{B(x)^{-1}} \euclOf{\psi\brOf{\frac{x_i - x}{h}}} \supNormof{K} \ind_{\ball^d(x_i, h)}
        \\&\leq
        c_\ell \const{egv} \const{ker} \ind_{\ball^d(x_i, h)}
        \eqcm
    \end{align*}
     which we use in each of the following parts.
    \begin{enumerate}[label=(\roman*)]
        \item
            Trivial.
        \item
            By \assuRef{Cover},
            \begin{align*}
                \br{\const{egv} \const{ker}}^{-1} \sum_{i=1}^{n} \abs{w_i(x)}
                &\leq
                c_\ell T^d h^{-(d+s)} \frac1n\# \setByEle{i \in \nnset{n}}{x_i \in \ball^d(x, h)}
                \\&\leq
                c_\ell T^d h^{-(d+s)} \max\brOf{\frac1{n},\, \const{cvr} \br{\frac{h}{T}}^d}
                \\&\leq
                c_\ell \const{cvr} h^{-s}
                \eqcm
            \end{align*}
            as we assume $n \geq \const{cvr}^{-1} h^{-d} T^d$.
         \item
            With
            \begin{equation*}
                 \abs{w_i(x)} \leq  c_\ell \const{egv} \const{ker} \frac{T^{d}}{nh^{d+s}}
            \end{equation*}
            and \ref{lmm:lp:weights:sumabs}, we obtain
            \begin{align*}
                \sum_{i=1}^{n} \abs{w_i(x)}^2
                &\leq
                \sup_{i\in\nnset n} \abs{w_i(x)} \, \sum_{i=1}^{n} \abs{w_i(x)}
                \\&\leq
                c_\ell \const{egv}^2 \const{ker}^2 \const{cvr} \frac{T^{d}}{nh^{d+2s}}
                \eqfs
            \end{align*}
    \end{enumerate}
\end{proof}
\begin{proof}[Proof of \cref{thm:localpoly}]
    \textbf{Bias.}
    By the model equation \eqref{eq:regression:model} and the linear representation of the estimator \eqref{eq:regression:lp}, we have
    \begin{equation*}
        \EOf{\hat g(x)}  = \sum_{i=1}^{n} w_i(x) f(x_i)
        \eqfs
    \end{equation*}
    Using a Taylor expansion (\cref{lmm:taylor}) of $f(x)$ at $x_i$ yields the existence of $\gamma_i \in [0,1]$ such that
    \begin{equation*}
        f(x) = \sum_{k=0}^{\ell-1} \frac{D^k f(x)(\{x_i-x\}^k)}{k!} + \frac{D^\ell f\brOf{x + \gamma_i(x_i - x)}(\{x_i-x\}^\ell)}{\ell!}
        \eqfs
    \end{equation*}
    Denote the remainder term as
    \begin{equation*}
        R_i(x) := \frac{D^\ell f\brOf{x + \gamma_i(x_i - x)}(\{x_i-x\}^\ell) - D^\ell f\brOf{x}(\{x_i-x\}^\ell)}{\ell!}
        \eqfs
    \end{equation*}
    By applying \cref{lmm:reproduction} to the multivariate polynomial
    \begin{equation*}
        q_x \colon \R^d \to \R, z\mapsto \sum_{k=0}^{\ell} \frac{D^k f(x)(\{z-x\}^k)}{k!}
        \eqcm
    \end{equation*}
    we obtain
    \begin{align*}
        \sum_{i=1}^{n} w_i(x) f(x_i)
        &=
        \sum_{i=1}^{n} w_i(x) q_x(x_i) + \sum_{i=1}^{n} w_i(x)  R_i(x)
        \\&=
        D_{\mo v} q_x(x) + \sum_{i=1}^{n} w_i(x)  R_i(x)
        \eqfs
    \end{align*}
    Together with $D_{\mo v} q_x(x) = D_{\mo v} f(x)$, we get
    \begin{align*}
        \EOf{\hat g(x)} - D_{\mo v} f(x) = \sum_{i=1}^{n} w_i(x) R_i(x)
        \eqfs
    \end{align*}
    For bounding the remainder term, we note that $f\in\Sigma^{d\to 1}(\beta, L)$, which yields
    \begin{equation*}
        \abs{D^\ell f\brOf{x + \gamma_i(x_i - x)}(\{x_i-x\}^\ell) - D^\ell f\brOf{x}(\{x_i-x\}^\ell)} \leq L \euclOf{x_i-x}^\beta
        \eqfs
    \end{equation*}
    Thus, with \cref{lmm:lp:weights} \ref{lmm:lp:weights:supp}, we have
    \begin{align*}
        \abs{\EOf{\hat g(x)} - D_{\mo v} f(x)}
        &\leq
        L \sum_{i=1}^{n} \abs{w_i(x)} \euclOf{x_i-x}^\beta
        \\&\leq
        L h^\beta \sum_{i=1}^{n} \abs{w_i(x)}
        \eqfs
    \end{align*}
    Finally, by \cref{lmm:lp:weights} \ref{lmm:lp:weights:sumabs}, we obtain
    \begin{equation}\label{eq:regression:bias}
        \abs{\EOf{\hat g(x)} - D_{\mo v} f(x)} \leq c_\ell \const{egv} \const{ker} \const{cvr} L h^{\beta-s}
        \eqfs
    \end{equation}

    \textbf{Variance.}
    As the noise is independent and has a second moment bounded by $\sigma^2$, we have
    \begin{align*}
        \EOf{\br{\hat g(x) - \EOf{\hat g(x)}}^2}
        &=
        \EOf{\br{\sum_{i=1}^{n} \noise_i w_i(x)}^2}
        \\&=
        \sum_{i=1}^{n} w_i(x)^2 \EOf{\noise_i^2}
        \\&\leq
        \sigma^2 \sum_{i=1}^{n} w_i(x)^2
        \eqfs
    \end{align*}
    Hence, \cref{lmm:lp:weights} \ref{lmm:lp:weights:sumsqr} yields
    \begin{equation}\label{eq:regression:variance}
        \EOf{\br{\hat g(x) - \EOf{\hat g(x)}}^2}
        \leq
        c_\ell \const{egv}^2 \const{ker}^2 \const{cvr} \frac{\sigma^2 T^{d}}{nh^{d+2s}}
        \eqfs
    \end{equation}

    \textbf{Together.} Putting \eqref{eq:regression:bias} and \eqref{eq:regression:variance} together, we obtain
    \begin{align*}
        \EOf{\br{\hat g(x) - D_{\mo v} f(x)}^2}
        &\leq
        c_\ell \br{\const{egv}^2 \const{ker}^2 \const{cvr}^2 L^2 h^{2(\beta-s)} + \const{egv}^2 \const{ker}^2 \const{cvr} \frac{\sigma^2 T^{d}}{nh^{d+2s}}}
        \eqfs
    \end{align*}
\end{proof}
For the proof in sup-norm (\cref{thm:localpoly:supnorm}), we want an upper bound on the supremum of a sub-Gaussian process using the chaining technique. In order to do this, we need some preparation: Recall \cref{def:subgauss:variable} of sub-Gaussian random variables.
\begin{definition}[Sub-Gaussian process]\label{def:subgauss:process}
    Let $(\mc X, \rho)$ be a metric space.
    A real-valued stochastic process $(Z_x)_{x\in\mc X}$ is \textit{sub-Gaussian} if
    \begin{equation*}
        \PrOf{\abs{Z_x - Z_{x\pr}} \geq t} \leq 2 \exp\brOf{-\frac{t^2}{2\rho(x, x\pr)^2}}
    \end{equation*}
    for all $t\in\Rpp$, $x,x\pr\in\mc X$.
\end{definition}
\begin{lemma}\label{lmm:subgauss:lin}
    Let $\epsilon_i$ be independent, centered, sub-Gaussian random variables with variance parameters $v_i\in\Rpp$.
    Let $(\mc X, \rho)$ be a metric space.
    Let $b_i\colon\mc X\to \R$.
    Define $Z_x := \sum_{i=1}^{n} b_i(x) \epsilon_i$.
    Then $Z_x$ is sub-Gaussian with variance parameter $v_x := \sum_{i=1}^{n} b_i(x)^2 v_i$.
    Furthermore, $(Z_x)_{x\in\mc X}$ is sub-Gaussian with $\rho(x, x\pr)^2 := \sum_{i=1}^{n} \br{b_i(x)-b_i(x\pr)}^2 v_i$.
\end{lemma}
\begin{proof}
    The results stem from the well-known fact that a linear combination of independent sub-Gaussian variables is again sub-Gaussian. It can be checked using the following equivalent condition for sub-Gaussian:
    A centered random variable $X$ is sub-Gaussian with variance parameter $v$ if and only if $\Eof{\exp(\lambda X)} \leq \exp\brOf{\frac{\lambda^2 v}{2}}$ for all $\lambda \in\Rp$.
\end{proof}
\begin{definition}[Covering Number]
    Let $(\mc X, \rho)$ be a metric space, $A\subset \mc X$, and $r\in\Rpp$.
    The covering number $N(A, \rho, r)$ of $A$ with balls $\ball^\rho(x, r) := \setByEleInText{z\in\mc X}{\rho(x,z)\leq r}$ of radius $r$ is defined as
    \begin{equation*}
        N(A, \rho, r) := \min \setByEle{n\in\N}{\exists (x_i)_{i\in\nnset n} \subset \mc X \colon A \subset \bigcup_{i\in\nnset n}\ball^\rho(x_i, r)}
        \eqfs
    \end{equation*}
\end{definition}
\begin{lemma}\label{lmm:covernum}
    Let $d\in\N$, $a, h,T\in\Rpp$.
    Define the metric $\rho$ in $\R^d$ as
    \begin{equation*}
        \rho(x, x\pr) := a \min\brOf{1\eqcm\ \frac{\euclOf{x-x\pr}}h}
        \eqfs
    \end{equation*}
    Assume $T \geq h$.
    Then
    \begin{equation*}
        \int_0^{\infty} \sqrt{\log(N([0, T]^d, \rho, r))} \dl r
        \leq
        c_{d} a \max\brOf{1, \sqrt{\log\brOf{\frac{T}{h}}}}
        \eqfs
    \end{equation*}
\end{lemma}
\begin{proof}
    First consider $a = 1$.
    For $r\leq 1$,
    \begin{equation*}
        N([0, T]^d, \rho, r) \leq c_d T^d h^{-d} r^{-d}
        \eqfs
    \end{equation*}
    For $r\geq 1$,
    \begin{equation*}
        N([0, T]^d, \rho, r) = 1
        \eqfs
    \end{equation*}
    Thus,
    \begin{align*}
        \int_0^{\infty} \sqrt{\log(N([0, T]^d, \rho, r))} \dl r
        &\leq
        \int_0^{1} \sqrt{\log(c_d T^d h^{-d} r^{-d})} \dl r
        \\&\leq
        c_d + \sqrt{d\log(T/h)} + \sqrt{d}\int_0^{1} \sqrt{\log(1/r)} \dl r
        \eqfs
    \end{align*}
    We have
    \begin{equation*}
        \int_0^{1} \sqrt{\log(1/r)} \dl r = \int_0^{\infty} \sqrt{t} \exp(-t) \dl t = \Gamma\brOf{\frac32} = \frac12 \sqrt{\pi}
        \eqcm
    \end{equation*}
    where $\Gamma(\alpha)$ is the Gamma-function.
    All in all, we obtain
    \begin{equation*}
         \int_0^{\infty} \sqrt{\log(N([0, T]^d, \rho, r))} \dl r \leq c_d \max\brOf{1\eqcm\, \sqrt{\log(T/h)}}
         \eqfs
    \end{equation*}
    For arbitrary $a\in\Rpp$, and any metric space $(\mc X, \tilde\rho)$, we note that
    \begin{equation*}
        \int_0^{\infty} \sqrt{\log(N(\mc X, a\tilde \rho, r))} \dl r
        =
        \int_0^{\infty} \sqrt{\log(N(\mc X, \tilde \rho, r/a))} \dl r
        =
        a \int_0^{\infty} \sqrt{\log(N(\mc X, \tilde \rho, s))} \dl s
        \eqfs
    \end{equation*}
\end{proof}
\begin{lemma}[Chaining]\label{lmm:chaining}
    Let $(\mc X, \rho)$ be a metric space. Let $(Z_x)_{x\in\mc X}$ be a centered, sub-Gaussian process. Let $x_0\in\mc X$. Let $k\in\N$. Then
    \begin{equation*}
        \EOf{\sup_{x\in\mc X} \abs{Z_x - Z_{x_0}}^k} \leq c_k \br{\int_0^\infty \sqrt{\log(N(\mc X, \rho, r))} \dl r}^k
        \eqfs
    \end{equation*}
\end{lemma}
\begin{proof}
    E.g., \cite[chapter 2]{Talagrand2021}. To be more precise, combine the results and definitions of Theorem 2.7.13, Definition 2.7.3, Exercise 2.7.6 (b), (2.40), (2.37), and Exercise 2.3.8 (b) in that reference.
\end{proof}
\begin{proof}[Proof of \cref{thm:localpoly:supnorm}]
    \textbf{Bias.}
    As
    \begin{equation*}
        \EOf{\supNormOf{\hat g - D_{\mo v} f}^2} \leq c \br{\EOf{\sup_{x\in[0,T]^d}\abs{\hat g(x) - \EOf{\hat g(x)}}^2} + \sup_{x\in[0,T]^d}\abs{\EOf{\hat g(x)} - D_{\mo v} f(x)}^2}
    \end{equation*}
    we can use the same bound of the bias term as in \cref{thm:localpoly}, i.e., \eqref{eq:regression:bias},
    \begin{equation*}
        \sup_{x\in[0,T]^d}\abs{\EOf{\hat g(x)} - D_{\mo v} f(x)} \leq c_\ell \const{egv} \const{ker} \const{cvr} L h^{\beta-s}
        \eqfs
    \end{equation*}
    \textbf{Variance.}
    For the variance term, we have to bound the expectation of the supremum of an empirical process,
    \begin{equation}\label{eq:lp:sup:var}
        \EOf{\sup_{x\in[0,T]^d}\abs{\hat g(x) - \EOf{\hat g(x)}}^2} =
        \EOf{\sup_{x\in[0,T]^d}\abs{\sum_{i=1}^{n} w_i(x) \epsilon_i}^2}
        \eqfs
    \end{equation}
    Write
    \begin{equation*}
        w_i(x) = \eta\tr B(x)^{-1} b_i(x)
    \end{equation*}
    with
    \begin{align*}
        b_i(x) &:= \psi\brOf{\frac{x_i - x}{h}}K\brOf{\frac{\euclOf{x_i-x}}{h}}\eqcm
        \\\eta &:= \frac{T^d}{nh^{d+s}} D_{\mo v} \psi(0)\eqfs
    \end{align*}
    Then
    \begin{align*}
        \abs{\sum_{i=1}^{n} w_i(x) \epsilon_i}
        &\leq
        \abs{\eta\tr B(x)^{-1} \sum_{i=1}^{n} b_i(x) \epsilon_i}
        \\&\leq
        \normof{\eta} \opNormOf{B(x)^{-1}} \euclOf{\sum_{i=1}^{n} b_i(x) \epsilon_i}
        \\&\leq
        c_\ell \frac{T^d}{nh^{d+s}} \const{egv} \euclOf{\sum_{i=1}^{n} b_i(x) \epsilon_i}
        \eqfs
    \end{align*}
    Hence, we have to bound the expectation of
    \begin{equation*}
        \sup_{x\in[0, T]^d}\euclOf{\sum_{i=1}^{n} b_i(x) \epsilon_i}
        \leq
        \sum_{j=1}^{N}
        \sup_{x\in[0, T]^d}\abs{\sum_{i=1}^{n} e_j\tr b_i(x) \epsilon_i}
        \eqfs
    \end{equation*}
    where $e_j\in\R^N$ is the $j$-th standard unit vector.
    Fix a $j\in\nnset N$ and define $Z_{j,i}(x) := e_j\tr b_i(x) \epsilon_i$. We want to find an upper bound for
    \begin{equation*}
        \EOf{\sup_{x\in[0, T]^d} \abs{\sum_{i=1}^{n} Z_{j,i}(x)}^2}
        \eqfs
    \end{equation*}
    Because of \assuRef{SubGaussian} and \cref{lmm:subgauss:lin}, $Z_j(x) := \sum_{i=1}^{n}  Z_{j,i}(x)$ is a sub-Gaussian process on $[0, T]^d$ with metric
    \begin{equation*}
        \rho_j(x, x\pr)^2 := \sigma^2 \sum_{i=1}^{n} \br{e_j\tr \br{b_i(x)-b_i(x\pr)}}^2
        \eqfs
    \end{equation*}
    As $z \mapsto \psi(z) K(z)$ is smooth, it is Lipschitz-continuous with a constant $\tilde L\in\Rpp$ depending only on $\ell, d, K$.
    If $\max(\euclOf{x_i - x}, \euclOf{x_i - x\pr}) \leq h$, then
    \begin{equation*}
        \abs{e_j\tr \br{b_i(x)-b_i(x\pr)}} \leq c \tilde L h^{-1} \euclOf{x - x\pr}
        \eqfs
    \end{equation*}
    If $\min(\euclOf{x_i - x}, \euclOf{x_i - x\pr}) \leq h \leq \max(\euclOf{x_i - x}, \euclOf{x_i - x\pr})$, then
    \begin{equation*}
        \abs{e_j\tr \br{b_i(x)-b_i(x\pr)}} \leq  c \tilde L
        \eqfs
    \end{equation*}
    If $\min(\euclOf{x_i - x}, \euclOf{x_i - x\pr}) \geq h$, then
    \begin{equation*}
        \abs{e_j\tr \br{b_i(x)-b_i(x\pr)}} = 0
        \eqfs
    \end{equation*}
    Hence, counting the number of indices with $x_i$ in a ball of radius $h$ yields
    \begin{equation*}
        \rho_j(x, x\pr)
        \leq
        c \sqrt{\const{cvr} n h^d T^{-d}}  \tilde L \sigma \min\brOf{1,\, \frac{\euclOf{x - \tilde x}}{h}}
        \eqfs
    \end{equation*}
    Now, \cref{lmm:chaining} and \cref{lmm:covernum} with $T\geq 2h$ yield
    \begin{equation*}
        \EOf{\sup_{x\in[0, T]^d} \br{Z_j(x) - Z_j(x_0)}^2} \leq c_{d,\ell,K} \const{cvr} n h^d T^{-d} \sigma^2 \log(T/h)
    \end{equation*}
    for an arbitrary $x_0\in[0,T]^d$.
    Using independence of $\noise_i$, we have
    \begin{equation*}
        \EOf{Z_j(x_0)^2}
        \leq
        \sum_{i=1}^{n} (e_j\tr b_i(x))^2 \EOf{\epsilon_i^2}
        \leq
        c_{d,\ell,K} \const{cvr} n h^d T^{-d} \sigma^2
        \eqfs
    \end{equation*}
    Hence, with $\log(T/h)\geq c$ as $T\geq 2h$, we get
    \begin{equation*}
        \EOf{\sup_{x\in[0, T]^d} Z_j(x)^2} \leq c_{d,\ell,K} \const{cvr} n h^d T^{-d} \sigma^2 \log(T/h)
        \eqfs
    \end{equation*}
    We finally obtain a bound on the variance term \eqref{eq:lp:sup:var},
    \begin{align*}
        \EOf{\sup_{x\in[0,T]^d}\abs{\sum_{i=1}^{n} w_i(x) \epsilon_i}^2}
        &\leq
        c_{d,\ell} \br{\frac{T^d}{nh^{d+s}} \const{egv}}^2  \EOf{\sup_{x\in[0, T]^d} \abs{\sum_{i=1}^{n}  Z_{j,i}(x)}^2}
        \\&\leq
        c_{d,\ell,K} \br{\frac{T^d}{nh^{d+s}} \const{egv}}^2  \const{cvr} n h^d T^{-d} \sigma^2 \log(T/h)
        \\&\leq
        c_{d,\ell,K}  \const{cvr} \const{egv}^2 \frac{\sigma^2T^d}{nh^{d+2s}} \log(T/h)
        \eqfs
    \end{align*}
    \textbf{Together.} Combining the bias and variance term yields
    \begin{equation*}
        \EOf{\supNormOf{\hat g - D_{\mo v} f}^2}
        \leq
        c_{d,\ell,K}  \br{\br{\const{egv} \const{cvr} L h^{\beta-s}}^2 + \const{cvr} \const{egv}^2 \frac{\sigma^2T^d}{nh^{d+2s}} \log(T/h)^2}
        \eqfs
    \end{equation*}
\end{proof}
\begin{proof}[Proof of \cref{prp:uniformgrid}]
    \assuRef{StrictKernel} directly implies \assuRef{Kernel} with $\const{ker}$.

    For \assuRef{Cover}, we note that $\ball^d(z, r)$ is a subset of a hypercube with side length $2r$. For each dimension, this hypercube can only contain $2n_0 r/ T + 1$ points separated by a distance $T/n_0$. Hence
    \begin{equation*}
        \frac{1}{n} \sum_{i=1}^{n} \indOfOf{\ball^d(x, r)}{x_i}
        \leq
        \begin{cases}
            \frac1n & \text{if } 2n_0 r/ T < 1\eqcm\\
            \frac1n \br{2n_0 r/ T + 1}^d  & \text{otherwise}\eqfs
        \end{cases}
    \end{equation*}
    If $2n_0 r/ T \geq 1$,
    \begin{align*}
        \frac1n \br{2n_0 r/ T + 1}^d
        \leq
        \frac1n \br{4n_0 r/ T}^d
        \leq
        4^d \br{r/ T}^d
        \eqfs
    \end{align*}
    Hence, we can set $\const{cvr} = 4^d$.

    Now consider \assuRef{Eigenvalue}. Let $v\in\mb D_d$ and $x_0\in[0,T]^d$. We want to find a lower bound on $v\tr B(x_0) v$. Let $h\in\Rpp$ be the bandwidth. Denote $z(x) = (x-x_0)/h$ and $z_i := z(x_i)$. Denote the square with side length $2h$ around $x_0$ as $\mc S(x_0) := \mc S_h(x_0) := \bigtimes_{k=1}^d[\Pi_k x_0-h, \Pi_k x_0+h]$. Denote indices of the grid points in $\mc S(x_0)$ as $\mc I(x_0) := \setByEleInText{i\in\nnset n}{x_i\in\mc S(x_0)}$.
    From the definition of $B(x_0)$, \eqref{eq:lp:B}, we obtain
    \begin{align*}
        v\tr B(x_0) v
        &=
        \frac{T^d}{nh^d} \sum_{i=1}^{n} \br{v\tr\psi\brOf{z_i}}^2 K(\euclOf{z_i})
        \\&=
        \frac{T^d}{nh^d} \sum_{i \in \mc I(x_0)} \br{v\tr\psi\brOf{z_i}}^2 K(\euclOf{z_i})
        \eqfs
    \end{align*}
    The $x_i$ form a grid of side length $T/n_0$ in $\mc S(x_0)$, which has a volume of $(2h)^d$. Assume $2h \geq T/n_0$.
    Then, the Riemann sum approximation for an $L$-Lipschitz continuous function $g\colon\mc S(x_0)\to\R$ yields
    \begin{align*}
        \abs{\frac{(2h)^d}{\# \mc I(x_0)} \sum_{i\in\mc I(x_0)} g(x_i) - \int_{\mc S(x_0)} g(x) \dl x}
        &\leq
        \sum_{i\in\mc I(x_0)} \int_{\mc X_i(x_0)} \abs{g(x_i) - g(x)} \dl x
        \\&\leq
        c_{d} L \frac{T}{n_0} h^d
        \eqcm
    \end{align*}
    where $\mc X_i(x_0)$, $i\in\mc I(x_0)$ form a suitable partition of $\mc S(x_0)$ with $x_i \in\mc X_i(x_0)$ and $\diam(\mc X_i(x_0)) \leq c_d T/n_0$.
    Applying this to $g_h(x) := \br{v\tr\psi(z(x))}^2 K(\euclOf{z(x)})$, we obtain
    \begin{equation*}
        \abs{\frac{n h^d (2h)^d}{T^d \# \mc I(x_0)} v\tr B(x_0) v - \int_{\mc S(x_0)} g_h(x) \dl x} \leq
        c_{d} \frac{T}{n_0} h^d \sup_{x\in[0,T]^d} \opNormOf{D g_h(x)}
        \eqfs
    \end{equation*}
    As $g_h(x) = g_1(x/h)$, with Lipschitz-continuity of $K$ by \assuRef{StrictKernel},
    \begin{equation*}
        \sup_{x\in\mc S_h(x_0)} \opNormOf{D g_h(x)}
        =
        h^{-1} \sup_{x\in\mc S_1(x_0)} \opNormOf{D g_1(x)}
        \leq
        c_{\ell,K} h^{-1}
        \eqfs
    \end{equation*}
    Note that $\lfloor(2h n_0/T)\rfloor^d \leq \# \mc I(x_0) (\lfloor(2h n_0/T)\rfloor + 1)^d$. Thus,
    \begin{equation*}
        \abs{v\tr B(x_0) v - h^{-d} \int_{\mc S(x_0)} g_h(x) \dl x} \leq
        c_{d,\ell,K} \frac{T}{n_0 h}
        \eqfs
    \end{equation*}
    This means, if $T/(n_0 h)$ is small enough, it only remains to show a positive lower bound on
    \begin{equation*}
        h^{-d} \int_{\mc S(x_0)} g_h(x) \dl x = \int_{\R^d} \br{v\tr\psi\brOf{x}}^2 K(\euclOf{x}) \dl x
    \end{equation*}
    With the same arguments as in the proof of \cite[Lemma 1.4]{Tsybakov09Introduction} and using \assuRef{StrictKernel}, one can show
    \begin{equation*}
        \int_{\R^d} \br{v\tr\psi\brOf{x}}^2 K(\euclOf{x}) \dl x \geq c_{d,\ell,K} > 0
        \eqfs
    \end{equation*}
\end{proof}

    \section{Polynomial Interpolation}\label{app:sec:interpol}
\subsection{Univariate Polynomial Interpolation}
\begin{proposition}
    Let $\ell \in \N_0$.
    Let $u\in \mc D^{\ell+1}(\R, \R)$. Let $(t_i)_{i\in\nnzset\ell}$ with $t_0 < \dots < t_\ell$. Let $p\in\Poly 1\ell$ be the (Lagrange) interpolation polynomial of $(t_i, u(t_i))_{i\in\nnzset\ell}$. Let $k\in\nnzset{\ell}$. Then,
    \begin{equation*}
        \sup_{t\in[t_0,t_\ell]} \abs{D^k u(t) - D^k p(t)} \leq \frac{1}{k!(\ell+1-k)!} \sup_{t\in[t_0,t_\ell]} \abs{D^{k} \omega(t)} \sup_{t\in[t_0,t_\ell]} \abs{D^{\ell+1} u(t)}
        \eqcm
    \end{equation*}
    where $\omega(t) := \prod_{i=0}^\ell(t_i-t)$.
\end{proposition}
\begin{proof}
    \cite[Theorem 2 and 3]{howell91}.
\end{proof}
%
In following lemma, we apply this result to the uniform grid $t_i = i\stepsize$.
\begin{lemma}\label{coro:univariatePolyApprox}
    Let $\ell\in\N_0$.
    Let $\stepsize \in \Rpp$.
    Let $u\in \mc D^{\ell+1}(\R, \R)	$.
    Let $p\in\Poly 1\ell$ be the interpolation polynomial of $(i\stepsize, u(i\stepsize))_{i\in\nnzset{\ell}}$. Let $k\in\nnzset\ell$. Then
    \begin{equation*}
        \sup_{t\in[0,\ell\stepsize]} \abs{D^k u(t) - D^k p(t)} \leq c_{k,\ell} \stepsize^{\ell+1-k} \abs{D^{\ell+1} u}_\infty
        \eqfs
    \end{equation*}
\end{lemma}
\subsection{Univariate Polynomial Comparison}
\begin{lemma}\label{lem:polycompare}
    Let $p_{n}, \tilde p_{n} \in \Poly{1}{\ell-1}$. Let $h_n,\delta_n\in\Rpp$. Let $\alpha_1, \dots, \alpha_\ell\in\R$ be distinct. Let $\alpha_{\ms{min}}, \alpha_{\ms{max}}\in\Rpp$.
    We allow $p_n, \tilde p_n, h_n, \delta_n$ to change with $n$ and consider $\alpha_1, \dots, \alpha_\ell, \alpha_{\ms{min}}, \alpha_{\ms{max}}$ to be fixed.
    Assume, for all $i\in\nnset \ell$,
    \begin{equation*}
        \abs{p_n(\alpha_i h_n) - \tilde p_n(\alpha_i h_n)} \asymleq \delta_n\eqfs
    \end{equation*}
    Then
    \begin{equation*}
        \sup_{t\in[-\alpha_{\ms{min}}h_n, \alpha_{\ms{max}}h_n]} \abs{D^k p_n(t) - D^k  \tilde p_n(t)} \asymleq h_n^{-k} \delta_n
    \end{equation*}
    for $k = \nnzset{\ell-1}$.
\end{lemma}
\begin{proof}
    Let $p_n(t) - \tilde p_n(t) =: q(t) = \sum_{k=1}^{\ell} a_{k} t^{k-1}$. Let $y_i := q(\alpha_i h_n)$. Then $a = V^{-1} y$, where $V\in\R^{\ell\times\ell}$ is the Vandermonde matrix $V_{i,j} = (\alpha_i h_n)^{j-1}$. By the inverse formula of the Vandermonde matrix, $\abs{(V^{-1})_{i,j}} \asymleq h_n^{1-i}$.
    Thus,
    \begin{equation*}
        \abs{a_i} = \abs{\sum_{j=1}^\ell (V^{-1})_{i,j} y_j} \asymleq h_n^{1-i} \delta_n
        \eqfs
    \end{equation*}
    Furthermore, for $\alpha\in[-\alpha_{\ms{min}}, \alpha_{\ms{max}}]$,
    \begin{align*}
        \abs{D^k q(\alpha h_n)}
        &=
        \abs{\sum_{i=1}^{\ell-k} a_{i+k} \br{\alpha h_n}^{i-1} \prod_{s=1}^k (i + s) }
        \\&\asymleq
        \sum_{i=1}^{\ell-k} \abs{\alpha}^{i-1} h_n^{-k} \delta_n
        \\&\asymleq
        h_n^{-k} \delta_n
        \eqfs
    \end{align*}
\end{proof}
    \subsection{Multivariate Polynomial Interpolation}
Recall \cref{def:poly} for the objects $\psi(x)$, $\Psi(\mo x)$, and $I(\mo x, \mo y, x)$.
\begin{lemma}[Linear Transformation]\label{lmm:poly:linearbase}
    Let $\ell,d\in\N$.
    Set $N := \dim(\Poly d\ell)$.
    Let $b\in \R^{d}$. Let $B\in \R^{d\times d}$.
    Then there is $A \in \R^{N\times N}$ such that $\psi(B x + b) = A \psi(x)$ for all $x\in\R^{d}$.
    Furthermore, if $B$ is invertible, then so is $A$.
\end{lemma}
\begin{proof}
    First we need to define $\tilde\psi_\ell\colon\R^d \to\R^M$ as a variant of the monomial transformation $\psi_\ell\colon\R^d \to\R^N$, where commutativity is not applied, i.e., $\tilde\psi_\ell(x)$ contains a dimension for $x_1 x_2$ and a different one for $x_2 x_1$. The \textit{monomial transform without commutativity} $\tilde\psi_\ell$ can be expressed using the \textit{Kronecker product}, see, e.g., \cite[Chapter 4]{Horn1994}.

    Denote by $\otimes$ the operator for the Kronecker product.
    For $x\in\R^d$, $k\in\N_0$, define $x^{\otimes k} \in \R^{d^k}$ as follows:
    Let $x^{\otimes 0} := 1$. For $k\in\N$, define $x^{\otimes k} = x^{\otimes (k-1)}\otimes x$.
    Denote
    \begin{equation*}
        M := \sum_{k=0}^{\ell} d^k = \frac{d^{\ell+1}-1}{d-1}
        \eqfs
    \end{equation*}
    Define the monomial transform without commutativity $\tilde\psi_\ell\colon\R^d \to\R^M$ as
    \begin{equation*}
        \tilde\psi_\ell(x) := \begin{pmatrix}
            x^{\otimes 0}\\
            x^{\otimes 1}\\
            \vdots\\
            x^{\otimes \ell}
        \end{pmatrix}
        \eqfs
    \end{equation*}
    Our intermediate goal is to find a matrix $\tilde A\in\R^{M\times M}$ with the property
    \begin{equation*}
        \tilde\psi(B x + b) = \tilde A \tilde\psi(x)
        \eqfs
    \end{equation*}
    For this, we first recall some properties of the Kronecker product:
    The Kronecker product is bilinear. In particular,
    \begin{equation*}
        M_1 \otimes (M_2 + M_3) = M_1 \otimes M_2 + M_1 \otimes M_3\ \text{and}\
        (M_1 + M_2) \otimes M_3  = M_1 \otimes M_3 + M_2 \otimes  M_3
    \end{equation*}
    for compatible matrices $M_1, \dots, M_3$.
    Moreover, The mixed-product property of the Kronecker product states
    \begin{equation*}
        \br{M_1 \otimes M_2}\br{M_3 \otimes M_4} = \br{M_1 M_3} \otimes \br{M_2 M_4}
    \end{equation*}
    for compatible matrices $M_1, \dots, M_4$.
    The Kronecker product is not commutative, but $M_1 \otimes M_2$ and $M_2 \otimes M_1$ are \textit{permutation equivalent}:  There are permutation matrices $P, Q$ such that
    \begin{equation*}
        M_1 \otimes M_2 = P (M_2 \otimes M_1) Q
    \end{equation*}
    for compatible matrices $M_1, M_2$.

    Now, using the mixed-product property, we obtain by induction
    \begin{equation}\label{eq:kroneck:lin}
        \br{B x}^{\otimes k} = B^{\otimes k} x^{\otimes k}
        \eqfs
    \end{equation}
    Next, the bilinearity and the permutation equivalence imply, also using induction, there are matrices $\bar A_j \in\R^{d^k \times d^j}$
    \begin{equation}\label{eq:kroneck:trans}
        \br{b + x}^{\otimes k} = \sum_{j=0}^{k} \bar A_j x^{\otimes j}
        \eqfs
    \end{equation}
    Thus, there are matrices $\tilde A_0, \tilde A_1\in\R^{M \times M}$ such that
    \begin{equation*}
        \tilde \psi(x + b) = \tilde A_0 \tilde \psi(x) \qquad\text{and}\qquad \tilde \psi(B x) = \tilde A_1 \tilde \psi(x)
    \end{equation*}
    and, hence,
    \begin{equation*}
        \tilde \psi(B x + b) = \tilde A \tilde \psi(x)
    \end{equation*}
    with $\tilde A = \tilde A_0 \tilde A_1$.
    From \eqref{eq:kroneck:lin}, we can see that $\tilde A_1$ is block diagonal with blocks $B^{\otimes k}$.
    From \eqref{eq:kroneck:trans}, we can see that $\tilde A_0$ is block triangular with identity  blocks $I_{d^k}\in\R^{d^k \times d^k}$ on the diagonal.
    Hence, $\tilde A$ is block triangular with blocks $B^{\otimes k}$ on the diagonal.
    For the determinant of the Kronecker product, we have
    \begin{equation}\label{eq:kronecker:det}
        \det(M_1 \otimes M_2) = \det(M_1)^{d_2} \det(M_2)^{d_1}
    \end{equation}
    for $M_1 \in\R^{d_1 \times d_1}$ and $M_2 \in\R^{d_2 \times d_2}$.
    Thus,
    \begin{equation*}
        \det\brOf{\tilde A} = \prod_{k = 0}^\ell \det\brOf{B^{\otimes k}} = \det(B)^{\kappa}
        \eqcm
    \end{equation*}
    where
    \begin{equation*}
        \kappa := \sum_{k=0}^{\ell} kd^{k-1} = \begin{cases}
            \frac{\ell (1 + \ell)}{2} & \text{ if } d = 1\\
            \frac{1 + \ell d^{\ell + 1} - (\ell + 1) d^\ell}{(d - 1)^2} & \text{ otherwise.}
        \end{cases}
    \end{equation*}
    To see this, show inductively using \eqref{eq:kronecker:det} that
    \begin{equation*}
        \det(B^{\otimes k})  :=  \det(B)^{kd^{k-1}}
        \eqfs
    \end{equation*}
    Now, we have the existence of $\tilde A\in\R^{M\times M}$ with
    \begin{equation*}
        \tilde \psi(B x + b) = \tilde A \tilde \psi(x)
    \end{equation*}
    and the property that if $B$ is invertible so is $\tilde A$. In the final step, we covert from $\tilde \psi$ to $\psi$ via a linear transformation of full rank:
    Let $Z\in\R^{N\times M}, \tilde Z\in\R^{M\times N}$ be linear mappings that fulfill
    \begin{equation*}
        \psi_\ell(x) = Z \tilde\psi_\ell(x)
        \qquad\text{and}\qquad
        \tilde\psi_\ell(x) = \tilde Z \psi_\ell(x)
    \end{equation*}
    for all $x\in\R^d$.
    Note that the rank of $Z$ and of $\tilde Z$ is $N$, i.e., both matrices have full rank.
    Set
    \begin{equation*}
        A := Z \tilde A \tilde Z
        \eqfs
    \end{equation*}
    Then
    \begin{equation*}
        \psi(B x + b) = Z \tilde\psi(B x + b) = Z \tilde A \tilde \psi(x) = Z \tilde A \tilde Z \psi(x)= A \psi(x)
        \eqfs
    \end{equation*}
    Moreover, if $B$ is invertible, so is $\tilde A$. Then, as $\tilde A$, $\tilde Z$, and $Z$ have full rank, $A$ also has full rank and is invertible.
\end{proof}
\begin{lemma}[Linear Transformation]\label{lmm:poly:linear}
    Let $\ell,\dm x,\dm y\in\N$.
    Set $N := \dim(\Poly{\dm x}\ell)$.
    Let $b\in \R^{\dm x}$. Let $B\in \R^{\dm x\times \dm x}$.
    Let $\mo x \in (\R^{\dm x})^N$ and $\mo y \in (\R^{\dm y})^N$.
    Let $x\in\R^{\dm x}$.
    Assume $B$ and $\Psi(\mo x)$ are invertible. Then
    \begin{enumerate}[label=(\roman*)]
        \item\label{lmm:poly:linear:inv}
        $\Psi(B \mo{x} + b)$ is invertible,
        \item\label{lmm:poly:linear:inter}
        \begin{equation*}
            I(\mo{x}, \mo{y}, x) = I(B \mo{x} + b, \mo{y}, B x + b)
            \eqcm
        \end{equation*}
        \item\label{lmm:poly:linear:psi}
        \begin{equation*}
            \psi(x)\tr \Psi(\mo{x})^{-1} = \psi(B x + b)\tr \Psi(B \mo{x} + b)^{-1}
            \eqfs
        \end{equation*}
    \end{enumerate}
\end{lemma}
\begin{proof}
    Part \ref{lmm:poly:linear:inv} follows directly from \cref{lmm:poly:linearbase}.

    As the interpolation is done componentwise, it is enough to show the results for the case $\dm y = 1$. Let us assume $\dm y = 1$.
    The function $g \colon \R^{\dm x}\to \R,\, x \mapsto I(\mo{x}, \mo{y}, x)$ is the unique interpolation polynomial of $(\mo x, \mo y)$ with degree at most $\ell$. The mapping $\tilde g\colon \R^d \to \R, x\mapsto I(B \mo{x} + b, \mo{y}, B x + b)$ is a polynomial of degree at most $\ell$ and fulfills $\tilde g(x_k) = y_k$ for $k\in\nnset N$. Thus $\tilde g = g$. This shows \ref{lmm:poly:linear:inter}.

    As \ref{lmm:poly:linear:inter} is true for all $\mo y$ and $I(\mo x, \mo y, x) = \psi(x)\tr\Psi(\mo x)^{-1} \mo y$, \ref{lmm:poly:linear:psi} follows.
\end{proof}
\begin{lemma}[Approximation]\label{lmm:polyinterpol}
    Let $\ell,d\in\N$.
    Set $N := \dim(\Poly d\ell)$.
    Let $\mo x\in(\R^d)^N$. Let $g\colon\R^d\to\R$ be $(\ell+1)$-times continuously differentiable and set $L := \sup_{x \in \R^d} \normop{D^{\ell+1} g(x)}$.
    Let $\mo y = (g(x_1), \dots, g(x_N))$.
    Then
    \begin{equation}\label{eq:interpolbound}
        \sup_{x \in \hull_0(\mo x)} \euclOf{g(x) - I_\ell(\mo x, \mo y, x)} \leq L C_{\ell}(\mo{x}) \diam(\mo x)^{\ell+1}
        \eqcm
    \end{equation}
    where $C_{\ell}(\mo{x})$ is defined as follows:
    For $k\in\nnset N$, let $e_k\in\R^N$ be the $k$-th standard unit vector.
    Define
    \begin{equation*}
        C_{\ell}(\mo x) :=
        \begin{cases}
            \frac1{(\ell+1)!}\sum_{k=1}^N\sup_{x\in\hull_0(\mo{x})} \abs{I(\mo{x}, e_k, x)} &\text{ if } \Psi(\mo x) \text{ is invertible,}\\
            \infty &\text{ otherwise.}
        \end{cases}
    \end{equation*}
\end{lemma}
\begin{proof}
    \cite[Theorem 2]{ciarlet72}.
\end{proof}
\begin{lemma}\label{lmm:snake:poly:constbound}
    Let $\ell,d\in\N$.
    Set $N := \dim(\Poly d\ell)$.
    If $\Psi(\mo x)$ is invertible,
    \begin{equation*}
        C_{\ell}(\mo x) \leq \frac {N^{\frac32}}{(\ell+1)!} \opNormOf{\Psi(\eta_{\mo x}(\mo x))^{-1}}
        \eqfs
    \end{equation*}
\end{lemma}
\begin{proof}
    By \cref{lmm:poly:linear}, we have
    \begin{equation*}
        \sup_{x\in\hull_0(\mo{x})} \abs{I({\mo x}, e_j, x)}
        =
        \sup_{x\in\hull_0(\mo{x})} \abs{I\brOf{\eta_{\mo x}(\mo x), e_j, \eta_{\mo x}(x)}}
        \leq
        \sup_{x\in\ball^d(0, 1)} \abs{I\brOf{\eta_{\mo x}(\mo x), e_j, x}}
        \eqfs
    \end{equation*}
    For $x\in\ball^d(0, 1)$,
    \begin{equation*}
        \abs{I(\eta_{\mo x}(\mo x), e_j, x)}
        \leq
        \euclof{\psi(x)} \opNormOf{\Psi(\eta_{\mo x}(\mo x))^{-1}} \euclOf{e_j}
        \leq
        \sqrt{N} \opNormOf{\Psi(\eta_{\mo x}(\mo x))^{-1}}
        \eqfs
    \end{equation*}
    Thus,
    \begin{align*}
        C_{\ell}(\mo x)
        &\leq
        \frac1{(\ell+1)!}\sum_{j=1}^N \sqrt{N} \opNormOf{\Psi(\eta_{\mo x}(\mo x))^{-1}}
        \\&\leq
        \frac{N^{\frac32}}{(\ell+1)!}\opNormOf{\Psi(\eta_{\mo x}(\mo x))^{-1}}\eqfs
    \end{align*}
\end{proof}
\begin{lemma}[Perturbed normalization]\label{lmm:perturb:norm}
	Let $N,d\in\N$. Let $\mo x = (x_1, \dots, x_N), \mo z = (z_1, \dots, z_N) \in \br{\R^{d}}^N$.
	Assume
	\begin{equation*}
		\gamma := \max_{k\in\nnset N}\euclOf{x_k - z_k} < \frac12 \diam(\mo x)
		\eqfs
	\end{equation*}
	Then
	\begin{equation*}
		\max_{k\in\nnset N} \euclOf{\eta_{\mo x}(x_k) -  \eta_{\mo z}(z_k)} \leq \frac{4\gamma}{\diam(\mo x) - 2\gamma}
		\eqfs
	\end{equation*}
\end{lemma}
\begin{proof}
	By \cref{def:normalization},
	\begin{align*}
		 \eta_{\mo x}(x_k) -  \eta_{\mo z}(z_k)
		 &=
		 \frac{x_k - \frac1N\sum_{i=1}^N x_i}{\diam(\mo x)}
		 -
		 \frac{z_k - \frac1N\sum_{i=1}^N z_i}{\diam(\mo z)}
		 \\&=
		 \frac{x_k - \frac1N\sum_{i=1}^N x_i}{\diam(\mo x)}
		 -
		 \frac{x_k - \frac1N\sum_{i=1}^N x_i}{\diam(\mo z)}
		 +
		 \frac{x_k - \frac1N\sum_{i=1}^N x_i}{\diam(\mo z)}
		 -
		 \frac{z_k - \frac1N\sum_{i=1}^N z_i}{\diam(\mo z)}
		 \eqfs
	\end{align*}
	Using the definition of $\gamma$, we have
	\begin{align*}
		\euclOf{\br{x_k - \frac1N\sum_{i=1}^N x_i} - \br{z_k - \frac1N\sum_{i=1}^N z_i}} &\leq 2\gamma
		\eqcm\\
		\abs{\diam(\mo x) - \diam(\mo z)} &\leq 2\gamma
		\eqfs
	\end{align*}
	As we assume $2\gamma < \diam(\mo x)$, we obtain
	\begin{align*}
		\euclOf{\eta_{\mo x}(x_k) -  \eta_{\mo z}(z_k)}
		&\leq
		\euclOf{x_k - \frac1N\sum_{i=1}^N x_i}\abs{\frac{1}{\diam(\mo z)}-\frac{1}{\diam(\mo x)}}
		+
		\frac{2\gamma}{\diam(\mo z)}
		\\&\leq
		\diam(\mo x) \abs{\frac{1}{\diam(\mo x)-2\gamma}-\frac{1}{\diam(\mo x)}}
		+
		\frac{2\gamma}{\diam(\mo x) - 2\gamma}
		\\&=
		\frac{4\gamma}{\diam(\mo x) - 2\gamma}
		\eqfs
	\end{align*}
	Thus,
	\begin{equation*}
		\max_{k\in\nnset N} \euclOf{\eta_{\mo x}(x_k) -  \eta_{\mo z}(z_k)} \leq \frac{4\gamma}{\diam(\mo x) - 2\gamma}
		\eqfs
	\end{equation*}
\end{proof}
\begin{lemma}[Perturbed monomials]\label{lmm:perturb:monom}
	Let $d,\ell\in\N$ and $x,z\in[-1,1]^d$.
	Then
	\begin{equation*}
		 \euclOf{\psi_\ell(x) - \psi_\ell(z)} \leq \sqrt{N} \ell \euclOf{x-z}
		 \eqcm
	\end{equation*}
	where $C_{d,\ell}\in\Rpp$ is a constant depending only on $d,\ell$.
\end{lemma}
\begin{proof}
	We first need following simple result:
	For $a_1, a_2,b_1, b_2 \in [-1, 1]$, we have 
	\begin{equation}
		\abs{a_1a_2 - b_1b_2} \leq \abs{a_1a_2 - a_1b_2} +\abs{a_1b_2 - b_1b_2} \leq \abs{a_2 - b_2} +\abs{a_1 - b_1}
		\eqfs 
	\end{equation}
	By induction, for $k\in\N$ and $a,b\in[-1, 1]^k$, we obtain
	\begin{equation}\label{eq:diffpolyeq}
		\abs{\prod_{i=1}^k a_i - \prod_{i=1}^k b_i} \leq \sum_{i=1}^k \abs{a_i - b_i}
		\eqfs
	\end{equation}
	
	Now, denote the vector components $(x_1, \dots, x_d) = x$ and $(z_1, \dots, z_d) = z$.
	Recall \cref{not:poly} for monomials.
	Let $i \in \nnset d$ and $\alpha \in \Nn^d$ with $\abs{\alpha} \leq \ell$. Then, as $x,z\in[-1, 1]^d$, we obtain from \eqref{eq:diffpolyeq},
	\begin{equation*}
		\abs{x^\alpha - z^\alpha} \leq \sum_{i=1}^d \alpha_i \abs{x_i - z_i}
		\eqfs
	\end{equation*}
	Thus, using Jensen's inequality
	\begin{align*}
		\euclOf{\psi_\ell(x) - \psi_\ell(z)}^2
		&\leq
		\sum_{\alpha\in\Nn^d,\abs{\alpha}\leq\ell} \br{\sum_{i=1}^d \alpha_i \abs{x_i - z_i}}^2
		\\&\leq
		\sum_{\alpha\in\Nn^d,\abs{\alpha}\leq\ell} \abs{\alpha} \sum_{i=1}^d \alpha_i \abs{x_i - z_i}^2
		\\&\leq
		N \ell^2 \euclOf{x-z}^2 
		\eqfs
	\end{align*}
	Hence,
	\begin{equation*}
		\euclOf{\psi_\ell(x) - \psi_\ell(z)} \leq \sqrt{N} \ell \euclOf{x-z}
		\eqfs
	\end{equation*}
\end{proof}
\begin{lemma}[Perturbed inverse matrix]\label{lmm:perturb:inverse}
	Let $d\in\N$. Let $A,B \in\R^{d\times d}$. Assume that $A$ is invertible and
	\begin{equation*}
		\omega := \opNormOf{A-B} \opNormOf{A^{-1}} < 1
		\eqfs
	\end{equation*}
	Then $B$ is invertible and
	\begin{equation*}
		\opNormOf{A^{-1}-B^{-1}}
		\leq
		\frac{\omega }{1 - \omega}  \opNormOf{A^{-1}}
		\eqfs
	\end{equation*}
\end{lemma}
\begin{proof}
	For $k\in \nnset d$, let $\sigma_k(A)$ be the $k$-th singular value of $A$ in increasing order, i.e.,
	\begin{equation*}
		\sigma_1(A) \leq \dots \leq \sigma_d(A)
		\eqfs
	\end{equation*}
	Then $\sigma_1(A) > 0$ as $A$ is invertible. Furthermore, $\opNormof{A^{-1}} = \sigma_1(A)^{-1}$.
	By Weyl's Theorem, we have
	\begin{equation*}
		\abs{\sigma_k(A) - \sigma_k(B)} \leq \opNormOf{A - B}
	\end{equation*} 
	for all $k \in \nnset d$. 
	In particular,
	\begin{equation*}
		\sigma_1(B) \geq \sigma_1(A) - \opNormOf{A - B} = \sigma_1(A) (1 - \omega) > 0
		\eqfs
	\end{equation*}
	Thus, $B$ is invertible.
	Next, we have
	\begin{align*}
		A^{-1} - B^{-1} = A^{-1} (B - A) B^{-1}
		\eqfs
	\end{align*}
	Thus,
	\begin{align*}
		\opNormOf{A^{-1}-B^{-1}}
		&\leq
		\opNormOf{A-B} \opNormOf{A^{-1}}\opNormOf{B^{-1}}
		\\&\leq 
		\opNormOf{A-B} \opNormOf{A^{-1}}\br{\sigma_1(A) (1 - \omega)}^{-1}
		\\&=
		\frac{\omega }{1 - \omega}  \opNormOf{A^{-1}}
		\eqfs
	\end{align*}
\end{proof}
\begin{lemma}\label{lmm:perturb}
	Let $d,\ell \in \N$. Set $N := \dim(\Poly d\ell)$.
	Let $\mo x = (x_1, \dots, x_N), \mo z = (z_1, \dots, z_N) \in (\R^d)^N$.
	Let $\delta_0, \gamma, s \in \Rpp$ such that
	\begin{align*}
		\opNormOf{\Psi(\eta_{\mo x}(\mo x))^{-1}}& \leq s\eqcm\\
		\delta_0 &\leq \diam(\mo x)\eqcm\\
		\max_{k\in\nnset N}\euclOf{x_k - z_k} &\leq \gamma
		\eqfs
	\end{align*}
	Assume 
	\begin{equation}\label{eq:perturb:cond}
		  \max\brOf{4, 16 N \ell s} \gamma \leq \delta_0
		  \eqfs
	\end{equation}
	Then $\Psi(\eta_{\mo z}(\mo z))$ is invertible and
	\begin{equation*}
		\opNormOf{
			\Psi(\eta_{\mo z}(\mo z))^{-1}
			-
			\Psi(\eta_{\mo x}(\mo x))^{-1}
		}
		\leq16 N \ell \frac{s^2\gamma}{\delta_0}
		\eqfs
	\end{equation*}
\end{lemma}
\begin{proof}
	Using \cref{lmm:perturb:monom}, we obtain 
	\begin{align*}
		\opNormof{\Psi(\eta_{\mo z}(\mo z)) - \Psi(\eta_{\mo x}(\mo x))}
		&\leq
		\sqrt{N} \max_{k\in\nnset N}
		\euclOf{\psi(\eta_{\mo z}(z_k)) - \psi(\eta_{\mo x}(x_k))}
		\\&\leq
		N \ell \max_{k\in\nnset N}
		\euclOf{\eta_{\mo z}(z_k) - \eta_{\mo x}(x_k)}
		\eqfs
	\end{align*}
	Using \eqref{eq:perturb:cond}, we have $\gamma \leq \frac14 \delta_0 < \frac12 \delta_0 \leq \frac12 \diam(\mo x)$. Thus, \cref{lmm:perturb:norm} yields
	\begin{equation*}
		\opNormof{\Psi(\eta_{\mo z}(\mo z)) - \Psi(\eta_{\mo x}(\mo x))}
		\leq
		N \ell \frac{4\gamma}{\diam(\mo x) - 2\gamma}
		\leq
		N \ell \frac{4\gamma}{\delta_0 - 2\gamma}
		\eqfs
	\end{equation*}
	Set
	\begin{equation*}
		\omega := 8 s N \ell \frac{\gamma}{\delta_0} 
		\eqfs
	\end{equation*}
 	By \eqref{eq:perturb:cond}, we have $\omega \leq \frac12$. Thus, 
 	\begin{equation*}
 		s N \ell \frac{4\gamma}{\diam(\mo x) - 2\gamma} \leq \omega < 1
 		\eqcm
 	\end{equation*}
 	and we can apply \cref{lmm:perturb:inverse}, which implies that $\Psi(\eta_{\mo z}(\mo z))$ is invertible and
 	\begin{equation*}
 		\opNormOf{
 			\Psi(\eta_{\mo z}(\mo z))^{-1}
 			-
 			\Psi(\eta_{\mo x}(\mo x))^{-1}
 		}
 		\leq \frac{\omega s}{1-\omega}
 		\leq 16 s^2 N \ell \frac{\gamma}{\delta_0} 
 		\eqfs
 	\end{equation*}
\end{proof}
    \section{Ordinary Differential Equations and Smoothness}\label{app:sec:ode}
\subsection{Difference Bound}
\begin{lemma}\label{lmm:diffinequcomp}
    Let $f\colon\R\to\R$ be Lipschitz continuous. Let $x \in\R$. Set $u : = U(f, x, \cdot)$. Let $v \colon \R \to \R$ be differentiable with $v(0) \leq x$ and $\dot v(t) \leq f(v(t))$ for all $t\in\Rp$. Then
    \begin{equation*}
        v(t) \leq u(t)
    \end{equation*}
    for all $t\in\Rp$.
\end{lemma}
\begin{proof}
    Assume there is $t_1 \in\Rpp$ with $v(t_1) > u(t_1)$. Then the first intersection time $t_0 := \inf\setByEleInText{t\in\Rp}{v(t) > u(t)}$ exists. As $v$ and $u$ are continuous, we have $v(t_0) = u(t_0)$.
    Let $t_2 = \sup\setByEleInText{t\in\R_{\geq t_0}}{\forall s\in[t_0, t]\colon v(s) \geq u(s)}$. Because of the existence of $t_1$ and continuity of $u$ and $v$, we have $t_2 > t_0$.
    Let $L\in\Rpp$ be the Lipschitz constant of $f$.
    Set $\delta(t) := v(t) - u(t)$.  Let $t\in[t_0, t_2]$. Then
    \begin{align*}
        \delta(t)
        &=
        \int_{t_0}^{t} \dot \delta(s) \dl s
        \\&\leq
        \int_{t_0}^{t} f(v(s)) - f(u(s)) \dl s
        \\&\leq
        L \int_{t_0}^{t} \delta(s) \dl s
        \eqfs
    \end{align*}
    As $\delta$ is continuous, there is $t_{\ms{max}} \in\argmax_{t\in[t_0, \min(t_0+1/L, t_2)]} \delta(t)$. By definition of $t_0$ and $t_{\ms{max}}$, we have $\delta(t_{\ms{max}}) > 0$. Thus, using $\delta(t_0) = 0$, the continuity of $\delta$, and $t_{\ms{max}} - t_0 \leq 1/L$, we obtain
    \begin{equation*}
        \delta(t_{\ms{max}}) \leq L \int_{t_0}^{t_{\ms{max}}} \delta(s) \dl s < L \br{t_{\ms{max}} - t_0} \delta(t_{\ms{max}}) \leq \delta(t_{\ms{max}})
        \eqcm
    \end{equation*}
    a contradiction.
\end{proof}
\begin{lemma}\label{lem:odeDiffBound}
    Let $d\in\N$, $L\in\Rpp$. Let $f,g\in\Sigma^{d\to d}(1, L)$. Let $x_0\in\R^d$. Denote $u := U(f, x_0, \cdot)$, $v := U(g, x_0, \cdot)$. Then, for $t\in\Rp$,
	\begin{equation}\label{eq:diffbound}
		\euclOf{u(t) - v(t)} \leq \frac{\exp(Lt) - 1}L \supNormOf{f-g}
		\eqfs
	\end{equation}
	In particular, if $t \in[0,\frac1L]$, then
	\begin{equation*}
			\euclOf{u(t) - v(t)} \leq 2 t \supNormOf{f-g}
            \eqfs
	\end{equation*}
\end{lemma}
\begin{proof}
	Set $C:= \abs{f-g}_\infty$.
    We represent the right-hand side of \eqref{eq:diffbound} with the solution $z$ of an ODE:
	Define $z\colon\Rp\to\R$ as the solution of the ODE $\dot z(t) = L z(t) + C$ with $z(0) = 0$. Then,
	\begin{equation*}
		z(t) =  \frac CL \br{\exp(Lt) - 1}
		\eqfs
	\end{equation*}
    Next consider the left-hand side of \eqref{eq:diffbound}.
	Define $w(t) := \euclOf{u(t) - v(t)}$. Then $w(0) = 0$ and
	\begin{align*}
	 	w(t)
	 	&=
	 	\euclOf{\int_0^t f(u(s)) - g(v(s)) \dl s}
	 	\\&\leq
	 	\euclOf{\int_0^t f(u(s)) - f(v(s)) \dl s} + \euclOf{\int_0^t f(v(s)) - g(v(s)) \dl s}
	 	\\&\leq
	 	L \int_0^t w(s) \dl s + C t
	 	\eqfs
	\end{align*}
    Thus, \cref{lmm:diffinequcomp} yields, $w(t) \leq z(t)$.
\end{proof}
\subsection{Smoothness of Solutions in Time}
Let $d,\beta\in\N$. Let $\indset L0{\beta} \subset \Rpp$. Assuming $f\in\bar\Sigma^{d\to d}(\beta, \indset L0\beta)$, we want to find $\indset{\tilde L}1{\beta + 1} \subset \Rpp$ such that $U(f,x,\cdot) \in\Sigma^{1\to d}(\beta+1, \indset{\tilde L}0{\beta + 1})$ with $\tilde L_0 = \infty$.

As detailed in \cite[Chapter 3]{Butcher2016}, higher order derivatives of $u$ can be expressed using tree structures. To state this result in \cref{thm:solutionderivtree} below, we first need to introduce \textit{increasing trees} (\cref{def:increasingtree}) and the derivative operator associated to such a tree (\cref{def:treederiv}).

In the following the term \textit{finite ordered set} refers to a finite set with strict total order that is denoted as $<$.
\begin{definition}[Increasing tree]\label{def:increasingtree}
    Let $S$ be a finite ordered set.
    \begin{enumerate}[label=(\roman*)]
        \item
        An \textit{increasing tree} $\mc T$ is a tuple $(S, E)$, where $E \subset \setByEleInText{(n_1, n_2)}{n_1, n_2 \in S,\, n_1 < n_2}$ with $\# E = \# S -1$. We call $S$ the \textit{nodes}, $E$ the \textit{edges}, and $\min S$ the \textit{root}.
        Furthermore, denote $\# \mc T := \# S$.
        \item
        Let $\mf T_S$ be the set of increasing trees with nodes $S$.
        \item
        We define $[\mc T_1, \dots,\mc T_k]$ as the increasing tree that connects all trees $\mc T_j$ at their roots to a new root $s$:

        Let $k\in \N_0$. Let $S_0 = \{s\}$ and $S_1, \dots, S_k$ be disjoint finite ordered sets so that $S := \bigcup_{j\in\nnzset k} S_k$ is also ordered with $s = \min S$. Let $\mc T_1 = (S_1, E_1), \dots, \mc T_k = (S_k, E_k)$ be increasing trees. Define $[\mc T_1, \dots, \mc T_k] := (S, E)$, where $E := \bigcup_{j\in\nnzset{k}} E_j$ and $E_0 := \setByEleInText{(s, \min S_j)}{j\in\nnset k}$.
    \end{enumerate}
\end{definition}
\begin{lemma}\label{lmm:numtrees}
    Let $S$ be a finite ordered set.
    Then
    \begin{equation*}
        \# \mf T_{S} = (\# S-1)!
        \eqfs
    \end{equation*}
\end{lemma}
\begin{proof}
    For $S = \{s\}$, we have $\mf T_{S} = \{(\{s\}, \emptyset)\}$. Hence, $\# \mf T_{S} = (\# S-1)!$ is true.
    Now assume  $\# S = k+1$, $\max S = s$, and $S\pr = S \setminus \{s\}$. Then
    $\mf T_{S}$ can be constructed from $\mf T_{S\pr}$, by adding the node $s$ to the node $n$ of $\mc T$ for all trees $\mc T \in \mf T_{S\pr}$ and all nodes $n \in S\pr$. Hence,
    \begin{equation*}
        \# \mf T_{S} = \# S\pr \#\mf T_{S\pr} = k\, \#\mf T_{S\pr}
        \eqfs
    \end{equation*}
    The desired result now follows by induction.
\end{proof}
\begin{definition}[Tree derivative operator]\label{def:treederiv}
    Let $S$ be a finite ordered set.
    Let $\mc T \in\mf T_S$.
    Let $k\in\N_0$.
    Assume $\mc T = [\mc T_1, \dots, \mc T_k]$.
    Let $d\in\N$.
    Let $f\in\mc D^{k}(\R^d, \R^d)$.
    Let $x\in\R^d$.
    Define the tree derivative of $f$ at $x$, denoted as $D^{\mc T}f(x)$, as
    \begin{equation*}
        D^{\mc T} f(x) := D^{k} f(x)[D^{\mc T_1} f(x), \dots, D^{\mc T_k} f(x)]
        \eqfs
    \end{equation*}
\end{definition}
\begin{theorem}\label{thm:solutionderivtree}
    Let $d,k\in\N$. Let $f\in\mc D^{k-1}(\R^d, \R^d)$. Let $x\in\R^d$.
    Assume $f$ is Lipschitz continuous.
    Set $u := U(f, x, \cdot)$.
    Then
    \begin{equation*}
        D^k u(t) = \sum_{\mc T \in\mf T_{\nnset k}} D^{\mc T} f (u(t))
        \eqfs
    \end{equation*}
\end{theorem}
\begin{proof}
    \cite[Theorem 311B]{Butcher2016}.
\end{proof}
\begin{corollary}\label{coro:trajSmooth}
	Let $d,\beta\in\N$. Let $\indset L0{\beta} \subset \Rpp$.
    Let $f\in\bar\Sigma^{d\to d}(\beta, \indset L0\beta)$.
    Let $\beta \in\N$.
    Let $x_0\in\R^d$.
    Set $u := U(f, x_0, \cdot)$.
	Set $L := \sup_{k=0,\dots,\beta} L_k$. Then, for all $k\in\nnset{\beta+1}$,
	\begin{equation}\label{eq:trajSmooth:generic}
		\supNormOf{D^ku} \leq (k-1)! L^k
        \eqfs
	\end{equation}
\end{corollary}
\begin{proof}
    Let $S$ be a finite ordered set. Let $\mc T\in \mf T_S$.
    We want to show by induction that
    \begin{equation}\label{eq:deriv:bound}
        \supNormOf{D^{\mc T} f} \leq L^{\# \mc T}
        \eqfs
    \end{equation}
    If $\# \mc T = 1$, $\supNormOf{D^{\mc T} f} =  \supNormOf{f} \leq L$.
    Now, assume  $\mc T = [\mc T_1, \dots, \mc T_\kappa]$ and $\supNormOf{D^{\mc T_j} f} \leq L^{\# {\mc T_j}}$. Then,
    \begin{align*}
        \supNormOf{D^{\mc T} f}
        &=
        \supNormOf{D^{\kappa} f(\cdot)[D^{\mc T_1} f(\cdot), \dots, D^{\mc T_\kappa} f(\cdot)]}
        \\&\leq
        L \prod_{j\in\nnset \kappa} \supNormOf{D^{\mc T_j} f}
        \\&\leq
        L^{1 + \sum_{j=1}^{\kappa}\# \mc T_j}
        \\&=
        L^{\# \mc T}
        \eqfs
    \end{align*}
    Hence, we have shown \eqref{eq:deriv:bound}.
    Using \cref{thm:solutionderivtree}, \eqref{eq:deriv:bound}, and \cref{lmm:numtrees}, we obtain
    \begin{align*}
        \supNormOf{D^ku}
        &\leq
        \sum_{\mc T \in\mf T_{\nnset k}} \supNormOf{D^{\mc T} f (u(\cdot))}
        \\&\leq
        \# \mf T_{\nnset k} L^k
        \\&\leq
        (k-1)! L^k
        \eqfs
    \end{align*}
\end{proof}
\subsection{Smoothness of Solutions in the Initial Conditions}
We state a simple consequence of the basic form of Grönwall's inequality, which we will make use of below.
\begin{lemma}[Grönwall's inequality]\label{lmm:groenwall}
    Let $d\in\N$. Let $u\colon\R\to \R^d$ be differentiable. Let $L\in\Rp$.
    Assume
    \begin{equation*}
        \euclOf{\dot u(t)} \leq L \euclOf{u(t)}\qquad \text{for }t\in\R
        \eqfs
    \end{equation*}
    Then, for all $t\in\R$, we have
    \begin{align*}
        \euclOf{u(t)} &\leq \euclOf{u(0)} \exp(L\abs{t})
        \eqcm\\
        \euclOf{u(t) - u(0)} &\leq \euclOf{u(0)} \br{\exp(L\abs{t}) -1 }
        \eqfs
    \end{align*}
\end{lemma}
\begin{proof}
    We assume below that $t > 0$ and $u(t) \neq 0$ for all $s\in [0, t]$.
    If $t < 0$, consider $\tilde u(t) = u(-t)$ instead. If there are $s\in [0, t]$ with $u(s) = 0$, the proof applies to all intervals $[a,b]$ with $u(\tilde s)\neq 0$ for $\tilde s \in [a,b]$, which yields an upper bound that is at least as strong as claimed.

    Assume $t > 0$ and $u(t) \neq 0$. Set $v(t) := \euclOf{u(t)}$. Then
    \begin{equation*}
        \dot v(t) = \frac{u(t)\tr \dot u(t)}{ \euclOf{u(t)}} \leq  \euclOf{\dot u(t)} \leq  L \euclOf{u(t)} = L v(t)
        \eqfs
    \end{equation*}
    Thus, by Grönwall's inequality in its standard form, we have
    \begin{equation*}
        v(t) \leq v(0) \exp(Lt)
        \eqcm
    \end{equation*}
    which is the first inequality that we want to prove.
    Now, we can calculate
    \begin{align*}
        \euclOf{u(t) - u(0)}
        &\leq
        \int_0^t \euclOf{\dot u(s)} \dl s
        \\&\leq
        L \int_0^t v(s) \dl s
        \\&\leq
        L \euclOf{u(0)} \int_0^t \exp(L s) \dl s
        \\&=
        \euclOf{u(0)} \br{\exp(L t) -1}
        \eqfs
    \end{align*}
\end{proof}
\begin{lemma}[Derivative with respect to initial conditions]\label{lmm:derivIncrem}
    Let $d\in\N$. Let $L_1\in\Rpp$. Let $f\in\bar\Sigma^{d\to d}(1, \infty, L_1)$.
    Then, for all $t\geq 0$,
	\begin{equation*}
		\abs{D U(f, \cdot, t)}_{\infty} \leq \exp(L_1 t)
		\quad\text{and}\quad
		\abs{D \Incr(f, t, \cdot)}_{\infty} \leq \exp(L_1 t)-1
		\eqfs
	\end{equation*}
\end{lemma}
\begin{proof}
    Let $v \in \mb D_d$ be a direction.
	Set $u(t) := u_{x}(t) := D_{x, v} U(f, x, t)$ denote the derivative of the solution $U(f, x, t)$ with respect to the initial conditions $x$.
	Then, using the chain rule \cref{lmm:derivrules} \ref{lmm:derivrules:chain},
	\begin{align*}
		\dot u(t)
		&=
		D_{x,v} \dot U(f, x, t)
		\\&=
		D_{x,v} \br{f(U(f, x, t))}
		\\&=
		Df (U(f, x, t)) [D_{x,v} U(f, x, t)]
        \\&=
        Df (U(f, x, t)) [u(t)]
		\eqfs
	\end{align*}
	Furthermore,
	\begin{equation*}
		u(0) = D_{x, v} U(f, x, 0) = v
		\eqfs
	\end{equation*}
	We have established the differential inequality
	\begin{equation*}
		\euclOf{\dot u(t)} \leq \euclOf{(Df) (U(f, x, t)) [u(t)]} \leq L_1 \euclOf{u(t)}
	\end{equation*}
	with $\euclof{u(0)} = 1$. Thus, by \cref{lmm:groenwall},
	\begin{equation*}
		\euclOf{u(t)} \leq \exp(L_1 \abs{t})
		\eqfs
	\end{equation*}
    Hence,
    \begin{align*}
        \abs{D U(f, \cdot, t)}_{\infty}
        &=
        \sup_{v \in \mb D_d} \sup_{x\in\R^d}  \euclOf{D_{x, v} U(f, x, t)}
        \\&\leq
        \exp(L_1 t)
        \eqfs
    \end{align*}
	Set $u^\circ(t) := u(t) - u(0)$. Then, $\dot u^\circ = \dot u$ and $u^\circ(0) = 0$. Thus, using \cref{lmm:groenwall} again,
	\begin{equation*}
		\euclOf{u^\circ(t)} \leq \exp(L_1 t) - 1
        \eqfs
	\end{equation*}
	We obtain,
	\begin{equation*}
		\abs{D \Incr(f, t, \cdot)}_{\infty}
		=
		\abs{D U(f, \cdot, t)-D U(f, \cdot, 0)}_{\infty}
		\leq
		\exp(L_1 t) - 1
        \eqfs
	\end{equation*}
\end{proof}
For higher order derivatives of $U(f, x, t)$ with respect to $t$, we used the tree derivatives in \cref{thm:solutionderivtree}. For higher order derivatives of $U(f, x, t)$ with respect to $x$, we introduce partition derivatives.
\begin{definition}
	Let $k\in\N$. 
    \begin{enumerate}[label=(\roman*)]
        \item
            Let $\mb A_k$ be the set of partitions of $\nnset k$.
        \item
            Let $\mo v = (v_1, \dots, v_k)\in\mb D_d^k$. Define the set of partitions of $\mo v$, as
            \begin{equation*}
                \mb A(\mo v) := \setByEle{(v_{A_1},\dots,v_{A_\ell})}{(A_1, \dots, A_\ell) \in \mb A_k}
                \eqfs
            \end{equation*}
         \item
             Let $\mo v\in\mb D_d^k$. Let $f,g \in \mc D^{k}(\R^d, \R^d)$.
             Let $\pi = (\mo v_1, \dots, \mo v_\ell) \in \mb A(\mo v)$.
             Denote the \textit{partition derivative} of $f\circ g$ as
             \begin{equation*}
                 D^\pi (f\circ g) (x) := (D^\ell f)(g(x))[D_{\mo v_1} g(x), \dots, D_{\mo v_\ell} g(x)]
                 \eqfs
             \end{equation*}
    \end{enumerate}
\end{definition}
The following theorem is a version of Faà di Bruno's formula \cite{Craik05Prehistory}.
\begin{theorem}\label{thm:deriv:ic}
    Let $k\in\N$. Let $\mo v\in\mb D_d^k$. Let $f \in \mc D^{k-1}(\R^d, \R^d)$.
    Assume that $f$ is Lipschitz continuous.
    Then
    \begin{equation*}
        D_{x,\mo v} \dot U(f,x,t) = \sum_{\pi \in \mb A(\mo v)} D^\pi \br{f \circ U(f, \cdot, t)}(x)
        \eqfs
    \end{equation*}
\end{theorem}
\begin{proof}
    For $\mo{v} \in \mb D_d^k$, denote $u_{\mo v}(x,t) := D_{x,\mo v} U(f,x,t)$.
    We prove the statement by induction over $k$. See the proof of \cref{lmm:derivIncrem} for $k=1$.
    Assume the statement is true for all $\mo{v} \in \mb D_d^k$.
    Let $v\in \mb D_d$ and $\mo{v\pr} := (\mo v, v) \in \mb D_d^{k+1}$. Then
    \begin{align*}
        \dot u_{\mo{v\pr}}(x,t)
        &=
        D_{x,v} \dot u_{\mo{v}}(x,t)
        \\&=
        \sum_{(\mo{v_1},\dots,\mo{v_\ell})\in\mb A(\mo v)}
        D_{x,v} \br{D^\ell f(U(f,x,t))[u_{\mo{v_1}}(x,t), \dots,u_{\mo{v_\ell}}(x,t)]}
    \end{align*}
    By the product rule, \cref{lmm:derivrules} \ref{lmm:derivrules:product},
    \begin{align*}
        & D_{x,v} \br{(D^\ell f)(U(f,x,t))[u_{\mo{v_1}}(x,t), \dots,u_{\mo{v_\ell}}(x,t)]}
        \\&=
        D_{x,v} \br{ D^\ell f(U(f,x,t))}[u_{\mo{v_1}}(x,t), \dots,u_{\mo{v_\ell}}(x,t)]
        \\&\phantom{=}\ +
        \sum_{j=1}^\ell
        \br{(D^\ell f)(U(f,x,t))[u_{\mo{v_1}}(x,t), \dots, D_{x,v}u_{\mo{v_j}}(x,t), \dots, u_{\mo{v_\ell}}(x,t)]}
        \eqfs
    \end{align*}
    By the chain rule, \cref{lmm:derivrules} \ref{lmm:derivrules:chain},
    \begin{align*}
        D_{x,v} \br{ D^\ell f(U(f,x,t))}[u_{\mo{v_1}}(x,t), \dots,u_{\mo{v_\ell}}(x,t)]
        =
        D^{\ell+1} f(U(f,x,t))[u_{\mo{v_1}}(x,t), \dots,u_{\mo{v_\ell}}(x,t), u_{v}(x,t)]
        \eqfs
    \end{align*}
    The set of all partitions of $\mo v\pr$ can be created from the set of all partitions of $\mo v$ as follows: Let
    $\mo{v_1},\dots,\mo{v_\ell}$ be a partition of $\mo v$. Then $(v,\mo{v_1},\dots,\mo{v_\ell})$ and $(\mo{v_1},\dots,(\mo v_j,v),\dots,\mo{v_\ell})$, $j\in\nnset\ell$ are partitions of $\mo v\pr$ (the order of the partition components does not matter). Iterating this construction over all partitions of $\mo v$ creates all partitions of $\mo v\pr$.

    Now putting all previous equations in this proof together, we obtain
    \begin{align*}
        D_{x,\mo v\pr} \dot U(f,x,t)
        &=
        \dot u_{\mo{v\pr}}(x,t)
        \\&=
        \sum_{(\mo{v_1},\dots,\mo{v_\ell})\in\mb A(\mo v\pr)}
        D^\ell f(U(f,x,t))[u_{\mo{v_1}}(x,t), \dots,u_{\mo{v_\ell}}(x,t)]
        \\&=
        \sum_{\pi \in \mb A(\mo v\pr)} D^\pi \br{f \circ U(f, \cdot, t)}(x)
        \eqfs
    \end{align*}
\end{proof}
\begin{corollary}\label{lmm:highDerivIncrem}
    Let $\beta\in\N$, $\indset L1\beta\in\Rpp$, $k\in\nnset\beta$, $d\in\N$. Let $f\in\bar\Sigma^{d\to d}(\beta, \indset L0\beta)$ with $L_0=\infty$. Let $\stepsize_0\in\Rpp$. Assume $\stepsize\in [0, \stepsize_0]$.
    Then
    \begin{equation*}
        \abs{D^k \Incr(f, \stepsize, \cdot)}_{\infty}
        \leq C_k\br{\exp(L_1 \stepsize) - 1}\eqcm
    \end{equation*}
    where $C_k\in\Rpp$ depends only on $k$, $\stepsize_0$, and $\indset L1k$.
\end{corollary}
\begin{proof}
    Define $C_1 := 1$, $\tilde C_1 := \exp(L_1 \stepsize_0)$,
    and, for $k\in\nset{2}{\beta}$ with an arbitrary $\mo v\in\mb D_d^k$,
    \begin{align*}
        C_k &:= \frac{1}{L_1}\sum_{(\mo v_1, \dots, \mo v_\ell) \in \mb A(\mo v)\setminus\{\mo v\}} L_\ell \prod_{j=1}^\ell \tilde C_{\# \mo v_j}\eqcm\\
        \tilde C_k &:= C_k \br{\exp(L_1 \stepsize_0) - 1}\eqfs
    \end{align*}

    Let $k\in\nnset \beta$. Let $\mo v\in\mb D_d^k$. Set $u_{\mo v}(x, t) = D_{x, \mo v} U(f,t,x)$.
    We will show, by induction over $k$, that
    \begin{equation}\label{eq:highDerivIncrem:induction}
        \sup_{x\in \R^d} \euclOf{u_{\mo v}(x, t)}
        \leq
        \begin{cases}
            \exp(L_1 t) & \text{ if } k=1\eqcm\\
            C_k\br{\exp(L_1 t) - 1}& \text{ otherwise}\eqfs
        \end{cases}
    \end{equation}
    \cref{lmm:derivIncrem} shows the case $k=1$.
    Now let $k\geq 2$ and assume \eqref{eq:highDerivIncrem:induction} for all $\mo v\in\mb D_d^{k-1}$. Let $\mo v\in\mb D_d^{k}$.
    By \cref{thm:deriv:ic},
    \begin{equation*}
        \dot u_{\mo v}(x, t)
        =
        \sum_{(\mo v_1, \dots, \mo v_\ell) \in \mb A(\mo v)} (D^\ell f)(U(f,x, t))[u_{\mo v_1}(x, t), \dots, u_{\mo v_\ell}(x, t)]
        \eqfs
    \end{equation*}
    Thus,
    \begin{align*}
        \euclOf{\dot u_{\mo v}(x, t)}
        &\leq
        \sum_{(\mo v_1, \dots, \mo v_\ell) \in \mb A(\mo v)} \supNormOf{D^\ell f} \prod_{j=1}^\ell \euclOf{u_{\mo v_j}(x, t)}
        \\&\leq
        L_1  \euclOf{u_{\mo v}(x, t)} + \sum_{(\mo v_1, \dots, \mo v_\ell) \in \mb A(\mo v)\setminus\{\mo v\}} L_\ell \prod_{j=1}^\ell \euclOf{u_{\mo v_j}(x, t)}
        \eqfs
    \end{align*}
     Set $z(t) := \euclOf{u_{\mo v}(x, t)}$. Then
    \begin{equation*}
        \dot z(t)
        =
        \frac{\dot u_{\mo v}(x, t)\tr u_{\mo v}(x, t)}{\euclOf{u_{\mo v}(x, t)}}
        \leq
        \euclOf{\dot u_{\mo v}(x, t)}
        \leq
        L_1 \br{z(t) + C_k}
        \eqfs
    \end{equation*}
    Consider the ODE
    \begin{equation*}
        \dot v(t) = L_1  \br{v(t) + C_k}
    \end{equation*}
    with $v(0) = \euclOf{u_{\mo v}(x, 0)} = 0$ for $k\geq 2$.
    Then
    \begin{equation*}
        v(t) = C_k \br{\exp(L_1 t) - 1}\eqfs
    \end{equation*}
    By the \cref{lmm:diffinequcomp},
    \begin{equation*}
        z(t) \leq v(t)
        \eqfs
    \end{equation*}
    Hence
    \begin{equation*}
        \euclOf{u_{\mo v}(x, t)} \leq C_k \br{\exp(L_1 t) - 1}
        \eqfs
    \end{equation*}
    For $k\geq2$, $D^k \Incr(f, t, x) = D^k U(f, x, t)$. Thus,
    \begin{equation*}
        \sup_{x\in \R^d} \normop{D^k \Incr(f, t, x)} \asymleq \sup_{x\in \R^d}  \sup_{\mo v\in\mb D_d^k}\euclof{u_{\mo v}(x, t)} \leq C_k \br{\exp(L_1 t) - 1}
        \eqfs
    \end{equation*}
    For $k=1$, the result follows from \cref{lmm:derivIncrem} with $C_1 = 1$.
\end{proof}
\end{appendix}
\printbibliography
\end{document}